\documentclass[11pt,reqno]{amsart}

\usepackage{amsfonts,amsmath,amssymb,amsthm}
\usepackage{bbm}
\usepackage{hyperref}
\usepackage[english]{babel}
\usepackage{parskip}
\usepackage{mathrsfs}
\usepackage{graphicx}
\graphicspath{ {images/} }
\usepackage{bm}
\usepackage{physics}
\usepackage{xcolor}
\usepackage{mathtools}
\usepackage{esvect}
\usepackage{polski}
\usepackage{subcaption}
\usepackage[format= plain, labelformat=default,labelsep=colon,skip=3pt]{caption}
\usepackage{graphicx}
\usepackage{verbatim}
\usepackage{placeins}

\newtheorem{theorem}{Theorem}[section]

\newtheorem{lemma}[theorem]{Lemma}
\newtheorem{proposition}[theorem]{Proposition}
\newtheorem{remark}[theorem]{Remark}
\newtheorem{definition}[theorem]{Definition}

\numberwithin{equation}{section}

\usepackage[a4paper,margin=2.5cm]{geometry}
\newcommand{\mres}{\mathbin{\vrule height 1.6ex depth 0pt width
		0.13ex\vrule height 0.13ex depth 0pt width 1.3ex}}

\newcommand{\Om}{\Omega}
\newcommand{\ep}{\varepsilon}

\newcommand{\lbd}{\lambda}
\newcommand{\R}{\mathbb{R}}
\newcommand{\N}{\mathbb{N}}

\newcommand{\Hs}{\mathcal{H}}
\newcommand{\B}{\mathcal{B}}
\newcommand{\C}{\mathcal{C}}
\newcommand{\K}{\mathcal{K}}

\newcommand{\xp}{\left( x \right)}

\newcommand{\sdist}{{\rm{sdist}}}
\newcommand{\topC}{\tau_\mathcal{C}}
\newcommand{\cngK}{\xrightarrow{\mathcal{K}}}
\newcommand{\Int}{{\rm{Int}}}

\newcommand{\EEE}{\color{black}}

\newcommand{\rk}{ \right\}}
\newcommand{\lk}{ \left\{}

\newcommand{\sym}{{\rm sym}}
\newcommand{\dist}{{\rm dist}}

\newcommand{\Ao}{A^{(1)}}
\newcommand{\Sz}{S^{(0)}_{h,K}}

\newcommand{\Ak}{{A_k}}
\newcommand{\Sk}{{S_k}}
\newcommand{\Ako}{A_k^{(1)}}

\newcommand{\Skz}{S_k^{(0)}}
\newcommand{\brbak}{\partial^* \sigma_{\rho_n} (A_\kn)}

\newcommand{\bbak}{\partial \sigma_{\rho_n} (A_\kn)}
\newcommand{\bbsk}{\partial \sigma_{\rho_n} (S_\kn)}
\newcommand{\boda}{ (\sigma_{\rho_n} (A_\kn))^{(1)}}
\newcommand{\bzda}{ (\sigma_{\rho_n} (A_\kn))^{(0)}}

\newcommand{\nubak}{ \nu_{\sigma_{\rho_n} (A_\kn)}}

\newcommand{\Uco}{\overline{Q_1}}
\newcommand{\kn}{{k_n}}

\newcommand{\AH}{ {\rm AH}(\Om)}
\newcommand{\AS}{ {\rm AS}(\Om)}

\newcommand{\wt}{\widetilde}
\newcommand{\wh}{\widehat}
\newcommand{\ov}{\overline}

\newcommand\restr[2]{{
  \left.\kern-\nulldelimiterspace 
  #1 
  \littletaller 
  \right|_{#2} 
  }}

\renewcommand{\d}{\mathrm{d}}

\newcommand{\res}{\mathbin{\vrule height 1.6ex depth 0pt width 0.13ex\vrule height 0.13ex depth 0pt width 1.3ex}}

\begin{document}
	\sloppy
	\title[Two-phase free-boundary problem]{Existence of minimizers for a two-phase free boundary problem  with coherent and incoherent interfaces 
 }

	\author[R. Llerena]{Randy Llerena}
	\address[Randy Llerena]{Research Platform MMM ``Mathematics-Magnetism-Materials" - Fak. Mathematik Univ. Wien, A1090 Vienna}
	\email{randy.llerena@univie.ac.at}
	
	\author[P. Piovano]{Paolo Piovano}
	\address[Paolo Piovano]{Dipartimento di Matematica, Politecnico di Milano, P.zza Leonardo da Vinci 32, 20133 Milano, Italy\footnote{MUR Excellence Department 2023-2027}
 }
	\email{paolo.piovano@polimi.it}

	\date{\today} 
	\begin{abstract}
 A variational model for describing the morphology of two-phase  continua by allowing for the interplay between coherent and
incoherent interfaces is introduced.  Coherent interfaces are characterized by the microscopical arrangement of atoms of the
two materials in a homogeneous lattice, with deformation being the solely stress relief mechanism, while at incoherent interfaces delamination between the two materials occurs. The model is  designed in the framework of the theory of Stress Driven Rearrangement Instabilities, which are  characterized by the competition between elastic and surface effects. The existence of energy minimizers is established in the plane by means of the direct method of the calculus of variations under a constraint on the  number of boundary connected components of the underlying phase, whose exterior boundary is prescribed to satify a  graph assumption, and of the two-phase composite region. Both the wetting and the dewetting regimes are  included in the analysis. 

 \EEE

	\end{abstract}
	

	\maketitle
	
	\tableofcontents
	
	\section{Introduction}

  In this manuscript we address the problem of providing a mathematical variational framework for the description of the morphology and the elastic properties of two-phase continua based on Gibbs's notion of a sharp phase-interface dividing them \cite{cermelligurtin1,FG,Gurtin1}. In the presence of two interacting media large stresses due to the  different crystalline order of the two materials  originate and, besides bulk deformation, various types of  morphological destabilization may occur as a further strain relief mode. These are often referred to as the family 
of \emph{stress driven rearrangement instabilities} (SDRI) \cite{AT,D,Gr,KP,S}, which include  the roughness of the exposed crystalline boundaries, the formation of  cracks in the bulk materials, the nucleation of dislocations in the crystalline lattices, and the delamination (as opposed to the adhesion) at the contact regions with the other material. 

Literature provides extensive studies of these phenomena under the assumption that one phase is a rigid  fixed continuous medium  underlying, such as the substrates for epitaxially-strained thin films \cite{ChB,CF,FFLM2,KreutzP}, or constraining,   such as crystal cavities \cite{FFLM} or the containers in capillarity problems \cite{DphM}, the other phase, which is instead let free, or by modeling the interactions with other media simply by means of fixed boundary conditions. There are though  settings in which the hierarchy between the phases is not clear, or  a rigidity ranking between them is not easily identifiable, since the interplay among the deformation and the interface instabilities affecting all phases is crucial, such as,  in the shock-induced transformations and mechanical twinning \cite{Gurtin1} or in the deposition of film multilayers \cite{LlP1}.

As described in \cite{Gurtin1} the extension of classical theories of continuum mechanics to two-phase deformable media is though  ``not as straightforward as it might appear'', since combining   the accretion and deletion of material constituents responsible for the moving of the interface between the two phases and their boundaries, with the framework of elasticity related to bulk deformation and fractures \cite{CC,FrMa}, by quoting \cite{Gurtin1}, ``leads to conceptual difficulties''. A critical modeling issue related to the   interface  between the two phases is  the interplay between \emph{coherency}, that is here intended  as the microscopical arrangement of atoms of the two materials in a homogeneous lattice, with deformation being the solely stress relief mechanism \cite{Gurtin1}, and \emph{incoherency}, that instead refers  to the debonding occurring between the atoms of the two materials \cite{cermelligurtin1}, which results in the composite delamination at the two-phase interface \cite{Baldelli:2014}. 
This manuscript seems to be, to the best of our knowledge, the first attempt to provide a  mathematical framework able to simultaneously describe coherent and incoherent interfaces for a two-phase setting, which we carry out by also keeping in the picture the other features of SDRI, such as the dichotomy between the \emph{wetting regime},
that is the setting in which it is more
convenient for a phase to cover the  surface of the other phase with an infinitesimal layer of atoms,
and the \emph{dewetting regime}, in which it is preferable to let such surface exposed to the vapor.  






The studies for the setting of  only coherent interfaces go  back to Almgren \cite{Al},   who was the first to formulate the problem in $\mathbb{R}^d$, $d>1$, in the context without elasticity for surface tensions proportional among the various  interfaces, by means of  \emph{integral currents} in geometric measure theory and by singling out a condition 
ensuring the lower semicontinuity of the overall surface energy with respect to the $L^1$-convergence of the sets in the partition. Then, Ambrosio and Braides in \cite{AB,AB2} extended the setting to also non-proportional surface tensions by introducing a new integral condition referred to as $BV$-\emph{ellipticity}, which they show to be both sufficient and necessary for the lower semicontinuity with respect to the $L^1$-convergence. As such condition is the analogous, for the setting of \emph{Caccioppoli partitions}, of Morrey’s \emph{quasi-convexity}, it is difficult to check it in practice. In \cite{AB,AB2}  $BV$-ellipticity is proved to coincide with a simpler to check  
\emph{triangle inequality condition} among surface tensions for the case of partitions in 3 sets  (like the setting considered in this paper,  in which one element of the partition is always represented by the \emph{vapor} outside the two phases), which was then confirmed to be the only case by \cite{cara1}. Other conditions therefore have been introduced, such as $B$-\emph{convexity} and \emph{joint convexity}, 
with though $BV$-ellipticity remaining so far the only condition known to be both necessary and sufficient for lower semicontinuity apart from specific settings (see \cite{cara2, Morgan} for more details). Recently, the $BV$-ellipticity has been extended in the context of $BD$-\emph{spaces} in \cite{friesolo}. Finally, we refer to  \cite{BonCris} for a variant of the \emph{Ohta–Kawasaki model} considered to model thin films of diblock copolymers in the unconfined case, which represents a recent example in the literature of a two-phase model in the absence of elasticity and of incoherent and crack interfaces, under a graph constraint for union of the two phases. 


Regarding incoherent interfaces the problem is intrinsically related to the renowned \emph{segmentation problem} in image reconstruction that was actually originally introduced  by Mumford and Shah in  \cite{MS} with a multiphase formulation, as a partition problem of an original image, with the connceted contours of the image areas characterized as discontinuity set of an auxiliary state function. 
Then, the approaches developed to  tackle the problem led to the study of a single phase setting with the jump set of the state function representing internal interfaces, proven to satisfy \emph{Ahlfors-type regularity} result  \cite{AFP,DMS}. Such single-phase framework has been then extended to the context of linear elasticity in fracture 
mechanics  with the state function being vectorial and representing the bulk displacement of a crystalline material and the energy replaced by the \emph{Griffith energy} \cite{CC,FrMa}. The attempt to recover the original setting of  \cite{MS} in a rigorous mathematical formulation (apart from some formulations with piecewise-constant state functions  or numerical investigations) has been then addressed by Bucur, Fragal\`a, and Giacomini in \cite{BFG} and \cite{BFG2} (see also \cite{CTV} for a related multiphase boundary problem in the context of reaction-diffusion systems). In \cite{BFG} Ahlfors-type regularity  is established for  \emph{ad hoc} notions of multiphase local \emph{almost-quasi minimizers} of an energy accounting for incoherent isotropic interfaces and disregarding the contribution of the coherent portions, while in \cite{BFG2} they introduce a  multiphase version of the {M}umford-{S}hah problem by treating all the reduced phase boundaries as incoherent interfaces (like the internal jump sets) and by adding an extra (statistical) term, which induces multi-phase minimizers.  


In order to finally include  in the model  both coherent and incoherent  portions (possibly also  on the same interface between the phases), we first restrict to the two-phase setting (for a multi-phase setting in the context of film multilayers we refer to the related paper \cite{LlP1} under finalization) and 
follow a different direction than the one of \cite{BFG,BFG2}, which works for $d=2$: We adopt the approach considered in \cite{KP,KP1}, that was relying on the strategy developed for the Mumford-Shah problem in   \cite{DMS}. Such approach consists in first imposing a fixed constraint on the number of connected components for the boundary of the free phases, in order then to employ adaptations of Golab’s Theorem \cite{G} for proving the compactness with respect to a proper selected topology, and then in studying the convergence of the  solutions of the different minimum problems related to different  constraints on the connected components, as such constraints  tend to infinity. This second step has been performed for the one-phase setting in  \cite{KP1} (and for higher dimension in \cite{KP2}) by means of density estimates.

Here, we performed the first step in this program, reaching an existence result analogous to the one in \cite{KP}. 
However,  the extension of \cite{KP} to the two-phase setting requires important modification in the model setting, since  characterizing the incoherent interface as the jump portion of the bulk displacement on the two-phase interface as in \cite{KP,KP1,KP2} appears to be not feasible, as in  our setting the two-phase interfaces need to be considered much less regular
than Lipschitz manifolds like in \cite{KP,KP1,KP2}. To solve this issue we  consider as set variables of the energy not the two phases, to which we refer to as the \emph{film} and the \emph{substrate phase}, but the substrate phase and the whole region occupied by the composite of both the two phases, and we characterize the incoherent interfaces as a portion of the boundary intersection of such variables. 
As a byproduct of this strategy we do not need to impose a constraint on the number of boundary components of the film phases (but only with respect to the substrates and the composite regions), so that the physical relevant setting of countable separated isolated film islands forming on top of the substrate is included in our analysis, even though it was  prevented by the formulation in \cite{KP}. Moreover, we can  also extend \cite{KP} to the presence of adjacent materials for the Griffith-type model with mismatch strain and delamination \cite{CC,FrMa,KP2}. 

In agreement with the SDRI theory \cite{AT,D,Gr} the total energy $\mathcal{F}$ is given by the sum of two contributions, namely the elastic energy $\mathcal{W}$ and  the surface energy $\mathcal{S}$, and it is defined on triples $(A,S,u)\in\wt{\mathcal{C}}$, where $u$ represent the \emph{bulk displacement} of the composite material of the two phases, and $A$ and $S$ are sets whose closures represents the \emph{composite region}  and the \emph{substrate region}, respectively, while the \emph{film region} is given by $\overline{A}\setminus S^{(1)}$   (for $S^{(1)}$ denoting the points with  density 1 in $S$). More precisely, given $\Om:= (-l,l)\times_{\R^2}(-L,L) \subset \R^2$ as the region where the composite material is located, which is defined for the two parameters $l,L>0$ and to which we refer as the \emph{container} in analogy to the notation of capillarity problems, we introduce
 \begin{align*}
	\wt{\mathcal{C}}:=\Large\{ (A,S,u):   \quad& \text{$A$ and $S$ are  $\mathcal{L}^2$-measurable sets with $S\subset  \overline{A}    \subset \overline \Om$ such that}\\   &\text{$\partial A \cap \Int(S) = \emptyset$, $\partial A$ and $\partial S$  are $\Hs^1$-rectifiable, } \\ &\text{$\Hs^1 (\partial A)+\Hs^1 (\partial S) < \infty$, and 
  $u \in  H^{1}_\mathrm{{loc}} ( \Int (A); \R^2)$}  \}. 
\end{align*}
We  define $\mathcal{F} :  \wt{\mathcal{C}}
	\rightarrow\mathbb{R}$ as
$$
		\mathcal{F}(A,S
		,u) := \mathcal{S}(A,S
		) + \mathcal{W}(A,u)
 $$
	for every $(A,S
	,u)\in\wt{\mathcal{C}}$. 
The elastic energy $\mathcal{W}(A,u)$  
is defined  
  analogously to \cite{DP, KP, KP1, KP2} by
	\begin{equation*}
		\mathcal{W}(A,u):= \int_{A}{W \left( x, E \left( u \xp - E_0 \xp \right) \right)dx},
	\end{equation*}
	where the elastic density $W$ is determined by the quadratic form
	\begin{equation*}
		W \left( x, M \right) := \mathbb{C}\xp M : M,
	\end{equation*}
	for a fourth-order tensor $\mathbb{C}: \Om \to \mathbb{M}^2_\sym$, $E$ denotes the symmetric gradient, i.e., $E(v) := \frac{\nabla v + \nabla^T v}{2}$ for any $v \in H^1_\mathrm{{loc}}(\Int (A); \R^2)$, representing the \emph{strain}, and $E_0$ is
	the \emph{mismatch strain} $x \in \Om \mapsto E_0 \xp \in \mathbb{M}^{2}_\mathrm{{sym}} $ defined as
	\begin{equation*}
		E_0 := \begin{cases}
			E(u_0) & \text{in } \Om \setminus S, \\ 0 & \text{in } S,
		\end{cases} 
	\end{equation*}
	for a fixed $u_0 \in H^1(\Om; \R^2)$. The mismatch strain is included in the SDRI theory to represent the fact that the two phases are given by possibly different crystalline materials whose free-standing equilibrium lattice could present a lattice mismatch. In this context notice that $\mathbb{C}$ is allowed to present discontinuities at the interface between the two materials   (see hypothesis (H3) in Section \ref{sec:mainresults}). 
	
The surface energy $\mathcal{S}(A,S)$ is given 
by
$$
\mathcal{S} (A,S):= \int_{\Om\cap(\partial A \cup\partial S)} {\psi (x, \nu(x)) \, d\Hs^{1}(x)},$$
where, by denoting with $\nu_U(z)$ the normal unit vector  pointing outward to a set  $U\subset\mathbb{R}^2$ with $\Hs^1$-rectifiable boundary at a point $x\in\partial U$,
$$\nu(x):=\begin{cases}
\nu_A(x) &\text{if  $z\in\partial A\setminus \partial S$},\\
\nu_S(x) &\text{if  $z\in\partial S$,}
    \end{cases}
$$
and $\psi: \overline{\Om} \times \R^2\to [0,\infty]$ represents the \emph{surface tension of the composite} of the two phases, which we allow to be anisotropic.  
 
 In order to properly define $\psi$ we need to consider the three surface tensions $\varphi_{\mathrm{F}},\, \varphi_{\mathrm{S}},\,  \varphi_{\mathrm{ FS}}: \overline{\Om} \times \R^2\to [0,\infty]$ characterizing the three possible interfaces for the two-phase setting, i.e., the interface between the film phase and the vapor, the interface between the substrate phase and the vapor, and the interface between the film and the substrate phases. 
 Furthermore, to simultaneously treat both the wetting and the dewetting regime, we introduce two auxiliary surface tensions, to which we refer as the \emph{regime surface tensions}, that we defined as:
$$
\varphi := \min\{\varphi_{\mathrm{S}}, \varphi_{\mathrm{F}} + \varphi_{\mathrm{FS}} \}\qquad\text{and}\qquad\varphi' : = \min\{\varphi_{\mathrm{S}},\varphi_{\mathrm{F}}\},
$$
since $\varphi_{\mathrm{S}}$ is the surface tension associated to the dewetting regime, as the substrate surface remains exposed to the vapor, while $\varphi_{\mathrm{F}} + \varphi_{\mathrm{FS}}$ and $\varphi_{\mathrm{F}}$ are both  associated to the wetting regime, respectively, to the situation of  an infinitesimal layer of film atoms covering the substrate surface (by being bonded to the substrate atoms), which is referred to as the \emph{wetting layer}, or of simply a detached  \emph{film filament}.
We define 
\begin{equation}
		\label{psi}
		\psi(x, \nu(x)) :=  \begin{cases} 
  \varphi_{\mathrm{ F}}(x, \nu \xp) & x \in \Om \cap (\partial^* A \setminus \partial^*{S}),  \\
  \varphi (x, \nu \xp) & x \in \Om \cap   \partial ^*{S} \cap \partial^* A, \\ 
  \varphi_{\mathrm{ FS}} (x, \nu \xp ) & x \in\Om \cap (\partial ^*{S} \setminus \partial A),\\
  (\varphi_{\mathrm{ F}} + \varphi )(x, \nu \xp ) & x \in\Om \cap  \partial ^*{S} \cap \partial A  \cap A^{(1)},\\
   2\varphi_{\mathrm{ F}}(x, \nu \xp )
   & x \in \Om \cap \partial A 
   \cap  A^{(1)} \cap  S^{(0)},  \\
   2 \varphi' (x, \nu \xp )
   & x \in  \Om \cap \partial A 
   \cap A^{(0)},  \\
   2\varphi_{\mathrm{ FS}} (x, \nu \xp)
    & x \in  \Om \cap (\partial {S} \setminus \partial A) \cap \left(S^{(1)} \cup S^{(0)}\right)  \cap A^{(1)},  \\
     2\varphi (x, \nu \xp )
     & x \in  \Om \cap \partial {S} \cap \partial A \cap S^{(1)},  
 \end{cases} 
	\end{equation}
where $\partial^*U$ and $U^{(\alpha)}$ denote, when well defined, the \emph{reduced boundary} and the set of points of density $\alpha\in[0,1]$ for a set $U\subset\mathcal{R}^2$. We notice that the 8 subregions of the domain $\Om\cap(\partial A \cup\partial S)$ in which the definition of $\psi$ is distinguished are the counterpart of the 5 terms appearing in the surface energy of 
\cite{KP} for the two-phase setting (see  Remark \ref{rem:previous_surface_tensions} for more details). Each subregion appearing in \eqref{psi} represents, by moving line by line, the \emph{film free boundary}, the \emph{substrate free boundary}, the \emph{film-substrate coherent interface}, the \emph{film-substrate incoherent interface}, the \emph{film cracks}, the \emph{exposed filaments}, 
	 the \emph{substrate filaments and cracks} in the film-substrate coherent interface, and  the \emph{substrate cracks} in the film-substrate incoherent interface, respectively.
  
  We observe that the surface tensions associated to the  film free boundary, the coherent substrate-film interface, and the substrate free boundary, are simply $\varphi_{\mathrm{F}}$,  $\varphi_{\mathrm{FS}}$, and (to accommodate both wetting and dewetting regimes) $\varphi$ respectively, while the surface tension associated to the incoherent film-substrate interface is chosen to be $\varphi_{\mathrm{F}}+\varphi$ in analogy with the film-substrate delamination or delaminated region  in \cite{KP,KP1, KP2}, since the incoherent interface coincides with the portion of the film-substrate interface in which there is no bonding between the film and the substrate surfaces. All remaining 4 terms are weighted double (in analogy to the lower-semicontinuity results previously obtained in \cite{ChB,DP,FFLM2,KP, KP1, KP2} for the one-phase setting) as they refer to either material filaments in the void  or cracks in the composite  bulk.   In particular, we notice that  in the substrate bulk region represented by  $S^{(1)}$ we distinguish between substrate cracks in the coherent and in the incoherent film-substrate interface, that are counted with weight $2\varphi_{\mathrm{FS}}$ and $2\varphi$, respectively, while in the film bulk region $A^{(1)}\cap S^{(0)}$ we distinguish between substrate filaments that are not film cracks counted with  $2\varphi_{\mathrm{FS}}$ and film cracks counted  $2\varphi_{\mathrm{F}}$ (see Figure \ref{figure1}). 

\begin{figure}[ht]
  \centering
\includegraphics{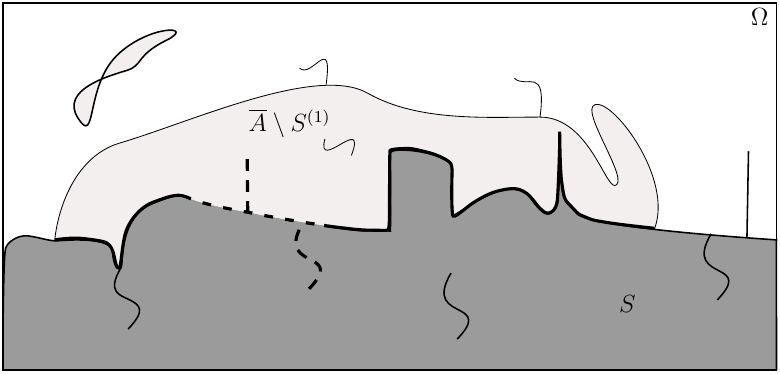}
\caption{ \small The admissible regions for an admissible configuration $(A,S,u)\in\C_\textbf{m}$   (see Definition \ref{def:admissibleconf}) are represented by indicating the substrate region $\overline{S}$ and the film region $\overline{A}\setminus S^{(1)}$ 
with a darker and a lighter gray, respectively. In particular, the film and the substrate free boundaries (with the film and substrate filaments) are indicated with a thinner line, while the film-substrate interface   is depicted with a thicker line that is continuous or dashed to distinguish between its incoherent portions and its coherent portions (inclusive of substrate cracks and filaments that are not film cracks), respectively.  
}
			\label{figure1}
		\end{figure}

The main result of the paper consists in finding a physically relevant family of admissible configurations in $\wt{\mathcal{C}}$, which is denoted by $\mathcal{C}_\textbf{m}$ for $\textbf{m}:=(m_0,m_1)\in\mathbb{N}^2$, 
in which we can prove that, under a \emph{two-phase volume constraint}, $\mathcal{F}$ admits a minimizer. We find such a family $\mathcal{C}_\textbf{m}\subset \wt{\mathcal{C}}$  by considering as admissible configurations $(A,S,u)\in\wt{\mathcal{C}}$ the ones for which (see  Definition \ref{def:admissibleconf} for more details):
\begin{itemize}
    \item[-]  the number of boundary connected components of  $S$ and  $A$ are fixed to be at most $m_0$ and $m_1$, respectively,
  \item[-]  the  substrate regions $S$ satisfy an \emph{exterior graph constraint} consisting in requiring  that $\partial^*S\cup (\partial S\cap S^{(0)})$ 
 is the graph of an upper semicontinuous function with pointwise bounded variation (while  internal, also non-graph-like, substrate cracks are allowed),
\end{itemize} 
as shown in Figure \ref{figure1}. We notice that such an exterior graph constraint allows to have  a more involved description of the substrate regions than the previously considered  graph constraint in the literature for the one-phase setting \cite{ChB,CF,DP,FFLM2},  which is indeed needed to achieve the compactness result contained in Theorem \ref{thm:compactness}.

 Therefore, for any  two volume parameters $\mathbbm{v}_0, \mathbbm{v}_1 \in [\mathcal{L}^2(\Om)/2, \mathcal{L}^2(\Om)]$  such that $\mathbbm{v}_0\leq \mathbbm{v}_1$, we consider the problem:
\begin{equation}\label{eq:introduction1}
  \inf_{
  \begin{subarray}{l}
  \qquad(A,S,u)\in\mathcal{C}_{\textbf{m}}
  \\ 
  \mathcal{L}^2(S) = \mathbbm{v}_0,\, \mathcal{L}^2({A}) = \mathbbm{v}_1
  \end{subarray}
  }{\mathcal{F}(A,S,u)}, 
\end{equation}
which we tackle by employing the \emph{direct method} of the calculus of variations,  namely by equipping $\wt{\mathcal{C}}$ with a  properly chosen topology $\tau_{\mathcal{C}}$ sufficiently weak to establish a compactness property for  energy-equibounded sequences in $\mathcal{C}_{\textbf{m}}$ and strong enough to prove the  lower semicontinuity of  $\mathcal{F}$ in $\mathcal{C}_{\textbf{m}}$.
The topology $\tau_{\mathcal{C}}$ is characterized by the convergence: 
$$
(A_k,S_k,u_k) 
 \xrightarrow[k\to\infty]{\topC}
(A,S,u) \, \Longleftrightarrow\,
\begin{cases}
\text{$\sup_{k\in\mathbb{N}}{ \Hs^1 \left( \partial A_k \right)} < \infty$, $\sup_{k\in\mathbb{N}}{ \Hs^1 \left( \partial S_k \right)} < \infty$,}\\
\text{$\mathrm{ sdist} \left( \cdot, \partial A_k \right) \xrightarrow[k\to\infty]{} \mathrm{ sdist} \left( \cdot, \partial A \right)$ locally uniformly in $\R^2$},\\
\text{$\mathrm{ sdist} \left( \cdot, \partial S_k \right) \xrightarrow[k\to\infty]{} \mathrm{ sdist} \left( \cdot, \partial S \right)$ locally uniformly in $\R^2$,}\\
u_k \xrightarrow[k\to\infty]{} u \quad\text{a.e.\ in $\mathrm{{Int} \left(A\right)}$,}
\end{cases}
$$
 where the \emph{signed distance function} is defined for any $E \subset \R^2$ as follows
	$$
	    \sdist(x, \partial E) := \begin{cases}
	        \dist(x, E) & \text{if } x \in \R^2 \setminus E, \\
	        -\dist(x, E) & \text{if } x \in  E.
	    \end{cases}
	$$

The compactness property shared by energy-equibounded sequences $(A_k,S_k,u_k)\in\C_\mathbf{m}$ that we establish in Theorem \ref{thm:compactness}  consists in the existence, up to a subsequence, of a possibly different sequence  $(\wt{A}_k,\wt{S_k},\wt{u_k})\in\C_\mathbf{m}$ compact in $\C_\mathbf{m}$  with respect to $\tau_\C$ such that 
$$
\liminf_{n \rightarrow \infty}{\mathcal{F}(A_k,S_k,u_k) } = \liminf_{n \rightarrow \infty}{\mathcal{F}\left(\wt{A}_k,\wt{S_k},\wt{u_k}\right)}.
$$
 This is achieved by   both extending to the two-phase setting and to the situation with the exterior graph constraint the strategy used in \cite[Theorem 2.7]{KP}. For the latter, we rely on the arguments already used in \cite{ChB,FFLM2}, while for the former we used the  Blaschke-type selection principle proved in \cite[Proposition 3.1]{KP} together with the Golab's Theorem \cite[Theorem 2.1]{G} and we implement to the two-phase setting the construction of  \cite[Proposition 3.6]{KP}. Such construction is needed to take care of  those connected components of $A_k$ that separate in the limit in multiple connected components, e.g., in the case of
neckpinches, in order to properly apply \emph{Korn's inequality} just after having introduced extra boundary to create different components also at the level  $A_k$ (by passing to the sequence with composite regions $\wt{A}_k$). We notice though that the characterization of the delamination region introduced in this manuscript  allows for a simplification in the arguments used \cite[Theorem 2.7]{KP}, as the surface energy does not involve the bulk displacements, also yielding an extension of the result by  including the situation with $S\neq\emptyset$ and $\mathbbm{v}_1:=\mathcal{L}^2(\Om)$.

The crucial point in proving the $\tau_\C$-lower semicontinuity of $\mathcal{F}$ is the $\tau_\C$-lower semicontinuity of $\mathcal{S}$, as 
the $\tau_\C$-lower semicontinuity of $\mathcal{W}$ directly follows by convexity similarly to \cite{FFLM2,KP}. In order to establish the $\tau_\C$-lower semicontinuity of $\mathcal{S}$ in Proposition \ref{thm:lowersemicontinuityS} we fix $(A_k,S_k,u_k)\in\C_\mathbf{m}$ and $(A,S,u)\in\C_\mathbf{m}$ such that 
$(A_k,S_k,u_k)\xrightarrow{\tau_\C}(A,S,u)$, and we associate the positive Radon measures  $\mu_k$  and  $\mu$  in $\R^2$  to the localized energy versions of $\mathcal{S}(A_k,S_k,u_k)$ and $\mathcal{S}(A,S,u)$, respectively. We have that 
\begin{equation}\label{lsc_F_saa}
\liminf\limits_{k\to+\infty} \mathcal{S}(A_k,S_k,u_k) \ge \mathcal{S}(A,S,u)\qquad\Longleftrightarrow\qquad \liminf\limits_{k\to+\infty} \mu_k(\R^2) \ge \mu(\R^2),
\end{equation}
and since, up to a subsequence, $\mu_k$ weakly* converges to some positive Radon measure $\mu_0$, and  $\mu$ is absolutely continuous with respect to $\mathcal{H}^1\res((\partial A\cup\partial S)\cap\Omega)$, by  proving   the  following   estimate involving  \emph{Radon-Nikodym derivatives}:
\begin{equation}\label{RN_estimates}
\frac{\d\mu_0}{\d\mathcal{H}^1\res(\Om\cap(\partial A\cup\partial S))} \ge \frac{\d\mu}{\d\mathcal{H}^1\res(\Om\cap(\partial A\cup\partial S))} \quad  \text{$\mathcal{H}^1$ -a.e. on $\Om\cap(\partial A\cup\partial S),$}
\end{equation}
which implies that
$
\lim \mu_k(\R^2) = \mu_0(\R^2) \ge \mu(\R^2),
$
in view of \eqref{lsc_F_saa}, the  $\tau_\C$-lower semicontinuity of $\mathcal{S}$ follows. 

The proof of \eqref{RN_estimates} is very involved and it is performed  by separating $\Om\cap(\partial A\cup\partial S)$ in 12 portions on  which we apply a \emph{blow-up technique} (see, e.g., \cite{FM} 
) together with \emph{ad hoc} (apart from the 2 portions in which it turns out that we can use \cite[Theorem 20.1]{M}) results, i.e., Lemmas
\ref{lema:creationOneBdr}-\ref{lema:creationcrackdelamination},  which can be seen as the counterpart in the two-phase setting of \cite[Lemmas 4.4 and 4.5]{KP} (see Table \ref{tab:my_label} for more details on  the 12 blow-ups). 
In order to prove  Lemmas \ref{lema:creationOneBdr}-\ref{lema:creationcrackdelamination},  firstly we formalize the notions of \emph{film islands}, \emph{composite voids}, and \emph{substrate grains} (see Definition \ref{def:notationislands2}), secondly we prove in Lemma \ref{lem:finiteness} that the  coherent interface associated to  any configuration $(A,S, u
) \in  \C_\textbf{m}$ can be regarded, up to an error and a modification of $(A,S, u
)$ by passing to the family $\wt \C$, as given by a finite number (depending on the initial configuration $(A,S, u
)$) of connected  components, and finally we design induction arguments (with respect to the number of such components) in which  we are able to use the induction hypothesis by  ``shrinking" islands,  ``filling"  voids,  and ``modifying grains in new voids'' as depicted in Figures \ref{figure:oneisland}, \ref{figure:replacementofvoids}, and \ref{figure:replacementofgrains}, respectively, by means of employing the \emph{anisotropic minimality of segments}  \cite[Remark 20.3]{M}.

 \EEE

	The manuscript is organized as follows: in Section \ref{sec:notation} we specify the notation used throughout this paper, in Section \ref{sec:model} we introduce the  model under consideration, present some preliminary results  and state the main results, i.e., the existence of a solution to \eqref{eq:introduction1} in Theorem \ref{thm:existence} together with the compactness result of Theorem \ref{thm:compactness} and the lower semicontinuity result of Theorem \ref{thm:lowersemicontinuity}, in Section \ref{sec:compactness}, we prove  Theorem  \ref{thm:compactness}, in Section \ref{sec:semicontinuity} we prove Theorem \ref{thm:lowersemicontinuity}, and finally in Section \ref{sec:existence} we prove Theorem \ref{thm:existence}. 

\EEE

	
	\section{Notation}\label{sec:notation}
	 In this section we collect the relevant notation used throughout the paper by separating it for different mathematical areas. 
	
	\subsection*{Linear algebra}
 We consider the orthonormal basis $\left\{ \mathbf{e}_\mathbf{1}, \mathbf{e}_\mathbf{2} \right\} =  \left\{ (1,0), (0,1) \right\}$ in $\R^2$ and indicate the coordinates of points $x$ in $\R^2$ by $(x_1,x_2)$. We indicate by $a \cdot b:=\sum_{i=1}^2 a_i b_i$ the Euclidean scalar product between points $a$ and $b$ in $\R^2$, and we denote the corresponding norm by $|a|:=\sqrt{a \cdot a}$. 
	
	Let $\mathbb M^2$ be the set of $(2 \times 2)$-matrices and by $\mathbb M^{2}_\mathrm{ sym}$ the space of symmetric $(2 \times 2)$-matrices. The space $\mathbb M^2$ is endowed with Frobenius  inner product $E:F:= \sum_{i,j =1}^{2} E_{ij} F_{ij}$ and, with a slight abuse of notation, we denote the corresponding norm by $|E|:=\sqrt{E:E}$. 
	
	\subsection*{Topology}
	Since the model considered in this manuscript is two-dimensional, if not otherwise stated, all the sets are contained in $\R^2$.  We denote the cartesian product in $\R^2$ of two sets $A \subset \R$ and $B\subset \R$ by $A\times_{\R^2}B:= \{ (a,b): \, a \in E \text{ and } b \in F \}$. 
    For any set $E \subset \R^2$, we denote 
by $\Int(E)$, $\overline{E}$ and $\partial E$  interior, the closure and the topological boundary of $E$, respectively.  Furthermore, we denote by $Cl(F)$ the closure of a set $F$ in $ E$ with respect to the relative topology in $E$. 
	For any $r \in \R$ we define $rE := \lk rx: x \in E  \rk$.
	
	Given $\nu \in \mathbb{S}^1$ and $x \in \R^2$, $ Q_{\rho, \nu} \xp$ is the open square of sidelength $2 \rho > 0$ centered at $x$ and whose sides are either perpendicular or parallel to $\nu$. Note that if $\nu = \textbf{e}_ \textbf{1}=  \bm{e}_{\bm{2}}$ and $x= 0$, $Q_\rho := Q_{\rho, \nu} (0)  = \left( -\rho, \rho \right)\times \left( -\rho, \rho \right)$. 
	We define by $I_\rho$ the symmetric segment $I_\rho := \left[ -\rho, \rho \right]\subset\mathbb{R} $. Furthermore, a \emph{(parametrized) curve} in $\mathbb{R}^2$ is a continuous  function $r:[a,b]\to\mathbb{R}^2$ injective in $(a,b)$ for $a,b\in\mathbb{R}$ with $a<b$ and its image is referred to as the \emph{support of the curve}. Moreover, since any continuous bijection between compact sets has a continuous inverse,  the support of any curve is homeomorphic to a closed interval in $\mathbb{R}$ (see \cite[Section 3.2]{F}).

	Given $y_0 \in \R^2$ and $\rho >0$, we define the \emph{blow-up map} centered in $y_0$ with radius $\rho$ as the function $\sigma_{\rho, y_0}: \R^2 \rightarrow \R^2$ by
	\begin{equation*}
		\sigma_{\rho, y_0} \left(y \right) := \frac{y-y_0}{\rho}
	\end{equation*}
	for all $y \in \R^2$. Observe that if we apply $\sigma_{\rho, x }$ to $Q_{\rho, \nu}\xp$ and $\overline{ Q_{\rho, \nu}\xp }$, we obtain that $\sigma_{\rho, x } \left( Q_{\rho, \nu}\xp \right) = Q_{1, \nu}(0)$ and $\sigma_{\rho, x } ( \overline{Q_{\rho, \nu}\xp} ) = \overline{Q_{1, \nu}(0)}$. When $y_0 = 0$, we write $\sigma_\rho$ instead of $\sigma_{\rho,0}$. We denote by $\pi_i$ the projections onto $x_i$-axis for $i=1,2$, i.e., the maps $\pi_1 : \R^2 \to  \R  $ and $\pi_2 : \R^2 \to \R $ are such that $\pi_1 (x_1, x_2) = x_1$ and $\pi_2 (x_1, x_2) = x_2$ for every $x = (x_1,x_2)\in \R^2$.
	
	Finally, let 
	$\dist(\cdot, E)$ and $\sdist(\cdot, \partial E)$ be the \emph{distance} function from $E$ and the \emph{signed distance} from $\partial E$ respectively, where we recall that $\sdist(\cdot, \partial E)$ is defined by 
	\begin{equation*}
		\sdist(x, \partial E):= \begin{cases}
			\dist (x, E) & \textrm{if } x \in \R^2 \setminus E, \\
			-\dist (x, E) & \textrm{if } x \in E
		\end{cases}
	\end{equation*}
	for every $x \in \R^2$.
	\subsection*{Geometric measure theory}
	We denote by $\mathcal{L}^2{(B)}$ the $2$-dimensional Lebesgue measure of any Lebesgue measurable set $B \subset \R^2$ and by $\mathbbm{1}_B$ the characteristic function of $B$. For $\alpha \in \left[0, 1 \right]$ we denote by $B^{(\alpha)}$ the set of points of density $\alpha$ in $\R^2$, i.e.,  
	$$ B^{(\alpha)} := \lk x \in \R^2: \lim_{r \to 0}{ \frac{\mathcal{L}^2({B \cap B_r \xp})}{\mathcal{L}^2({B_r \xp})} } = \alpha \rk. $$ 
	
	We denote the distributional derivative of a function $f \in L^1_{\mathrm{ loc}}(\R^2)$ by $Df$ and define it as the operator $D: C^\infty_c (\R^2) \to \R$ such that
	$$\int_{\R^2}{Df \cdot \varphi} =- \int_{\R^2}{f \cdot  \grad \varphi \, dx}, $$
	for any $\varphi \in C^\infty_0 (\R^2)$.
	
	We denote with $\Hs^1$ the 1-dimensional Hausdorff measure. 
	We say that $K \subset \R^2$ is $\Hs^1$-rectifiable if $0 <\Hs^1 (K) < +\infty$ and $\theta_* (K,x) = \theta^* (K,x)=1 $ for $\Hs^1$-a.e. $x\in K$, where
	$$ \theta_* (K,x) :=  \liminf_{r \to 0^+}{ \frac{\Hs^1 (K \cap B_r \xp) }{2r} } \quad \textrm{and} \quad \theta^* (K,x):=\limsup_{r \to 0^+}{ \frac{\Hs^1 (K \cap B_r \xp) }{2r} }. $$ 
	We define sets of finite perimeter as in \cite[Definion 3.35]{AFP} and the reduced boundary $\partial ^* E$ of a set $E$ of finite perimeter by
	\begin{equation}
		\label{reducedboundary}
		\partial^* E := \lk  x \in \R^2: \exists \nu_E \xp := - \lim_{r \to 0} {\frac{D \mathbbm{1}_{E}(B_r \xp)}{\abs{D \mathbbm{1}_ {E} }(B_r \xp)}  }, \abs{\nu_E \xp } =1 \rk,
	\end{equation}
	where we refer to $\nu_E \xp$ as the measure-theoretical unit normal at $x \in \partial E$.
	
	Let $x \in \R^2$, if $\nu_E \xp$ exists $T_{x,\nu_E \xp} := \lk y \in \R^2: y \cdot \nu_E \xp = 0 \rk $ stands for the approximate tangent line and $H_{x,\nu_E \xp} := \lk y \in \R^2: y \cdot \nu_E \xp \le 0 \rk$ is its corresponding half space at $x$. 
	
	For any set $E \subset \R^2$ of finite perimeter, by \cite[Corollary 15.8 and Theorem 16.2]{M} it yields that
	\begin{align}
		\Hs^{1}(E^{(1/2)} \setminus \partial^* E) = 0 \quad \text{and} \quad \partial ^* E \subset E^{(1/2)}. \label{eq:differenceonE_halfreducedboudnary} 
	\end{align}
	Moreover, for any set $E \subset \R^2$ of finite perimeter, we have 
		\begin{equation}
			\label{eq:decompositionE} 
			\partial E = N \cup \partial^* E \cup (E^{(1)} \cup E^{(0)}) \cap \partial E, 
	\end{equation}
	where $N$ is a $\Hs^1$-negligible (see \cite[Section 16.1]{M}).

We denote by $U_{k,l}$ the (half-open) dyadic square, i.e., 
$$U_{k,l} := [0, 2^{-k}) \times_{\R^2} [0, 2^{-k}) + 2^{-k}l$$
for any $k \in \N$ and $l \in \mathbb{Z}\times \mathbb{Z}$. We denote by $\mathcal{Q}$ the family of dyadic squares and by $\mathcal N^1$ {\itshape the net measure}. More precisely, for any Borel set $E$,
	\begin{equation}
		\label{eq:netmeasure1}
		\mathcal N^1 (E) := \liminf_{\delta \to 0}{\mathcal{N}^1_\delta}(E),
	\end{equation}
	is the $\mathcal N^1$-measure of $E$, where
	\begin{multline}
	\label{eq:netmeasure2}
	    \mathcal{N}^1_\delta (E) := \inf \{ \sum_{i \in I}\mathrm{ diam}(U_i): \lk U_i \rk_{i\in I } \subset \mathcal{Q} \text{ is a countable disjoint covering of } E \\  \text{ and } \mathrm{ diam}(U_i) \le \delta     \} .
	\end{multline}
	Note that this measure does not coincide with the Hausdorff measure, however we have the following equivalence (see \cite[Chapter 5]{Mp})
	\begin{equation}
		\label{eq:equivalencE_hausdorffnetmeasure}
		\Hs^1 (E) \leq \mathcal{N}^1 (E) \leq 2^\frac{5}{2} \Hs^1 (E),
	\end{equation}
	for any Borel set $E\subset \R^2$.

	\subsection*{Functions of bounded pointwise variation}
	Given a function $h:[a,b] \to \R$ we denote the pointwise variation of $h$ by
	$$ \textrm{Var}\, h := \sup{ \left\{ \sum_i^n \abs{ h(x_i) - h(x_{i-1})}: P:= \{ x_0, \ldots , x_n\} \ \textrm{is a partition of $[a,b]$}  \right\} } . $$
	
	We say that $h:(a,b) \to \R$ has finite pointwise variation if $\textrm{Var}\, h < \infty.$ We recall that for any function $h$ such that $\textrm{Var}\, h < \infty$, $h$ has at most countable discontinuities and there exists $h(x^\pm):= \lim_{z \to x ^\pm}h(z)$. In the following given a function $h: [a,b] \to \R$ with finite pointwise variation, we define
	\begin{equation*}
		h^{-} (x):= \min \{ h (x^{+} ), h(x^{-} )\} = \liminf_{z \to x} h(z) 
	\end{equation*}
and 
\begin{equation*}
	h^{+} (x):= \max \{ h (x^{+} ), h(x^{-} )\} = \limsup_{z \to x} h(z).
\end{equation*}
In view of \cite[Corollary 2.23]{L} and with slightly abuse of notation, the limits
\begin{equation}
    \label{eq:limitsh}
    h^+(a):= \lim_{x \to a^-} h(x) \quad \text{and} \quad h^-(b):=\lim_{x \to b^-} h(x)
\end{equation}
are finite.  
	\section{Mathematical setting and main results} \label{sec:model}

In this section we present  the model introduced in this manuscript with some preliminaries, and then we state the main results of the paper outlining the consequences for the related  one-phase setting of \cite{KP,KP1,KP2} and the multiple-phase setting of film multilayers considered in \cite{LlP1}. 
 
	\subsection{The two-phase model}
	Let $\Om:= (-l,l)\times_{\R^2}(-L,L) \subset \R^2$ for positive parameters $l,L\in \R$. We begin by introducing the family $\mathcal{C}_\mathbf{m}$ of admissible configurations and, in particular, the admissible substrate regions. 
 
 Roughly speaking, an  admissible substrate region  $S\subset\Om$ is characterized as the subgraph of a upper semicontinuous height function $h$ with finite pointwise variation to which we subtract a closed $\mathcal{H}^1$-rectifiable set $K$ such that $\mathcal{H}^1(K)<\infty$, which represents the substrate internal cracks. 
 More precisely, we consider the \emph{family of admissible (substrate) heights} $\mathrm{ AH}(\Om)$ 
	defined by 
	\begin{align}
		\mathrm{ AH}(\Om) := \left\{  h : [-l,l] \to [0,L]:  h \text{ is upper semicontinous and } \mathrm{ Var }\, h < \infty
		\right\} \label{eq:definitionAH}
	\end{align}
	and let $S_h$ denote the closed subgraph  with height $h  \in \mathrm{ AH}(\Om)$, i.e., 
	\begin{equation}
		\label{eq:uppergraphS}
		S_h:= \{ (x,y): -l<x<l, y\le h(x) \}.
	\end{equation}
	We then define the \emph{family of admissible  (substrate) cracks} $\text{AK}(\Om)$  by
	\begin{equation}
		\label{eq:definitioncracksS}
		\mathrm{ AK} (\Omega) := \{ K \subset \Om : \text{ $K$ is a closed set in $\R^2$, 
			$\Hs^1$-rectifiable and } \Hs^1(K) < \infty \}
	\end{equation}
	and the \emph{family of pairs of admissible heights and  cracks} $\text{AHK}(\Om)$ by
	\begin{equation}
		\label{eq:definitionAHK}
		\mathrm{ AHK} (\Omega) := \{ (h,K)\in \mathrm{ AH}(\Om)\times\mathrm{ AK}(\Om): \,  K \subset { \overline{\Int(S_h)} } \}.
	\end{equation}
	Finally, given $(h,K)\in \text{AHK}(\Om)$ we refer to the set  
	\begin{equation}
		\label{eq:substrateheight}
		S_{h,K}  :={ ({S_h} \setminus K)} 
		\cap \Omega,
	\end{equation}
	as the \emph{substrate with height $h$ and cracks $K$},  
	and  we define the \emph{family of admissible substrates} as 
	\begin{equation}
		\label{eq:definitionAS}
		\mathrm{ AS}(\Om):= \{S \subset \Om\,:\,  \text{$S=S_{h,K}$ for a pair $(h,K) \in \mathrm{ AHK} (\Omega)$}\}.
	\end{equation}

	We observe that  
	\begin{equation}
		\label{eq:uniongraphs}
		\partial {S_{h,K}} =\partial {S_h} \cup K 
	\end{equation}
	for every $S_{h,K} \in \mathrm{ AS}(\Om)$, 
	so that 
	$\overline{S_{h,K}} = S_h$ and $\Int(S_{h,K}) = \Int(S_h) \setminus K$. We denote the  \emph{jumps points} and the \emph{vertical filament points} of the graph of $h \in  \AH$ by 
\begin{equation}
	\label{eq:filamentjumph}
		J(h) := \{ x \in (-l,l):h^- (x) \neq h^+ (x) \} \quad \text{and}  \quad 
		F(h):= \{ x \in (-l,l)  : h^+(x) < h (x) \},
	\end{equation}
	respectively.
By \cite[Corollary 2.23]{L} and thanks to the fact that $h \in \AH$, $J(h)$ and $F(h)$ are countable.
	Moreover, it follows that 
	$\partial {S_h}$ 
	is connected and, $\partial S_h$ and $\partial S_{h,K}$ have finite $\Hs^1$-measure. By \cite[Lemma 3.12 and Lemma 3.13]{F}, for any $h \in \mathrm{ AH}$, $\partial S_h$ is rectifiable and applying the Besicovitch-Marstrand-Mattila Theorem (see \cite[Theorem 2.63]{AFP}), $\partial {S_h}$ is $\Hs^1$-rectifiable, and hence, $\partial S_{h,K}$ is $\Hs^1$-rectifiable. Furthermore, applying \cite[Proposition A.1]{KP1} $S_h$ and $S_{h,K}$ are sets of finite perimeter. 
	
	The following result allows to interchangeably bound the pointwise  variation of a function $h\in\AH$ from the $\Hs^1$-measure of $\partial S_h$, and vice versa. 
	
	\begin{lemma}
		\label{lem:appendix1}
		Let $h \in \AH$. Then
		\begin{equation}
			\label{eq:equivalenceVar}
			\mathrm{ Var}\, h \le \Hs^1 (\partial S_h) \le 2 l + 2\mathrm{ Var}\, h,
		\end{equation}
		where $S_h$ is defined as in \eqref{eq:uppergraphS}.
	\end{lemma}
	\begin{proof}
		The proof is divided in two steps.
		
		{\itshape Step 1.} We prove the left inequality of \eqref{eq:equivalenceVar}. We proceed as in \cite[Section A.2.3]{ChB}. Let $m \in \N$ and let $\{ l_i: i \in \{0 , \ldots , m\}, l_0 = -l, l_m = l \text{ and } l_i < l_{i+1} \}$ be a partition of $[-l,l]$. Take $i \in \{0, \ldots, m-1 \}$ and define by $L_i$ the segment connecting $(l_i, h(l_i))$ with $(l_{i+1}, h(l_{i+1}))$. By \cite[Lemma 3.12]{F}, there exists a parametrization $r_i : [0,1] \to \R^2$ of $\partial \Int(S_h) \cap ((l_{i},l_{i+1}) \times_{\R^2} [0,L))$, whose support $\gamma_i$ joins the points $(l_i, h(l_i))$ and $(l_{i+1}, h(l_{i+1}))$, furthermore, it follows that
		$$
		\abs{ h(l_{i+1}) -  h(l_{i})} \le \sqrt{ \abs{l_{i+1} - l_i}^2 + \abs{ h(l_{i+1}) -  h(l_{i})}^2  } = \Hs^1(L_i) \le \Hs^1(\gamma_i).
		$$
		Moreover, repeating the same argument for any $i \in \{0 , \ldots, m-1 \}$ we have that
		$$
		\sum_{i=0}^{m-1} \abs{ h(l_{i+1}) -  h(l_{i})} \le \sum_{i=0}^{m-1} \Hs^1(\gamma_i) \le \Hs^1(\partial S_h),
		$$
		where in the last inequality we have used that 	    \begin{equation}\label{eq_boundary}
			\partial S_h = \partial \Int(S_h) \cup (\partial S_h \cap S_h^{(0)}) \cup N,
		\end{equation}
		where $N$ is a $\Hs^1$-negligible set. Taking the supremum aver all partitions of  we obtain the left inequality of \eqref{eq:equivalenceVar}.
		
		{\itshape Step 2.} In this step, we prove the right inequality of \eqref{eq:equivalenceVar}. We observe that
		$$
		\Int(S_h) = \{ (x,y): -l<x<l, y<h^-(x) \}
		$$
		and so, 
		$$
		\begin{matrix}
			r: & [-l,l] & \to & \partial \Int(S_h)\\
			&x & \mapsto & (x, h^-(x))
		\end{matrix}
		$$
		is a parametrization of $\partial \Int(S_h)$, whose support we denote by  $\gamma$. Therefore, from \cite[Definition 4.6 and Remark 4.20]{L} it follows that
		\begin{equation}
			\label{eq:appendix1}
			\Hs^1 (\partial \Int(S_h)) = \sup\left\{ \sum_{i=0}^{m-1} \abs{r(x_{i+1}) - r(x_i)}    \right\},
		\end{equation}
		where the supremum is taken over all partitions of $[-l,l]$. By definition of $r$ and thanks to \eqref{eq:appendix1}, we see that
		$$
		\Hs^1 (\partial \Int(S_h)) \le \sup \left\{ \sum_{i=0}^{m-1} \left( \abs{x_{i+1} - x_i} + \abs{h^-(x_{i+1}) - h^-(x_i)} \right) \right\} \le 2l + \mathrm{ Var} \,h,
		$$
		where we used the fact that $\mathrm{ Var} \,h^- \le \mathrm{ Var} \,h$. Finally, by \eqref{eq_boundary} we have that $\Hs^1 (\partial S_h) = \Hs^1(\partial \Int (S_h)) + \Hs^1 (\partial S_h \cap S_h^{(0)}) $ and since $\Hs^1 (\partial S_h \cap S_h^{(0)}) \le \mathrm{ Var}\, h$, by the fact that $\partial S_h \cap S_h^{(0)}$ is the union of vertical segments, we can deduce the right inequality of \eqref{eq:equivalenceVar}.
	\end{proof}

	We now introduce the family of admissible region pairs and configurations. 
 
	\begin{definition}[Admissible regions and configurations] \normalfont
		\label{def:admissibleconf}
		We define the families of admissible pairs $\B (\Om)$ and of admissible configurations $\mathcal{C}$ by
		\begin{multline*}
			\B (\Om):= \{ (A,S):   A\text{ is } \mathcal{L}^2\text{-measurable,} \, \partial A \textrm{ is $\Hs^1$-rectifiable, } \Hs^1 (\partial A) < \infty,\, \\
 {\text{there exists }} (h,K) \in \mathrm{ AHK} (\Omega), \, S=  S_{h, K} \in \mathrm{AS}(\Om),  \, \\
  S_{h,K} \subset  \overline{A}    \subset \overline \Om  \text{ and }  \partial A \cap \Int(S_{h,K}) = \emptyset \}, 
		\end{multline*}		and  
		\begin{equation*}
			{{ \mathcal{C}}} 
			:= \left\{ \left( A,S, u \right) { \in \wt  \C}: { (A,S) \in \B}
			 \right\},
		\end{equation*}
		respectively.
	\end{definition}
In the following we also refer to the sets $\overline{A}$, $\overline{S}$, and $\overline{A}\setminus S^{(1)}$ with respect to  an admissible pair $(A,S)\in\B  $ as the \emph{composite region}, the \emph{substrate region}, and the \emph{film region} of the admissible pair.  Moreover, we refer to $S^{(1)}$ and $A^{(1)}\cap S^{0}$ as the \emph{substrate} and the \emph{film bulk regions}, respectively. 

In Theorem \ref{thm:lowersemicontinuity} we will need to consider a natural extension of the families $\mathcal{B}$ and $\mathcal{C}$, which we denote as $\wt{\mathcal{B}}$ and $\wt{\mathcal{C}}$, respectively.

\begin{definition}\normalfont
\label{def:generaladmissibleconf}
	We define the families of admissible pairs $\wt \B(\Om)$ and of admissible configurations $\wt {\mathcal{C}}$ by
	\begin{multline*}
			\wt \B(\Om) := \{ (A,S):   A,S\text{ are } \mathcal{L}^2\text{-measurable,} \, \partial A , \partial S \textrm{ are $\Hs^1$-rectifiable, } \\ \Hs^1 (\partial A),\Hs^1 (\partial S) < \infty, 
    \, 
  S\subset  \overline{A}    \subset \overline \Om  \text{ and }  \partial A \cap \Int(S) = \emptyset \}, 
		\end{multline*}
and  
		\begin{equation*}
			\wt{ \mathcal{C}} 
			:= \left\{ \left( A,S, u \right): (A,S) \in \wt \B , u \in { H^{1}_\mathrm{{loc}} \left( \Int (A); \R^2 \right)  } \right\},
		\end{equation*}
		respectively.
\end{definition}
We observe that $\B\subset \wt \B$, since for any $(A,S)\in \B$ there exists $(h,K) \in \text{AHK}(\Om)$ such that $S= S_{h,K} \in \text{AS}(\Om)$ and $\partial S_{h,K}$ is $\Hs^1$-rectifiable, and thus, $(A,S)\in\wt \B(\Om)$.

Notice that, for simplicity, in the absence of ambiguity we omit the dependence on the set $\Om$ in the notation ${ \wt \B (\Omega)\,}$ and $\B(\Om)$ by writing in the following only ${ \wt \B\,}$ and $\B$, respectively.
	
	\begin{remark} \normalfont
		\label{rem:finiterperimeter}
		We observe that any bounded $\mathcal{L}^2$-measurable set $A \subset \R^2$ such that $\Hs^1(\partial A) < \infty$ is a set of finite perimeter in $\R^2$ by \cite[Proposition A.1]{KP1}.
	\end{remark}

 We now equip the family $\wt\B$ with a topology. 
	
	\begin{definition}[${\tau_{\B}}$-convergence] \normalfont
		\label{def:taubconvergence}
		A sequence $\{(A_k, S_k) \} \subset \wt \B$ 
${\tau_\B}$-converges to $(A,S)\in \wt \B$,  
if
		\begin{itemize}
			\item[-] $\sup_{k\in\mathbb{N}}{ \Hs^1 \left( \partial A_k \right)} < \infty$,$\sup_{k\in\mathbb{N}}{ \Hs^1 \left( \partial S_k \right)} < \infty$, 
			\item[-] $\mathrm{ sdist} \left( \cdot, \partial A_k \right) \rightarrow \mathrm{ sdist} \left( \cdot, \partial A \right)$ locally uniformly in $\R^2$ as $k \rightarrow \infty$,
			\item[-] $\mathrm{ sdist} \left( \cdot, \partial S_k \right) \rightarrow \mathrm{ sdist} \left( \cdot, \partial S \right)$ locally uniformly in $\R^2$ as $k \rightarrow \infty$.
		\end{itemize}
	\end{definition}
	
	It will follow from  Lemma \ref{lem:1} 
	 below that the  $\tau_B$-convergence is closed in the subfamily of admissible triples $\B_{\mathbf{m}}\subset\B$ whose definition depending on the vector ${\mathbf m}=(m_0,m_1) \in \N\times\N$ we now provide.  
	
	\begin{definition} \normalfont
		For any $\mathbf{m} := {(m_0,m_1)} \in \N \times \N$  the family $\B_{\mathbf{m}}$  is given by all  pairs $(A,S) \in \B$
	such that  $\partial A$ and $\partial S$
have at most $m_1$ and $m_0$-connected components, respectively. Let us also define   
		\begin{equation}
			\label{eq:initialproblem}
			\mathcal{C}_\mathbf{m} 
			:= \left\{ \left( A, { S,}
			u \right) \in \mathcal{C}
			: (A,S
) \in \B_{\mathbf{m}} \right\} \subset \mathcal{C}
			.
		\end{equation}
	\end{definition} 

 We denote the topology with which we equip the family $\wt\C$ by $\topC$. 
 
	\begin{definition}[$\topC$-Convergence] \normalfont
		A sequence $\{ \left(A_k,{ S_k}
		, u_k \right) \}_{k\in \N} \subset \C$ is said to $\topC$-convergence to $\left( A,S
		, u \right) \in \C$, denoted as $\left(A_k,{S_k}
		, u_k \right) \xrightarrow{\topC} \left( A,S
		, u \right)$, if
		\begin{itemize}
			\item [-] $(A_k,{S_k}
			) \xrightarrow{\tau_\B} (A,S
			)$,
			\item[-] $u_k \rightarrow u$ a.e. in $\mathrm{{Int} \left(A\right)}$.
		\end{itemize}
	\end{definition}
	We now state some properties of the topology $\topC$. 
 
	\begin{remark}
		\label{remark:1}
		\normalfont
		We notice that:
		\begin{itemize}
			\item[(i)] 
			The following assertions are equivalent
			\begin{itemize}
				\item[(i.$1$)] $\sdist (\cdot, \partial E_k) \rightarrow \sdist (\cdot, \partial E)$ locally uniformly in $\R^2$.
				\item[(i.$2$)] $E_k \xrightarrow{\mathcal{K}} \overline{E}$ and $\R^2 \setminus E_k \xrightarrow{\mathcal{K}} \R^2 \setminus \Int(E)$.
			\end{itemize}
			Moreover, these imply that $\partial E_k \xrightarrow{\K} \partial E$.
			\item[(ii)]If there exist $(h_k,K_k) \in \mathrm{ AHK} (\Omega)$ and $(h,K)  \in \mathrm{ AHK} (\Omega)$ such that $E_k = E_{h_k,K_k} \in \AS$ and $E = E_{h,K} \in AS$, for every $k \in \N$, we observe that Item (i.$1$) above is equivalent to
			$$ E_k = E_{h_k,K_k} \xrightarrow{\mathcal{K}} E_{h} \quad  \text{and} \quad \R^2 \setminus E_k = \R^2 \setminus E_{h_k,K_k}\xrightarrow{\mathcal{K}} (\R^2 \setminus \Int( E_{h})) \cup K,$$
			where $E_{h_k,K_k}$, $E_{h,K}$ are defined as in \eqref{eq:substrateheight} and $E_{h}$ is defined as in \eqref{eq:uppergraphS}.
			\item[(iii)] Let $\{(E_k,F_k)\}\subset \R^2 \times \R^2$ be a sequence of bounded sets and let $E,F \subset \R^2 $ be two bounded sets such that $\partial E_k \cngK \partial E$ and  $\partial F_k \cngK \partial F$. 
			In view of the Kuratowski convergence (see  \cite[Section 6.1]{AFP}, \cite[Chapter 4]{Dm} or \cite[Appendix A.1]{KP}), we observe that for every $x \in \partial E \setminus \partial F $ there exist $r := r(x) >0$ and $k_{r,x} \in \N$ such that  $B(x,r) \cap \partial F_k = \emptyset$ for any $k \ge  k_{r,x}$. Similarly, for every $x \in \partial F \setminus \partial E $ there exists $r':= r'(x)>0$ and $k_{r',x} \in \N$ such that  $B(x,r)\cap \partial E_k = \emptyset$ for any $k \ge k_{r',x}$. 
		\end{itemize} 
	\end{remark}
	

From the next result the closedness and the compactness (see Theorem \ref{thm:compactnessB})
of the family $\B_{\textbf{m}}$  with respect to the topology  $\tau_{\B}$  follows for every $\textbf{m}:=(m_0,m_1) \in \N \times \N$. 
 
	\begin{lemma}
		\label{lem:1}
		Let $\{E_k\}$ be a sequence of $\mathcal{L}^2$-measurable subsets of $\overline \Om$ having $\Hs^1$-rectifiable boundaries $\partial E_k$ with at most $m$-connected components such that
		\begin{itemize}
			\item[-] $\sup_{k} \Hs^1 (\partial E_k)< \infty$,
			\item[-] $\sdist(\cdot, \partial E_k) \to \sdist(\cdot, \partial E)$ locally uniformly in $\R^2$ as $k \to \infty$ for a set $E \subset \R^2$.
		\end{itemize} 
		Then,  $\partial E$ is $\Hs^1$-finite, $\Hs^1$-rectifiable, and with at most $m$-connected components, and $E \subset \overline \Om$ is $\mathcal{L}^2$-measurable.  Furthermore, if 
		$E_k = E_{h_k,K_k}\in \mathrm{ AS}(\Om)$ for every $k \in \N$ and for some $(h_k,K_k)   \in \mathrm{ AHK} (\Omega)$, then  
		\begin{equation}
			\label{eq:varfinite}
			\sup_{k} (\Hs^1 (K_k) + \mathrm{ Var}\, h_k)<\infty 
		\end{equation}
		and there exists $(h,K) \in \mathrm{ AHK} (\Omega)$ such that $E = E_{h,K} \in \mathrm{ AS}(\Om)$.
	\end{lemma}
	\begin{proof}
		The fact that $\partial E$ is $\Hs^1$-finite, $\Hs^1$-rectifiable, and with at most $m$-connected components, is a direct consequence of \cite[Lemma 3.2]{KP}. Since $\Hs^1 (\partial E) < \infty$, it follows that $ \mathcal{L}^2 (\overline{E} \setminus \Int(E)) = \mathcal{L}^2 (\partial E) =0$, by applying \cite[Theorem 14.5]{bartle} to $E \setminus \Int(E) \subset \overline{E} \setminus \Int(E)$ we infer that $E \setminus \Int(E)$ is $\mathcal{L}^2$-measurable and so,  $\mathcal{L}^2 (E \setminus \Int(E)) = 0$. Therefore, $E = E \setminus \Int(E) \cup \Int(E)$ is  $\mathcal{L}^2$-measurable.
				
		It remains to prove the last assertion of the statement. Let $(h_k,K_k) 
		\in \mathrm{ AHK} (\Omega)$ such that $E_k = E_{h_k,K_k} \in \mathrm{ AS}(\Om)$ for every $k$. We begin by observing that \eqref{eq:varfinite} is a direct consequence of \eqref{eq:uniongraphs}, by applying \eqref{eq:equivalenceVar} to $h_k$. To conclude the proof we proceed in 2 steps.	
				
		{\itshape Step 1.} We claim that  
		$\overline{E}  = E_{h}$, where $h  $ is the upper semicontinous function defined by
		$$h(x_1):= \sup \{ \limsup_{k \to \infty} h_k (x_1^k) : x^k_1 \to x_1 \}.$$
		Let $x=(x_1,x_2) \in \overline{E}$, by Remark \ref{remark:1}-(i) we observe that there exists $x_k = (x_1^k,x_2^k) \in E_k$ such that $x_k \to x$.  We deduce that
		$$x_2 = \lim_{k \to \infty} x^k_2 \le \limsup_{k \to \infty} h_k(x_1^k) \le h(x_1),$$
		and 
		by \eqref{eq:uppergraphS} we deduce that $\overline{E} \subset E_{h} $. Now let $x= (x_1,x_2) \in E_{h}$, by definition we observe that
		$$x_2 \le h (x_1) := \sup \{ \limsup_{k \to \infty} h_k (x^k_1) : x^k_1 \to x_1 \}.$$
		Let $x_k = (x^k_1, x^k_2) \in \Om$ such that $x^k_1 \to x_1$, $h_k(x^k_1) \to h(x_1)$ and define $x_2^k := \min \{ x_2, h_k(x^k_1) \}$. It follows that $x_k \in \overline{E_k}$ and $x_k \to x$, by Kuratowski convergence we have that $x \in \overline{E}$, therefore, $E_{h} \subset \overline{E}$.
		
		{\itshape Step 2.} We claim that $\Int(E) = \Int ({E_{h,K}})$, where $K:= \partial E \cap { \overline{\Int(E_h)}} $. 
		Notice that $K$ is a closed set in $\R^2$ and since $\partial E$ is $\Hs^1$-rectifiable, we deduce that $K$ is also $\Hs^1$-rectifiable. 
			On one hand, we see that
			\begin{equation*}
			\Int(E) = \Int(E) \setminus (\partial E \cap  { \overline{\Int(E_h)}}) \subset \Int(E_h) \setminus (\partial E \cap { \overline{\Int(E_h)}}) =: \Int(E_h) \setminus K= \Int ({E_{h,K}}),
		\end{equation*}
		where in the first equality we used the fact that $\Int(E) \cap \partial E = \emptyset$ and in the inclusion we used Step 1 and the fact that $E \subset \overline{E} = E_h$. On the other hand, let $x \in \Int(E_{h,K}) = \Int(E_h) \setminus (\partial E \cap { \overline{\Int(E_h)}})$ and assume by contradiction that $x \notin \Int(E)$. This assumption implies that either $x \in \partial E$ or $x \in \Om \setminus \overline{E} $, which is a contradiction by the facts that $\partial E \subset \overline{E} = E_h$ and 
		$$
		x \in \Int(E_h) \setminus(\partial E \cap { \overline{\Int(E_h)}}) =\Int(E_h) \setminus \partial E   \subset E_h = \overline{E} .
		$$
		Finally, observe that $(h,K) \in \mathrm{ AHK} (\Omega)$. Thanks to the uniqueness of Kuratowski convergence, the facts that $\overline{E_{h,K}} = E_h $ and $\Int(E_{h,K}) = \Int(E_h) \setminus (\partial E \cap { \overline{\Int(E_h)}})$, and in view of Remark \ref{remark:1}-(i) we conclude from the previous two steps that $E = E_{h,K}$. 
	\end{proof}

	The total energy  $\mathcal{F} :  \wt{\mathcal{C}}
	\rightarrow \left[ 0, +\infty \right]$ of admissible configurations is given as the sum of two contributions, namely the surface energy $\mathcal{S}$ and the elastic energy $\mathcal{W}$, i.e., 
 $$
		\mathcal{F}(A,S
		,u) := \mathcal{S}(A,S
		) + \mathcal{W}(A,u)
  $$
	for any $(A,S
	,u)\in\wt\C$, where we observe that the surface energy does not depend on the displacements (as a difference from \cite{KP,KP1}). The surface energy $\mathcal{S}$ is defined for any $(A,S 
	) \in \wt\B$ by
	\begin{equation}
	\label{eq:definitionenergyS}
		\begin{split}
			{\mathcal{S}} (A ,{ S} ) 
			&:= \int_{ \Om \cap (\partial^* A \setminus \partial^*{S})}{\varphi_{\mathrm{ F}}(x, \nu_A \xp)\, d\Hs^1} + \int_{\Om \cap   \partial ^*{S} \cap \partial^* A }{\varphi (x, \nu_A \xp)\, d\Hs^1} \\
   & \quad + \int_{\Om \cap (\partial ^*{S} \setminus \partial A)\cap A^{(1)}}{ \varphi_{\mathrm{ FS}} (x, \nu_{S} \xp )\, d\Hs^1}  
 + \int_{\Om \cap  \partial ^*{S} \cap \partial A  \cap A^{(1)}}{ (\varphi_{\mathrm{ F}} + \varphi )(x, \nu_A \xp ) \, d\Hs^1}    \\
   & \quad + \int_{ \Om \cap \partial A 
   \cap  A^{(1)} \cap  S^{(0)}}{2 \varphi_{\mathrm{ F}}(x, \nu_A \xp )\, d\Hs^1}   + \int_{ \Om \cap \partial A 
   \cap A^{(0)} }{2 \varphi' (x, \nu_A \xp )\, d\Hs^1}  \\
			& \quad + \int_{\Om \cap (\partial {S} \setminus \partial A) \cap \left(S^{(1)} \cup S^{(0)}\right)  \cap A^{(1)} }{ 2 \varphi_{\mathrm{ FS}} (x, \nu_{S} \xp) \, d\Hs^1}\\
   & \quad + \int_{\Om \cap \partial {S} \cap \partial A \cap S^{(1)}}{2 \varphi (x, \nu_A \xp )\, d\Hs^1  }      ,
		\end{split}
	\end{equation}
	where $\varphi_{\mathrm{F}},\, \varphi_{\mathrm{ FS}}: \overline{\Om} \times \R^2\to [0,\infty]$ and, given also the function $\varphi_{\mathrm{S}}: \overline{\Om} \times \R^2\to [0,\infty]$,   we define the functions $\varphi$ and $\varphi'$ in in 
 $C ( \overline{\Om} \times \R^2;[0,\infty])$ by 
 $$\varphi := \min\{\varphi_{\mathrm{S}}, \varphi_{\mathrm{F}} + \varphi_{\mathrm{FS}} \}\qquad\text{and}\qquad \varphi' : = \min\{\varphi_{\mathrm{F}}, \varphi_{\mathrm{S}}\}.$$
 Notice that  $\varphi_{\mathrm{ F}}, \varphi_{\mathrm{S}}, \varphi_{\mathrm{ FS}}$ represent  the anisotropic surface tensions of the film/vapor, the substrate/vapor and the substrate/film interfaces, respectively, while $\varphi$ and $\varphi'$ are referred to as the anisotropic \emph{regime surface tensions} and are introduced to include into the analysis the wetting and dewetting regimes. We refer the Reader to the Introduction for related explanation and for the motivation for the integral densities choice in \eqref{eq:definitionenergyS}.

	Similarly to \cite{DP, KP, KP1, KP2}, by also taking into account that in our setting the film and substrate regions are given as subsets of the composite regions,  the elastic energy is defined for  configurations $\left( A,S
	, u \right) \in { \wt{\mathcal{C}}}$   by
	\begin{equation*}
		\mathcal{W}(A,u):= \int_{A}{W \left( x, E \left( u \xp - E_0 \xp \right) \right)dx},
	\end{equation*}
	where the elastic density $W$ is determined by the quadratic form
	\begin{equation*}
		W \left( x, M \right) := \mathbb{C}\xp M : M,
	\end{equation*}
	for a fourth-order tensor $\mathbb{C}: \Om \to \mathbb{M}^2_\sym$, $E$ denotes the symmetric gradient, i.e., $E(v) := \frac{\grad v + \grad^T v}{2}$ for any $v \in H^1_\mathrm{{loc}} (\Om)$ and $E_0$ is
	the mismatch strain $x \in \Om \mapsto E_0 \xp \in \mathbb{M}^{2}_\mathrm{{sym}} $ defined as
	\begin{equation*}
		E_0 := \begin{cases}
			E(u_0) & \text{in } \Om \setminus S, \\ 0 & \text{in } S,
		\end{cases} 
	\end{equation*}
	for a fixed $u_0 \in H^1(\Om)$.

	\subsection{Main results}\label{sec:mainresults}
	We state here the main results  of the paper and the connection to the one-phase and multiple-phase settings. 
 
 Fix $l, L> 0$ and consider $\Om := (-l,l) \times_{\R^2} (-L,L)$.  Let $\varphi := \min\{\varphi_{\mathrm{S}}, \varphi_{\mathrm{F}} + \varphi_{\mathrm{FS}} \}, \, \varphi' := \min\{\varphi_{\mathrm{F}}, \varphi_{\mathrm{S}}\} $ for three functions $\varphi_{\mathrm{ F}}, \varphi_{\mathrm{ S}} , \varphi_{\mathrm{ FS}}: \overline{\Om} \times \R^2\to [0,\infty]$. 
		We assume throughout the paper that: 
	\begin{itemize}
		\item[(H1)] $\varphi_{\mathrm{ F}} , \varphi_{\mathrm{ FS}}, \varphi,   \varphi'   \in C ( \overline{\Om} \times \R^2;[0,\infty]) $ are Finsler norms such that 
		\begin{equation}
			\label{eq:H1}
			c_1 \abs{\xi} \le \varphi_{\mathrm{ F}}(x, \xi), { \varphi} (x, \xi), \varphi_{\mathrm{ FS}} (x, \xi) \le c_2 \abs{\xi} .
		\end{equation}
for every $x \in \overline{\Om}$ and $\xi \in \R^2$ and for two constants $0<c_1\leq c_2$.
		\item[(H2)] We have
		\begin{equation}
				\label{eq:H2}
    \varphi(x, \xi)\geq |\varphi_\textrm{FS} (x, \xi)- \varphi_{\mathrm{ F}}(x, \xi)|
			\end{equation}
	for every $x \in \overline{\Om}$ and $\xi \in \R^2$.
	\item[(H3)] $\mathbb{C} \in L^\infty (\Om; \mathbb{M}^2_\sym)$  and there exists $c_3 > 0$ such that
	\begin{equation}
		\label{eq:H3}				\mathbb{C} \xp M:M \ge 2c_3 M:M
	\end{equation}
	for every $M \in \mathbb{M}_\mathrm{{sym}}^{2 \times 2}.$
\end{itemize}
We notice that under assumptions (H1)-(H3), the energy $\mathcal{F}(A,S,u) \in [0, \infty]$ for every $(A,S,u) \in \wt{\mathcal{C}}$.

The main result of the paper is the following existence result.

\begin{theorem}[Existence of minimizers]
\label{thm:existence}
Assume (H1)-(H3) and let $\mathbbm{v}_0, \mathbbm{v}_1 \in [\mathcal{L}^2({\Om}/2), \mathcal{L}^2({\Om}))$ such that $\mathbbm{v}_0 \le \mathbbm{v}_1$. Then for every $\mathbf{m}= (m_0,m_1) \in \N\times \N$ the volume constrained minimum problem 
\begin{equation}
	\label{eq:const}
	\inf_{(A,S,u) \in \mathcal{C}_\mathbf{m},\, \mathcal{L}^2({A})= \mathbbm{v}_1,\, \mathcal{L}^2({S_{h,K}})= \mathbbm{v}_0}{\mathcal{F}(A,S,u)}
\end{equation}
and the unconstrained minimum problem
\begin{equation}
	\label{eq:uncost}
	\inf_{(A,S,u) \in \mathcal{C}_\mathbf{m}}{\mathcal{F}^{\lambda} (A,S,u)}
\end{equation}
have solution, where $\mathcal{F}^{\lambda} : \mathcal{C}_\mathbf{m} \rightarrow \R$ is defined as
\begin{equation*}
	\mathcal{F}^{\lambda} (A,S,u) := \mathcal{F} (A,S,u) + \lambda_1 \abs{ \mathcal{L}^2({A})- \mathbbm{v}_1 }+ \lambda_0 \abs{ \mathcal{L}^2({S})- \mathbbm{v}_0 },
\end{equation*}
for any ${\lambda} = (\lbd_0, \lbd_1),$ with $\lbd_0$, $\lbd_1>0$.
\end{theorem}

To prove Theorem \ref{thm:existence} 
we  apply the direct method of calculus of variations. 
On the one hand, in Section \ref{sec:compactness}, 
we show that any energy equi-bounded sequence $\{(A_k,S_k, u_k)\} \subset \mathcal{C}_\mathbf{m}$ satisfy the following compactness property. 

\begin{theorem}[Compactness in  $\mathcal{C}_\mathbf{m}$]
\label{thm:compactness}
Assume (H1) and (H3). Let $\{(A_k,{ S_{h_k,K_k}}
, u_k)\}_{k \in \N} \subset \mathcal{C}_\mathbf{m}$ be such that
\begin{equation}
\label{eqthm:compactness}
\sup_{k\in\mathbb{N}}{ \mathcal{F}(A_k,{ S_{h_k,K_k}}
, u_k) } < \infty.
\end{equation}
Then, there exist an admissible configuration $(A,S
,u) \in \mathcal{C}_\mathbf{m}$ of finite energy, a subsequence $\{\left(A_{k_n},S_{h_{k_n},K_{k_n}}
, u_{k_n}\right)\}_{n \in \N}$, a sequence $\left\{\left(\widetilde{A}_n,{\wt S_n}
, u_{k_n}\right)\right\}_{n \in \N} \subset \mathcal{C}_\mathbf{m}$ and a sequence $\{b_n\}_{n\in \N}$ of piecewise rigid displacements associated to $\widetilde{A}_n$ such that $\left(\widetilde{A}_n,{ \wt S_n}
,  u_{k_n} + b_n\right) \xrightarrow{\tau_\mathcal{C}} (A,S
,u)$, $\mathcal{L}^2 (A_{k_n}) = \mathcal{L}^2 (\wt A_{n})  $, $\mathcal{L}^2 (S_{h_{k_n}, K_{k_n}}) = \mathcal{L}^2 (\wt S_{n})  $ for all $n \in \N$ and 
\begin{equation}\label{liminfequalliminf}
\liminf_{n \rightarrow \infty}{\mathcal{F}(A_\kn,S_{h_{k_n},K_{k_n}}
, u_\kn) } = \liminf_{n \rightarrow \infty}{\mathcal{F}\left(\widetilde{A}_n,{\wt S_n}
,  u_{k_n} + b_n\right)}.
\end{equation}
\end{theorem}

On the other hand, in Section \ref{sec:semicontinuity} we show that $\mathcal{F}$ is lower semicontinuous in $\mathcal{C}_\mathbf{m}$ with respect to the topology $\tau_\C$.

	\begin{theorem}[Lower semicontinuity of $\mathcal{F}$]
		\label{thm:lowersemicontinuity}
		Assume (H1)-(H3). Let $\{ \left(A_k, S_{h_k,K_k}, u_k \right) \}_{k \in \N} \subset \mathcal{C}_\mathbf{m}$ and $\left(A,S_{h,K},u \right) \in \mathcal{C}_\mathbf{m}$ be such that $\left( A_k,S_{h_k,K_k}, u_k \right) \xrightarrow{\tau_\mathcal{C}} \left(A,S_{h,K},u \right)$. Then
		\begin{equation}
  \label{eq:lowersemicontinuityF}
			\mathcal{F} \left( A,S_{h,K}, u \right)\le \liminf_{k \rightarrow \infty}{ \mathcal{F} \left( A_k ,S_{h_k,K_k}, u_k \right) }.
		\end{equation}
	\end{theorem}

 We now describe  the consequences for the one-phase setting of the results obtained in this manuscript for the two-phase setting.   

\begin{remark}[Relation to literature models with fixed substrate] \label{rem:previous_surface_tensions}\normalfont
The energy considered in this paper can be seen as an extension of the  energies previously considered in the literature, e.g.,  in \cite{KP, FFLM2,FFLM}, by ``fixing the substrate regions''. 
More precisely, if we consider the subfamily $\B' \subset \B$ where 
$$
\B':=\{(A,S 
)\in \B: \,  \overline{\partial S\cap \Om} \text{ is a Lipschitz 1-manifold}\}$$ 
and $\C' :=\{(A,S
,u) \in \C: \, (A,S
)\in \B'\}$, then the energy $\mathcal{F}'$ defined for every $(A,S
)\in\B'$ by 
\begin{align*}
{\mathcal{F}'} \left(A ,S
,u \right) &:= \mathcal{F} \left(A ,S
,u \right) -   \int_{ \Om \cap  \partial^*{S} 
}{\varphi_{\mathrm{ FS}}\left( z, \nu_{S}(z) \right) \, d\Hs^1}- \int_{ \Om \cap (\partial A \setminus \partial {S})\cap A^{(0)} }{2 { \varphi'
}(z, \nu_A (z)  )d \Hs^1} \\ 
& =   \int_{ \Om \cap \partial^* A \setminus  { \partial {S}}}{\varphi_{\mathrm{ F}}(z, \nu_A (z))d\Hs^1} + \int_{ \Om \cap (\partial A \setminus \partial {S})\cap  \cup A^{(1)} }{2 \varphi_{\mathrm{ F}}(z, \nu_A (z)  )d \Hs^1}  \\
& \quad  + \int_{\Om \cap  \partial ^*{S} \cap \partial A  \cap A^{(1)}}{ (\varphi_{\mathrm{ F}} - \beta )(x, \nu_A \xp ) \, d\Hs^1}- \int_{\Om \cap  \partial ^*{S} \cap \partial^* A }{ \beta(z, \nu_A (z))d\Hs^1}, 
\end{align*}
where 
$\beta :=   \varphi_{\textrm{FS}} - \varphi$, is analogous to  
 the energy $\mathcal{F}'$ of  \cite[Theorem 2.9]{KP} (where the notation $\varphi$ was referring to $\varphi_{\mathrm{ F}}$), which is an extension of the energies of \cite{FFLM2,FFLM} as described in \cite[Remark 2.10]{KP}. We notice though that, even in the situation of a fixed regular substrate, i.e., by considering  the family $\C'' :=\{(A,S
,u) \in \C: \, S=S_0\}\subset\C'$
for a fixed admissible region  $S_0$ such that $\overline{\partial S_0\cap \Om}$ is a Lipschitz 1-manifold,  
 the setting considered in this manuscript allows to include into the analysis the possibility of an uncountable number of film islands (or film voids) on top of the substrate (which was instead  precluded in \cite{KP}), because of the crucial difference introduced in the setting of this manuscript consisting of always including  the substrate regions inside the admissible region $A$ (with the film region then being represented by $\overline{A}\setminus S^{(1)}$). We also notice that the hypotheses  on the surface tensions in this manuscript coincide with the ones in \cite{KP} up to the observation  that on the right hand-side of \eqref{eq:H2} one can disregard the absolute value for the setting with a fixed substrate).   
 \end{remark}

We conclude the section by outlining the consequences of the results obtained in this manuscript for the two-phase setting in the multiple-phase setting of film multilayers that is the object of investigation in  \cite{LlP1}. 

\begin{remark}[Relation to the setting of film multilayers] \label{rem:multilayers}\normalfont
In the parallel paper \cite{LlP1} the authors consider the setting in which also the film region is subject to a graph constraint, by introducing a family of admissible regions for the film and the substrate of the form $\B^1:= \{ (S_{h^1,K^1}, S_{h^0,K^0}) \in \mathrm{AS}(\Om) \times \mathrm{AS}(\Om):\, h^0 \le h^1 \text{ and } \partial S_{h^1,K^1} \cap \Int(S_{h^0,K^0}) = \emptyset  \} \subset \B$ and the related family $\C^1:= \{ (S_{h^1,K^1}, S_{h^0,K^0},u) :\,(S_{h^1,K^1}, S_{h^0,K^0})\in \B^1 \text{ and } u \in H^1_{\mathrm{loc}}(S_{h^1,K^1}; \R^2) \} \subset \mathcal{C}$ of admissible configurations. By implementing in the compactness for both the substrate and the film the arguments employed in this manuscript for the substrate, a similar result to Theorem \ref{thm:existence} is established, thus providing an existence result for the problem of films resting on deformable substrates in the presence of delaminations, which can be seen  as an extension of \cite{ DP, DP2,FFLM2}. Furthermore, in \cite{LlP1} by then  performing also an iteration procedure an existence result is  provided also for the setting  of finitely-many  film multilayers.
\end{remark}


\section{Compactness}\label{sec:compactness}

In this section, we fix $\mathbf{m} :=(m_0, m_1)\in\N \times \N$ and we prove that the families $\B_\mathbf{m}$ and $\mathcal{C}_\mathbf{m}$ are compact with respect to $\tau_\B$ and $\tau_{\mathcal{C}}$ topologies, respectively. 

\begin{proposition}
\label{prop:blaschke}
The following assertions hold:
\begin{enumerate}
\item[(i)] For every sequence of closed sets $\{E_k\}$, there exists $E \subset \R^2$ such that $E_k \xrightarrow{\K}E$.
\item[(ii)] For every sequence $\{E_k\}_{k \in \N}$ of subsets of $\R ^2 $ there exists a subsequence $ \{E_{k_l } \}_{l\in \N}$ and $E \subset  \R^2$ such that $\mathrm{{sdist}} (\cdot, \partial E_{k_l} ) \rightarrow \mathrm{{sdist}} (\cdot, \partial E) $ locally uniformly in $\R^2$.
\end{enumerate}
\end{proposition}

The proof of Proposition \ref{prop:blaschke}-(i)  and -(ii) can be found in  \cite[Theorem 6.1]{AFP} and \cite[Theorem 6.1]{KP}, respectively (see also \cite[Theorem 6.1]{AFP} for the version of Item (ii) with the signed-distance convergence replaced by the \emph{Hausdorff-metric convergence}).

\begin{theorem}[Compactness of $\mathcal{B}_\mathbf{m}$]
\label{thm:compactnessB}
Let $\{(A_k,S_{h_k,K_k}
)\} \subset \mathcal{B} _\mathbf{m}$ such that 
\begin{equation*}
\sup_{k \in \N}{\mathcal{S}(A_k,S_{h_k,K_k}
)} < \infty 
\end{equation*}
for $(h_k,K_k) \in \mathrm{AHK}(\Om)$. Then, there exist a not relabeled subsequence $\{(A_{k},S_{h_k,K_k})\} \subset \B_\mathbf{m}$ and $(A,S_{h,K}
) \in \mathcal{B}_\mathbf{m}$ such that $(A_k, S_{h_k,K_k}
) \xrightarrow{\tau_\B} (A,S_{h,K}
)$.
\end{theorem}
\begin{proof}
For simplicity we denote $S_k := S_{h_k,K_k}$. 
We begin by observing that by 
 Proposition \ref{prop:blaschke}-(ii) there exist $A \subset \R^2$ and $S \subset \R^2$ such that $\mathrm{{sdist}}( \cdot, \partial A_{k_l} ) \rightarrow \mathrm{{sdist}}( \cdot, \partial A )$ and $\mathrm{{sdist}}\left( \cdot, \partial {S_{{k_l}}} \right) \rightarrow \mathrm{{sdist}}( \cdot, \partial {S} )$ locally uniformly in $ \R^2$.	 
Let $R:= \sup_{k \in \N}{\mathcal{S}(A_k, S_{h_k,K_k}
)} $. 
In view of Remark \ref{rem:finiterperimeter} and by \eqref{eq:decompositionE} we have the following decomposition of $\partial A_k$,
\begin{align*}
\partial A_k &= \partial^* A_k \cup (\partial A_k \cap (A_k^{(0)} \cup A_k^{(1)})) \cup N_k,
\end{align*}
where $N_k$ is a $\Hs^1$-negligible set for every $k\in \N$. Thus, for every $k\in \N$ we observe that
\begin{align}
\partial A_k \setminus \partial {S_k} =  \partial^* A_k \setminus \partial {S_k} \cup \left( (\partial A_k \setminus \partial {S_k} ) \cap \left(A_k^{(0)} \cup A_k^{(1)}\right) \right)  \cup N'_k, \label{cmpct1}   
\end{align}
where $N'_k:= N_k \setminus \partial {S_k}$ is a $\Hs^1$-negligible set. Since for any $k \in \N$, $S_k$ is a set of finite perimeter, by \eqref{eq:decompositionE} we have that
\begin{align*}
\partial S_k &= \partial^* S_k \cup (\partial S_k \cap (S_k^{(0)} \cup S_k^{(1)})) \cup \widetilde{N}_k,
\end{align*}
where $ \widetilde{N}_k$ is a $\Hs^1$-negligible set for every $k\in \N$. Reasoning similarly to  \eqref{cmpct1} we have that
\begin{align}
\partial S_k \setminus \partial {A_k} &=  \partial^* S_k \setminus \partial {A_k} \cup (\partial S_k \setminus \partial {A_k} ) \cap \left(S_k^{(0)} \cup S_k^{(1)}\right)  \cup \widetilde{N}'_k, \label{cmpct11}   \\
(	\partial {S_k} \setminus \partial A_k) \cap \Ako &= \left( \partial ^*{S_k} \setminus \partial A_k \cup \left(\partial {S_k} \setminus \partial A_k  \cap  \left(S^{(1)}_k \cup \Skz\right) \right)\right) \cap \Ako \cup \widetilde{N}''_k \label{cmpct2}   
\end{align}
and
\begin{align}
\partial {S_k} \cap \partial \Ak \cap \Ako = \left( (\partial ^*{S_k} \cap \partial \Ak) \cup\left( (\partial {S_k} \cap \partial \Ak ) \cap  (S^{(1)}_k \cup \Skz) \right) \right) \cap \Ako  \cup \widetilde{N}'''_k, \label{cmpct3}   
\end{align}
where $\widetilde{N}'_k$, $\widetilde{N}''_k$ and $\widetilde{N}'''_k$ are $\Hs^1$-negligible sets for every $k \in \N$. Furthermore, we can deduce that
\begin{equation}
\label{eq:compactintersectionlast}
\begin{split}
	\partial S_k \cap \partial A_k &= \left( \partial^* S_k \cup (\partial S_k \cap (S_k^{(0)} \cup S_k^{(1)})) \right) \cap \left(  \partial^* A_k \cup (\partial A_k \cap (A_k^{(0)} \cup A_k^{(1)})) \right)  \cup \widehat{N}_k \\
	& = \left( \partial S_k^* \cap \partial^* A_k \right) \cup \left( \partial S_k\cap \partial ^* A_k \cap (S_k^{(0)} \cup S_k^{(1)} ) \right) \cup \left( \partial^*S_k \cap \partial A_k \cap (A_k^{(0)} \cup A_k^{(1)} )  \right) \\
	& \quad \left( \partial S_k \cap \partial A_k \cap ( (A_k^{(0)} \cup A_k^{(1)}) \cup (S_k^{(0)} \cup S_k^{(1)}) ) \right) \cup \widehat{N}_k,
\end{split}
\end{equation}
where $\widehat{N}_k$ is a $\Hs^1$-negligible set, for every $k \in \N$. By (H1),(H3), \eqref{cmpct1}-\eqref{eq:compactintersectionlast} and thanks to the fact that for every $k\in\mathbb{N}$, $S_k \subset \ov{A_k}$ we have that
\begin{equation}
\label{eq:compactness.1}
\begin{split}
	& c_1 \left( \Hs^1 ( \partial A_k  ) +\Hs^1 ( \partial {S_{k}} \setminus \partial A_k) \right)\\ 
	&\quad \le c_1 \left( \int_{ \partial^* {A_k} \setminus  \partial {S_k}}{d\Hs^1} + \int_{ \partial ^*{S_k} \cap \partial^* {A_k} }{d\Hs^1} \right.  + \int_{(\partial {A_k} \setminus \partial {S_k})\cap ({A_k}^{(0)} \cup {A_k}^{(1)})}{d \Hs^1}  \\
 & \qquad + \int_{(\partial^* {S_k} \setminus \partial {A_k})\cap {A_k}^{(1)}}{ d \Hs^1}   + \int_{\partial {S_k} \cap \partial {A_k} \cap {S_k}^{(1)} }{d \Hs^1}  + \int_{\partial {S_k} \cap \partial ^* {A_k} \cap {S_k} ^{(0)}}{d \Hs^1} \\ 
	&\qquad
 + \int_{(\partial {S_k} \setminus \partial A_k) \cap ({S_k}^{(1)} \cup {S_k}^{(0)}) \cap {A_k}^{(1)} }{d \Hs^1} + \int_{ \partial ^*{S_k} \cap \partial A_k  \cap {A_k}^{(1)}}{d \Hs^1}   \\
	&\qquad  + \left.  \int_{( \partial {S_k}\cap \partial A_k) \cap {S_k}^{(0)} \cap {A_k}^{(1)} }{d \Hs^1} \right)  \\
	& \quad  
 \le 2 ({\mathcal{S}} (A_k,{ S_k}) ) \leq 2R,
\end{split}
\end{equation}
for every $k \in \N$, where in the first inequality we used Lemma \ref{lem:appendix1} and in the second inequality we used \eqref{eq:H1}. It follows  that
\begin{equation}
\label{compactness4}
\Hs^1 (\partial A_k) \le \frac{2R}{c_1}
\end{equation}
and
\begin{equation}
\label{compactness5}
\Hs^1 (\partial {S_k}) = \Hs^1 ( \partial {S_k}\cap \partial A_k) + \Hs^1 (\partial {S_k} \setminus \partial A_k)  \le \frac{2R}{c_1},
\end{equation}
for any $k\in\mathbb{N}$.

In view of \eqref{compactness4} and \eqref{compactness5} we conclude by Lemma \ref{lem:1} that $A\subset \overline \Om$ is $\mathcal{L}^2$-measurable, $\partial A$ is $\Hs^1$-finite, $\Hs^1$-rectifiable, and with at most $m_1$-connected components, and that  there exists $(h,K) \in \mathrm{ AHK} (\Omega)$ such that $S = S_{h,K} \in \AS$ and $\partial S_{h,K}$ has at most $m_0$-connected components. Furthermore, in view of Remark \ref{remark:1}-(i), since $\overline{S_k} \subset \overline{A_k} $ for any $k \in \N$, we have that 
$S_{h,K}\subset \overline {A}$. 

In order to prove that $(A,S_{h,K}) \in \B_\mathbf{m}$ it remains to check that $\partial A \cap \Int(S_{h,k}) = \emptyset$, to which the rest of the proof is devoted. Assume by contradiction that 
			\begin{equation}
				\label{eq:closedness1}
				\partial A \cap \Int(S_{h,k}) \neq \emptyset.
			\end{equation}	
			Then, there exists $x \in \partial A \cap \Int(S_{h,k})$. By Remark \ref{remark:1}-(i), there exists $x_k \in \partial A_k$ such that $x_k \to x$ and hence,  by $\tau_{\B}$-convergence we observe that 
			\begin{equation}
				\label{eq:closedness2}
				\sdist(x , \partial S_{h_k, K_k}) \to \sdist(x, \partial S_{h,K}) \quad \text{as $k  \to \infty.$}
			\end{equation}
			Since by \eqref{eq:closedness1} there exists $\ep >0$ such that $
			\sdist(x, \partial S_{h,K}) = - \ep, 
			$
			we can find  $k_0 := k_0(x)$ for which $\sdist(x, \partial S_{h_{k_0},K_{k_0}})$ is negative. Then,
			$
			x \in \Int(S_{h_{k_0},K_{k_0}})
			$
			and so, there exists $\delta \le \ep/2$ such that 
			$$
			x_{k_0} \in B_\delta (x) \subset \Int(S_{h_{k_0},K_{k_0}}),
			$$
			which is an absurd since $\partial A_k \cap \Int(S_{h_k,K_k}) = \emptyset$ for every $k \in \N$.
\end{proof}

We are now in the position to prove Theorem \ref{thm:compactness}. To this end, we implement the arguments used in \cite[Theorem 2.7]{KP} to the situation with free-boundary substrates and in particular the ones contained in Step 1 of the proof of \cite[Theorem 2.7]{KP}. In fact, the original setting introduced in this paper with respect to \cite{KP} to model the delaminated interface regions allows to avoid the further modification of the film admissible regions that was performed in Steps 2 and 3 of \cite[Theorem 2.7]{KP}.

\begin{proof}[Proof of Theorem \ref{thm:compactness}]
Denote $R:= \sup_{k \in \N}{\mathcal{F}(A_k,S_{h_{k},K_{k}}
, u_k)}$. 
Without loss of generality (by passing, if necessary, to a not relabeled subsequence), we assume that
\begin{equation}\label{lim_bound}
\liminf_{k \rightarrow \infty}{\mathcal{F}(A_k,S_{h_{k},K_{k}}
, u_k)} = \lim_{k \rightarrow \infty}{\mathcal{F}(A_k,S_{h_{k},K_{k}}
, u_k) }\leq R.
\end{equation}
Since  $\mathcal{W}$ is non-negative, by  Theorem \ref{thm:compactnessB} there exist a  subsequence $\{(A_{k_n},S_{h_{k_n},K_{k_n}}
)\} \subset \B_\mathbf{m}$ and $(A,S
) \in \mathcal{B}_\mathbf{m}$ such that $(A_{k_n},S_{h_{k_n},K_{k_n}}
) \xrightarrow{\tau_\B}(A,S
)$. As a consequence of Theorem \ref{thm:compactnessB}, there exists $(h,K)\in \text{AHK}$ such that $S= S_{h,K}$.

The rest of the proof is devoted to the construction of a sequence $(\widetilde{A}_n, { \wt S_n} 
) \subset \B_\mathbf{m}$ to which we can apply \cite[Corollary 3.8]{KP} (with $P=\Int{(A)}$ and $P_n=\Int{(\widetilde{A}_{n})}$, respectively)  in order to obtain $u\in H^1_{\mathrm{ loc}}(\Int{(A)};\R^2)$ such that $(A,S
,u) \in \mathcal{C}_\mathbf{m}$ has finite energy, and a sequence $\{b_n\}_{n\in \N}$ of piecewise rigid displacements such that $\left(\widetilde{A}_n, { \wt S_n}
, u_{k_n} + b_n\right) \xrightarrow{\tau_\mathcal{C}} (A,S
,u)$. Furthermore, we observe that also Equation \eqref{liminfequalliminf} will be a consequence of such construction and hence, the assertion will directly follow.

By \cite[Proposition 3.6]{KP} applied to $A_{k_n}$ and $A$ there exist a  not relabeled subsequence $\{A_{k_n}\}$ and a sequence $\{\widetilde{A}_n\}$ with $\Hs^1$-rectifiable boundary $\partial \widetilde{A}_n$ of at most $m_1$-connected components such that  
\begin{equation}
\label{eq:boundpartialtildeA}
\sup_{n\in \N} \Hs^1 (\partial \widetilde{A}_n) < \infty,
\end{equation} that satisfy the following  properties:
\begin{itemize}
\item[(a1)] $\partial A_{k_n} \subset \partial \widetilde{A}_n$ and $\displaystyle \lim_{n \to \infty} \Hs^1 (\partial \widetilde{A}_n \setminus \partial A_{k_n}) = 0$,
\item[(a2)]  $\sdist(\cdot, \partial \widetilde{A}_n) \to \sdist(\cdot, \partial A)$ locally uniformly in $\R^2$ as $n \to \infty$,
\item[(a3)] if $\{E_i\}_{i \in I}$ is the family of all connected components of $\mathrm{{Int}}(A),$ there exist connected components of $\mathrm{{Int} }(\widetilde{A}_n)$, which we enumerate as $\{E_i^n\}_{i \in I}$, such that for every $i$ and $G \subset \subset E_i$ one has that $G \subset \subset E_i^n$ for all $n$ large (depending only on $i$ and $G$),
\item[(a4)] $\mathcal{L}^2 (\widetilde{A}_n) = \mathcal{L}^2(A_{k_n})$.
\end{itemize}
Furthermore, from the construction of $\widetilde{A}_n$ (namely from the fact that $\widetilde{A}_n$ is constructed by adding extra ``internal'' topological boundary to the selected subsequence $A_{k_n}$,  see \cite[Propositions 3.4 and 3.6]{KP}) it follows  that 
\begin{equation}\label{fromproof3.6}
\widetilde{A}_n = A_{k_n} \setminus (\partial \widetilde{A}_n \setminus \partial A_{k_n})
\end{equation}
with $\partial \widetilde{A}_n \setminus \partial A_{k_n}$ given by a finite union of closed 
$\Hs^1$-rectifiable sets  
connected to $\partial A_{k_n}$. More precisely, there exist a finite index set $J$ and a family $\{\Gamma_j\}_{j\in J}$ of closed $\Hs^1$-rectifiable sets of $\Om$ connected to $\partial A_\kn$ such that
$$
\partial \widetilde{A}_n \setminus \partial A_{k_n} = \bigcup_{j \in J} \Gamma_j. 
$$ 
We define 
$$
\widetilde{K}_n := K_{k_n} \cup ((\partial \widetilde{A}_n \setminus \partial A_{k_n}) \cap {\overline{\Int(S_{h_\kn})}}) \subset {\overline{\Int(S_{h_\kn})}}
$$
and we observe that $\widetilde{K}_n$ is  closed and $\Hs^1$-rectifiable in view of the fact that  $\partial \widetilde{A}_n \setminus \partial A_{k_n}$ is a closed set in $\Omega$ and is $\Hs^1$-rectifiable, since  $\partial \widetilde{A}_n$ is $\Hs^1$-rectifiable. Therefore, $(h_\kn, \widetilde{K}_n) \in \text{AHK}(\Om)$ and $\wt S_n :=S_{h_{k_n}, \widetilde{K}_n}\subset \overline{\widetilde{A}_n}$. We claim that $\partial { \wt S_n } 
$ has at most $m_0$-connected 
components, so that $(\widetilde{A}_n, { \wt S_n }
) \in \B_\textbf{m}$. Indeed, if for every $j \in J$, $S_{h_\kn,K_\kn} \cap \Gamma_j$ is empty there is nothing to prove, so we assume that there exists $j \in J$ such that $S_{h_\kn,K_\kn} \cap \Gamma_j \neq \emptyset$. On one hand if $\Gamma_j \subset S_{h_\kn,K_\kn}$, thanks to the facts that $\Gamma_j$ is connected to $\partial A_\kn$ and $S_{h_\kn,K_\kn} \subset \overline{A_\kn}$, we deduce that $\Gamma_j$ needs to be connected to $\partial S_{h_\kn,K_\kn}$. On the other hand, if $\Gamma_j \cap (A_\kn \setminus S_{h_\kn}) \neq \emptyset$, then 
we can find $x_1 \in \Gamma_j \cap S_{h_\kn,K_\kn} $ and $x_2  \in \Gamma_j \cap (A_\kn \setminus S_{h_\kn})$. Since $\Gamma_j$ is closed and connected, by \cite[Lemma 3.12]{F} there exists a parametrization $r:[0,1]\to\mathbb{R}^2$ whose support $\gamma \subset \Gamma_j$ joins the point $x_1$ with $x_2$. Thus, $\gamma$ crosses $\partial S_{h_\kn,K_\kn}$ and we conclude that $\Gamma_j$ is connected to $\partial S_{h_\kn,K_\kn}$.

We claim that $(\widetilde{A}_n, {\wt S_n}
) \xrightarrow{\tau_\B} (A,S_{h,K})$ as $n \to \infty$. In view of \eqref{eq:boundpartialtildeA}, (a2) and the fact that by \eqref{eq:uniongraphs} and the previous construction of $\widetilde{K}_n$, 
$$
{ \sup_{n \in \N} \Hs^1(\partial \wt S_n) = }\sup_{n \in \N} \Hs^1(\partial S_{h_\kn, \widetilde{K}_n}) < \infty,
$$
it remains to prove that 
\begin{equation}\label{claim_convergence}
\sdist(\cdot, \partial S_{h_\kn, \widetilde{K}_n}) \to \sdist(\cdot, \partial S_{h,K})
\end{equation} 
locally uniformly in $\R^2$ as $n \to \infty$. Indeed, by Remark \ref{remark:1}-(i), it suffices to prove that $S_{h_\kn, \widetilde{K}_n} \cngK S_h$ and that $\Om \setminus S_{h_\kn, \widetilde{K}_n} \cngK \Om \setminus \Int(S_{h,K})$. On one hand, by the $\tau_\B$-convergence of $\{ (A_\kn, { S_{\kn}}
) \}$, the fact that ${\ov{\wt S_n} := } \overline{S_{h_\kn, \widetilde{K}_n}} = S_{h_\kn}$, and the properties of Kuratowski convergence, it follows that ${\wt S_n:=}S_{h_\kn, \widetilde{K}_n} \cngK S_h$. On the other hand, let $x \in \Om \setminus \Int(S_{h,K})$, since 
$$
\Int(S_{h_\kn, \widetilde{K}_n} ) = \Int(S_{h_\kn}) \setminus\widetilde{K}_n \subset \Int(S_{h_\kn}) \setminus K_\kn = \Int(S_{h_\kn, K_\kn} )
$$
and by the fact that $\Om\setminus \Int(S_{h_\kn, K_\kn}) \cngK \Om \setminus \Int(S_{h,K})$, there exists 
$$x_n \in \Om \setminus \Int(S_{h_\kn,K_\kn}) \subset \Om \setminus \Int(S_{h_\kn,\widetilde{K}_n}) $$
such that $x_n \to x$. Now, we consider a sequence  $x_n \in \Om \setminus \Int(S_{h_\kn,\widetilde{K}_n})$ converging to a point $x\in\Omega$. We proceed by contradiction, namely we assume that $x \in \Int(S_{h,K})$. Therefore, there exists $\epsilon >0$ such that $\sdist(x, \partial S_{h,K}) = - \epsilon$, which implies that $\sdist(x, \partial S_{h_\kn,K_\kn}) \to -\epsilon$ as $n \to \infty$. Thus, there exists $n_\epsilon \in \N$, such that $x_n \in B_{\epsilon/2}(x) \subset \Int(S_{h_\kn,K_\kn})$, for every $n \ge n_\epsilon$. However, notice that
\begin{equation}
\label{eq:compactnessinteriorS}
\begin{split}
x_n \in \Om \setminus \Int\left(S_{h_\kn, \widetilde{K}_n}\right) &= \Om \setminus \left( \Int({S_{h_\kn}}) \setminus \widetilde{K}_n \right) \\
& = \left( \Om \setminus \Int\left(S_{h_\kn,K_\kn}\right)\right) \cup  \left( \left( \partial \widetilde{A}_n \setminus \partial A_\kn \right) \cap{\overline{\Int(S_{h_\kn})}} \right),
\end{split}
\end{equation}
where in the last equality we used the definition of $\widetilde{K}_n := K_\kn \cup ((\partial \widetilde{A}_n \setminus \partial A_\kn) \cap { \overline{\Int(S_{h_\kn})}})$ and the fact that $\Int(S_{h_\kn,K_\kn}) = \Int(S_{h_\kn}) \setminus K_\kn$. Therefore, by \eqref{eq:compactnessinteriorS} we deduce that $x_n \in \partial \widetilde{A}_n \setminus \partial A_\kn$ for every $n \ge n_\epsilon$ and hence, $x \in \partial A$ by (a2) and Remark \ref{remark:1}-(i). We reached an absurd as it follows that $x\in \Int (S_{h,K})\cap\partial A=\emptyset$. This conclude the proof of \eqref{claim_convergence} and hence, of the claim.

By \eqref{eq:H1} and by conditions (a1), (a4) and \eqref{fromproof3.6}, we observe that
\begin{equation}
\label{eq:compactnessIntS1}
\lim_{n \to \infty} {\abs{\mathcal{S}(A_\kn, S_\kn) - \mathcal{S}(\widetilde{A}_n,\wt S_n) }} = \lim_{n \to \infty} {\abs{\mathcal{S}(A_\kn, S_{h_\kn,K_\kn}) - \mathcal{S}(\widetilde{A}_n,S_{h_\kn,\widetilde{K}_n}) }} = 0,
\end{equation}
and
\begin{equation}
\label{eq:compactnessIntW}
\mathcal{W}(A_\kn, u_\kn) = \mathcal{W}(\widetilde{A}_n,u_\kn). 
\end{equation}

By \eqref{eq:H3}, \eqref{lim_bound},  \eqref{fromproof3.6}, \eqref{eq:compactnessIntW}, (a3) and thanks to the fact that $\mathcal{S}$ is non-negative, we obtain that
\begin{equation*}		\int_{E_i^n}{\abs{e(u_\kn)}^2 dx} \le \int_{\widetilde{A}_n}{\abs{e(u_\kn)}^2 dx} \le \frac{R}{2c_3},
\end{equation*}
for every $i\in I$ and for $n$ large enough. Therefore, by a diagonal argument and by \cite[Corollary 3.8]{KP} (applied to, with the notation of \cite{KP}, $ P = E_i$ and $P_n = E_i^n$) up to extracting not relabelled subsequences both for  $\{u_{\kn}\} \subset H^1_{\mathrm{ loc}}(\Om; \R^2)$ and $\{E_i^{n}\}_n $ there exist $w_i \in H^1_\mathrm{{loc}} (E_i, \R^2)$,  and a sequence of rigid displacements $\{b_n^i\}$ such that $(u_\kn + b_n^i) \mathbbm{1}_{E^{n}_i} \rightarrow w_i$ a.e. in $E_i$. Let $\{D^n_i\}_{i \in \widetilde{I}}$ for an index set $\widetilde{I}$ be the family of open and connected components of $\widetilde{A}_n \setminus \bigcup_{i\in I}{E_i^n}$ such that by (a3) $\Int(D^n_i)$ converges to the empty set for every $i \in \widetilde{I}$. In $D_i^n$ we consider the null rigid displacement, and we define
\begin{equation*}
{b}_n := \sum_{i \in I}{b_n^i} \mathbbm{1}_{E_i^n}  \quad \text{and} \quad u := \sum_{i \in I}{w_i \mathbbm{1}_{E_i}}.
\end{equation*}
We have that $u \in H^1_\mathrm{{loc}}(\mathrm{{Int}} (A); \R^2 )$, ${b}_n$ is a rigid displacement associated to $\widetilde{A}_n$,  $u_\kn + {b}_n \rightarrow u$ a.e. in $\mathrm{{Int}}(A)$ and hence,  $ { (A,S, u) = }(A,S_{h,K}, u) \in \mathcal{C}_\mathbf{m}$ and ${(\wt A_n, \wt S_n,u_\kn + {b}_n ) := } (\widetilde{A}_n,S_{{h}_\kn,\widetilde{K}_n},  u_\kn + {b}_n) \xrightarrow{\tau_{\mathcal{C}}} (A,S_{h,K}, u)$.
Furthermore, as $e(u_\kn + b_n) = e(u_\kn)$, from 
\eqref{eq:compactnessIntS1} and \eqref{eq:compactnessIntW} it follows that
\begin{equation*}
\label{eq:compactnessIntS3}
\begin{split}
&{ \lim_{n \to \infty}  {\abs{\mathcal{F}(A_\kn,S_{h_{k_n},K_{k_n}}
,u_\kn) - \mathcal{F}(\widetilde{A}_n,\wt S_n, u_\kn + b_n) }}  } \\
&=\lim_{n \to \infty} {\abs{\mathcal{F}(A_\kn,S_{h_\kn,K_\kn},u_\kn) - \mathcal{F}(\widetilde{A}_n,S_{h_\kn,\widetilde{K}_n}, u_\kn + b_n) }} = 0, 
\end{split}
\end{equation*}
which implies \eqref{liminfequalliminf} and completes
the proof.
\end{proof}


\section{Lower semicontinuity}
\label{sec:semicontinuity}

In this section we prove that for any fixed $\mathbf{m} :=(m_0, m_1)\in\N \times \N$ the energy $\mathcal{F}$ is lower semicontinuous in the family of configurations $\mathcal{C}_\mathbf{m}$  with respect to the topology $\tau_{\mathcal{C}}$.  Since $\mathcal{F}$ is given as the sum of the surface energy $\mathcal{S}$ and the elastic energy $\mathcal{W}$, we proceed by proving that both $\mathcal{S}$ and $\mathcal{W}$ are independently lower semicontinuous with respect to $\tau_{\mathcal{C}}$.

We begin with $\mathcal{S}$ and we adopt   {\itshape Fonseca-M\"uller blow-up technique}  \cite{FM}, for which we make use of a localized version ${\mathcal{S}}_L$ of the surface energy, which we can consider with surface tensions constant with respect to the variable in $\overline{\Omega}$. 

\begin{definition}
	Let $\phi_{\text{F}}, \phi_{\text{S}}, \phi_{\text{FS}}$ be three functions and let $\phi := \min\{\phi_{\mathrm{S}}, \phi_{\mathrm{F}} + \phi_{\mathrm{FS}} \},\,   \phi' := \min\{\phi_{\mathrm{F}}, \phi_{\mathrm{S}}\} $ be such that $\phi_{\text{F}},\, \phi_{\text{FS}},\,  \phi\,,   \phi' \in C(\R^2; [0, \infty])$ are Finsler norms and the hypotheses (H1) and (H2) are satisfied by the functions $\varphi_{\alpha}{, \varphi} {, \varphi'}\in  C (\overline{\Omega}\times\R^2)$  given for  $\alpha=\text{S},\text{FS},\text{F}$ by $\varphi_{\alpha}(x,\cdot):=\phi_{\alpha}(\cdot)$, 
 $\varphi(x,\cdot):=\phi(\cdot)$ and ${ \varphi'(x, \cdot) = \phi'(\cdot)} $ for every $x\in\overline{\Omega}$. 
We define the localized surface energy ${\mathcal{S}}_L : { \B_L} \to [0,+\infty]$  by 
\begin{equation}
\label{eq:definitionlocenergyS}
\begin{split}
{\mathcal{S}}_L \left(A ,S
, O \right) 
&:= \int_{ O \cap (\partial^* A \setminus \partial{S})}{\phi_{\mathrm{ F}}(\nu_A)\, d\Hs^1}  + \int_{O \cap   \partial ^*{S} \cap \partial^* A }{\phi (\nu_A )\, d\Hs^1}  \\
& \quad + \int_{O \cap (\partial ^*{S} \setminus \partial A)\cap A^{(1)}}{ \phi_{\mathrm{ FS}} ( \nu_{S}   )\, d\Hs^1} + \int_{O \cap  \partial ^*{S_{h,K}} \cap \partial A  \cap A^{(1)}}{ (\phi_{\mathrm{ F}} + \phi )( \nu_A   ) \, d\Hs^1} \\
& \quad + \int_{O \cap \partial A  \cap  A^{(1)} \cap S^{(0)}}{2 \phi_{\mathrm{ F}}(\nu_A )\, d\Hs^1} { + \int_{O \cap \partial A \cap A^{(0)}}{2 \phi'(\nu_A )\, d\Hs^1} }\\
& \quad    + \int_{O \cap (\partial {S} \setminus \partial A) \cap \left(S^{(1)} \cup S^{(0)}\right)  \cap A^{(1)} }{ 2 \phi_{\mathrm{ FS}} ( \nu_{S}  ) \, d\Hs^1}   + \int_{O \cap \partial {S} \cap \partial A \cap S^{(1)}  }{  2{ \phi }  
( \nu_{S}   ) \, d\Hs^1}
\end{split}
\end{equation} 
for every $(A,S,O) \in \B_L:=\{(A,S,O)\,:\, (A,S) \in \B, \text{ $O$ open and contained in $\Om$}\}$.
\end{definition}
We start with  some technical results needed in the blow-up argument used in Theorem \ref{thm:lowersemicontinuityS}.

	\begin{lemma}
	\label{lem:A2}
	Let $Q$ be any open square, 
 $K \subset \overline{Q}$ be a nonempty closed set and $E_k \subset \overline{Q}$ be such that $\sdist(\cdot, \partial E_k) \to \dist(\cdot, K)$ uniformly in $\overline{Q}$ as $k \to \infty$. Then $E_k \cngK K$ as $k \to \infty$. Analogously, if $\sdist(\cdot, \partial E_k)\to -\dist(\cdot, K)$ uniformly in $\overline{Q}$ as $k \to \infty$, then $\overline{Q} \setminus E_k \cngK K$ as $k \to \infty$.
	\end{lemma}

 The proof of the previous lemma follows from the same arguments of \cite[Lemma 4.2]{KP}. 

	
	\begin{proposition}
		\label{prop:aux1}
Let $m\in\N$ and let $E \subset \R^2$ be a set such that $\partial E$ has at most $m$-connected components, is $\Hs^1$-rectifiable, and satisfies $\Hs^1 (\partial E) < \infty$. Let $x \in \partial E$ be such that the measure-theoretic unit normal $\nu_E \xp$ of $\partial E$ at $x$ and there exists $R>0$ such that $\overline{Q_{R, \nu_E \xp} \xp} \cap  \partial {\sigma}_{\rho, x} ( E ) \cngK \overline{Q_{R, \nu_E \xp} \xp} \cap T_x$ 
  as $\rho \to 0^+$, where $T_x:= T_{x,\nu_E(x)}$. Then, the following assertions hold true:
		\begin{enumerate}
			\item[(a)] If $x \in E^{(1)} \cap \partial E$, then $\sdist (\cdot, \partial {\sigma}_{\rho, x} (   E)) \to - \dist(\cdot, T_x)$ uniformly in $\overline{Q_{R, \nu_E \xp}}$ as $\rho \to 0^+$;
			\item[(b)] If $x \in E^{(0)} \cap \partial E$, then $\sdist (\cdot,  \partial {\sigma}_{\rho, x} (  E)) \to  \dist(\cdot, T_x)$ uniformly in $\overline{Q_{R, \nu_E \xp}}$ as $\rho \to 0^+$;
			\item[(c)] If $x \in \partial ^* E$ then $\sdist(\cdot , \partial {\sigma}_{\rho, x} (  E )) \rightarrow \sdist( \cdot, \partial H_x)$  uniformly in $\overline{Q_{R, \nu_E \xp }}$ as $\rho \to 0^+$.
		\end{enumerate}
	\end{proposition}
	
	\begin{proof}
		The cases (a) and (b) follow directly from \cite[Proposition A.5]{KP}. It remains to prove the case (c) to which the remaining of the proof is devoted.  Let $x \in \partial^* E$ and $R> 0$ be such that 
	\begin{equation}
		\label{eq:c1}
		\overline{Q_{R, \nu_E \xp} \xp} \cap \partial {\sigma}_{\rho, x} ( E ) \cngK \overline{Q_{R, \nu_E \xp} \xp} \cap T_x  \text{ as } \rho \rightarrow 0^+.
	\end{equation}
	Without loss of generality, we assume that $x=0, \nu_E (0) = \mathbf{e}_\mathbf{2}$, $H_x = H_0$ and $T_x =T_0=\partial H_0$. 
 
 Let $\{\rho_k\} \subset (0,1)$ be such that $\rho_k \to 0$ and let $f_k := \sdist(\cdot, \sigma_{\rho_k} (\partial E ) ) \big|_{Q_R} \in W^{1, \infty} (Q_R)$. We see that for any $k>0$, $f_k$ is $1$-Lipschitz continuous, moreover by the fact that $f_k (0)= \sdist(0,  \sigma_{\rho_k} (\partial E )) = 0$, we deduce that $\{f_k\}$ is uniformly bounded. By applying Ascoli-Arzelà Theorem, there exists $f \in W^{1, \infty} (Q_R)$ and a non-relabeled subsequence $\lk f_k \rk$ such that $f_k \to f$ uniformly in $Q_R$. 
	In view of \cite[Proposition A.1]{KP}, by \eqref{eq:c1} we obtain that  $\abs{f_k} = \dist(\cdot ,  \sigma_{\rho_k} (\partial E )) \to \dist(\cdot, T_0)$ uniformly in $Q_R$ and thus, $\abs{f(x)} = \dist(x, \partial H_0)$ for any $x \in Q_R$.
	
	It remains to prove that $f(\cdot) = \sdist(\cdot, \partial H_0)$ in $Q_R$. We proceed by absurd. Assume by contradiction that $f \neq \sdist( \cdot, \partial H_0)$ then, either $f \equiv \sdist(\cdot, \partial (Q_R \setminus H_0))$ or $f \equiv \dist(x, T_0) $ or $f \equiv- \dist(x, T_0)$. Let us first consider the case in which $f \equiv \sdist(\cdot, \partial (Q_R \setminus H_0))$. In view of Remark \ref{remark:1}-(i), it follows that $\sigma_{\rho_k}(E) \cngK Q_R \setminus \Int(H_0)$ and so, as a consequence we have that $\mathcal{L}^2 \left(\sigma_{\rho_k}(E) \triangle (Q_R \setminus  \Int(H_0)) \right) \to 0$ as $k \to \infty$, which is in contradiction with the De Giorgi's structure theorem of sets of finite perimeter (see \cite[Theorem 5.13]{EG} or \cite[Theorem 15.5]{M}).
	
 In the case $f(\cdot) \equiv \dist(\cdot, T_0)$, thanks to the fact that $f_k \to f$ uniformly in $\overline{Q_R}$ and by Lemma \ref{lem:A2},  
 we have that $\sigma_{\rho_k} (E)\cap \overline{Q_R} \cngK T_0\cap \overline{Q_R}$, and thus, $\mathbbm{1}_{\sigma_{\rho_k}(E)} \to \mathbbm{1}_{T_0}$ in $L^1(\overline{Q_R})$, which is a contradiction with the fact that  $\mathbbm{1}_{\sigma_{\rho_k}(E)} \to \mathbbm{1}_{H_0}$ in $L^1_\mathrm{ loc}(\R^2)$ by \cite[Theorem 5.13]{EG}. 
	
	In the last case in which $f(\cdot) \equiv -\dist(\cdot, T_0)$, we proceed analogously, and by Lemma \ref{lem:A2} 
	we obtain  that $\overline{Q_R} \setminus \sigma_{\rho_k}(E) \cngK T_0 $ and hence, $\mathbbm{1}_{(\overline{Q_1} \setminus \sigma_{\rho_k}(E))} \to \mathbbm{1}_{T_0}$ in $ L^1(\overline{Q_R})$. Therefore, we reach a contradiction again by applying \cite[Theorem 5.13]{EG} since $x\in\partial^*(\R^2\setminus E)$ and so, $\mathbbm{1}_{(\overline{Q_R} \setminus\sigma_{\rho_k}(E))} \to \mathbbm{1}_{{\R^2\setminus \Int(H_0)}}$ in $L^1_\mathrm{ loc}(\R^2)$.
\end{proof}

We now introduce the notions  of \emph{film free boundary}, \emph{substrate free boundary},  and \emph{film-substrate adhesion interface} for triples $(A,h,K)\in\mathcal{B}$, and of \emph{triple junctions} at the points where they ``meet''.

  \begin{definition} \normalfont
      \label{def:notationislands}
    For any admissible pair $(A,S
    ) \in{ \wt \B}$ we denote:
      \begin{itemize}
          \item the \emph{film free boundary}, the \emph{substrate free} boundary and the \emph{film-substrate adhesion interface} by
          \begin{align*}
              \Gamma_\mathrm{ F}(A,S
              )&:=\left((\partial A\setminus \partial S) \cup (\partial S \cap \partial A  \cap A^{(1)}) {  \cup (\partial {S} \cap \partial ^* A \cap S^{(0)})} 
                           \right) \cap { \Om}, \\
              \Gamma_\mathrm{ S}(A,S
              )& :=\left( (\partial S \cap \partial A) \setminus { (\partial {S} \cap \partial ^* A \cap S^{(0)})}  \right) \cap { \Om},\\
              \Gamma^{\text{A}}_\mathrm{ FS}(A,S
              )&:= \left( \left(( \partial S \setminus \partial A) \cap A^{(1)} \right) 
 \right) \cap { \Om},
          \end{align*}
          respectively. Notice that the \emph{film-substrate delamination interface},  that we define by 
          $$
            \Gamma^{\text{D}}_\mathrm{ FS}(A,S
            ) := \left((\partial S \cap \partial A ) \cap A^{(1)}\right) \cap { \Om},
          $$
          is contained both in $\Gamma_\mathrm{ F}(A,S
              )$ and in $\Gamma_\mathrm{ S}(A,S
              )$.
    \item \emph{triple junction} (by including for simplicity also the ``double'' junctions at the boundary) any point  
    \begin{align*}
         p\in &\left(Cl(\Gamma_\mathrm{ F}(A,S
              ))\cap Cl(\Gamma_\mathrm{ S}(A,S
              ))\cap Cl(\Gamma^{\text{A}}_\mathrm{ FS}(A,S
              ))\cap {\Om}\right)\\  
         &\cup\left(Cl(\Gamma_\mathrm{ F}(A,S
              ))\cap Cl(\Gamma_\mathrm{ S}(A,S
              ))\cap \partial \Int(\ov{S})\cap \partial { \Om})\right)\\
         &\cup\left(Cl(\Gamma_\mathrm{ F}(A,S
              ))\cap Cl(\Gamma_\mathrm{ FS}^A(A,S
              ))\cap \partial { \Om})\right),
\end{align*}
        where the closures are considered with respect to the relative topology of $\partial A\cup \partial S$. 
      \end{itemize}
  \end{definition}

The next result allows us to assume that the adhesion interface of any admissible pair $(A,S
) \in  \B_\textbf{m}$ (without the substrate internal cracks) can be considered, up to an error and up to passing to the family $\wt \B$, to be given by a finite number (depending on the initial pair $(A,S)$) of connected  components.  

	\begin{lemma}		
 \label{lem:finiteness}
  Let $R$ be an open rectangle with two sides, that are denoted by $T_1$ and $T_2$, perpendicular to $\textbf{e}_\textbf{1}$. Let $(A,S_{h,K}
  ) \in  \B_\textbf{m}$ for  $(h,K) \in \mathrm{AHK}$ be such that 
  ${\mathcal{S}}_L(A,S_{h,K}
  ,R) < \infty$, where ${\mathcal{S}}_L$ is the localized surface energy defined in 
  \eqref{eq:definitionlocenergyS}. If $\Hs^1\left( \left( \Gamma^\mathrm{{A}}_\mathrm{ FS}\left(A,S_{h,K}\right) \setminus \Int(S_h)\right)  \cap R \right) > 0$, for every $\eta \in(0,1)$ small enough there exist $M:= M(A,S_{h,K},\eta) \in \N \cup \{0\}$  and $(\widetilde{A} , \widetilde{S}) \in \wt \B $ such 
that $(\Gamma^\mathrm{{A}}_\mathrm{ FS}(\widetilde{A}, \widetilde{S}  ) \setminus \Int(\overline{\wt{S}}))\cap R$ has at most $M$ connected components and 
		\begin{equation}
			\label{eqlem:finiteness1}
			{\mathcal{S}}_L(A,S_{h,K},R) \ge {\mathcal{S}}_L (\widetilde{A}, \widetilde{S},R ) - \eta.
		\end{equation} 
	Furthermore, $(\widetilde{A} , \widetilde{S})$ satisfies the following properties:
 \begin{itemize}
     \item[(i)] If  $\overline{\partial S_{h}\cap R}\cap T_\ell\neq\emptyset$ for $\ell=1,2$, then also
$\overline{\partial \widetilde S\cap R}\cap T_\ell\neq\emptyset$ for $\ell=1,2$;
  \item[(ii)] If there exists a closed connected set $\Lambda\subset\overline{\partial A\cap R}$ such that $\Lambda\cap T_\ell\neq\emptyset$ for $\ell=1,2$,  then there exists a curve with support $\wt\Lambda\subset\overline{\partial \wt{A} \cap R}$ such that $\wt\Lambda\cap T_\ell\neq\emptyset$ for $\ell=1,2$.
 \end{itemize}
  	\end{lemma}
	
	\begin{proof}
 Notice that we cannot a priori exclude that  $(\Gamma^{\text{A}}_\mathrm{ FS}(A,S_{h,K})\setminus \Int(S_h)) \cap  R$ is a totally disconnected set with positive $\Hs^1$-measure (see, for instance, the {\itshape Smith-Volterra-Cantor set} in \cite[Chapter 3]{R}). 
  We denote by $L(h)$ the set of substrate filaments of the substrate $ S_{h,K}$, namely, 
  \begin{equation}
\label{eq:definitionLh}
    L(h):= \{(x_1,x_2) \in \overline{\Om}: x_1 \in  (-l,l) \text{ and } h^+(x_1)<x_2 \le h(x_1) \} \subset \partial S_h.
  \end{equation}
  Since $h$ is upper semicontinuous, there exist an index set $J_1$ and a countable family of disjoint points $\{x_1^j\}_{j \in J_1} \subset  (-l,l) $ such that 
  \begin{equation}
      \label{eq:definitionLhj}
      L(h) = \bigcup_{j\in J_1} L_j(h),
  \end{equation}
where $L_j (h) := \{(x^j_1,x_2) \in \overline \Om:  h^+(x^j_1)<x_2 \le h(x^j_1) \}$ for 
every $j \in J_1$. In the following three
 steps, in view of the  outer regularity of Borel measures, we construct an admissible 
pair $(\widetilde{A}, \widetilde{S}) \subset \wt \B$ by modifying some portions of $\partial S_{h}$ and $\partial A$. More 
precisely, in the first step, we construct an 
admissible height $h^1$  by eliminating  a family of  ``small'' filaments of $L(h)$ so that $L(h^1)$ consists of only a finite number of filaments, we 
accordingly modify $A$ in an admissible region $A^1$ containing $L(h)\setminus L(h^1)$, and we define $K^1:=K$. 
In the second step,  we construct $S^2$ by modifying $S_{h^1,K^1}$ 
and we introduce an admissible region  $A^2$ in such a way that $(\Gamma^\text{A}_\text{FS}(A^2, S^2) \setminus \Int(\ov{S^2})) \cap R$ is a countable  union of connected components, and (i) and (ii) hold true.
In the third step, 
by eliminating  some components of $\Gamma^\text{A}_\text{FS}(A^2, S^2) \setminus \Int(\ov{S^2})) \cap R$ we define an admissible pair $(\wt A,\wt S) \in \wt \B$
for which  
$(\Gamma^{\text A}_\mathrm{ FS}(\wt A,\wt S) \setminus \Int(\ov{\wt S})) \cap R$  has at most $M$-connected components, (i) and (ii) are preserved, and \eqref{eqlem:finiteness1} holds true.

    	
{\itshape Step 1  (Modification of substrate filaments).} 
We modify $(A,S_{h,K})$ in $( A^0,S_{h^1,K^1}) $ to have a finite number of substrate filaments. 
 We denote by $J_2\subset J_1$ the set of indexes $j \in J_1$ such that $\Hs^1( L_j \setminus \Int(A)) =0 $  and $L_j$ is not connected to $\partial A$, and we denote by 
 $F_{J_2} (h)$ the set of the $x_{1}$-coordinates corresponding to the points in each vertical segment $L_j$ for  $j\in J_2$, i.e., $F_{J_2} (h):=\{x^j_{1}\}_{j\in J_2}$, where $x^j_1\in {[-l,l]}$ is such that $(x^j_1,h(x^j_1)) \in{L_j(h)}$ and  $h^+(x^j_1) < h(x^j_1)$ for every $j\in J_2$. 
 We define as ${h^0}$ the modification of $h$, given by 
	\begin{equation*}
	    \begin{matrix}
	            h^0: & {[-l,l]} & \to & {[0,L]}&\\
	            & x & \mapsto & h^0(x)& := \begin{cases}
	                h(x) & \text{if } x \in {[-l,l]} \setminus F_{J_2} (h), \\
	                h^+(x) & \text{if } x \in F_{J_2}(h),
	            \end{cases}
	    \end{matrix}
	\end{equation*}
	and observe that by construction $h^0 \in \text{AH} $ and 
	\begin{equation}
	    \label{eq:modvertical1}
	    \Hs^1(\partial S_{h^0}) \le \Hs^1(\partial S_{h}),
	\end{equation}
	where $S_{h^0}$ is defined as in \eqref{eq:uppergraphS}.  We notice that the triple $(A, S_{h^0, K^0} ) \in \B_\textbf{m}$, where $K^0:= K $.  
    As a consequence of the construction and of the non-negativity of $
    \phi_{\text{FS}} $, it follows that
	\begin{equation}
	    \label{eq:modvertical2}
	    {\mathcal{S}}_L (A,S_{h,K},R) \ge {\mathcal{S}}_L (A,S_{h^0,K^0},R) .
 	\end{equation}
  We notice that by \eqref{eq:definitionLhj}  we have that 
  $$
  \Hs^1 (L(h^0)
  )= \Hs^1 (L(h)
  ) - \sum_{j\in J_2} \Hs^1 (L_j(h))= \sum_{j\in J_3} \Hs^1 (L_j(h)),
  $$
where  $J_3:= J_1 \setminus J_2$ and $L(h^0)$ is the set of substrate filaments of $S_{h^0,K^0}$ defined as in \eqref{eq:definitionLh}. Therefore, for a fix $\widetilde \eta>0$ there is $j'_1 := j'_1(\widetilde \eta) \in J_3$ such that 
	\begin{equation}
	\label{eq:modvertical3}
	   \sum_{j  = j'_1 +1 , j \in J_3}^{\infty} \Hs^1 (L_j(h)) \le \widetilde \eta,
	\end{equation}
	and we define $$ A^1 := \left(A \setminus \left( \bigcup_{j= j'_1 + 1, j \in J_3}^{\infty} (L_j(h) \cap \Int(A)) \right)  \right) \cup  \bigcup_{j= j'_1 + 1, j \in J_3}^{\infty} (L_j(h) \cap (\Om \setminus \overline{\Int(A)}) ).$$
 Furthermore, we denote by $ h^1$ the modification of $h^0$, defined by
	\begin{equation*}
	    \begin{matrix}
	            h^1: & {[-l,l]} & \to & {[0,L]}&\\
	            & x & \mapsto & h^1(x)& := \begin{cases}
	                h^0(x) & \text{if } x \in {[-l,l]}\setminus F_{J_4} (h^0), \\
	                {h^0}^+(x) & \text{if } x \in F_{J_4}(h^0),
	            \end{cases}
	    \end{matrix}
	\end{equation*}
	where $J_4 := \{j \in J_3: j \ge j'_1 +1 \}$ and $F_{J_4}(h^0):=\{x^j_{1}\}_{j\in J_4}$ such that 
 $(x^j_1,h(x^j_1)) \in {L_j(h)}$ 
 for every $j\in J_4$.
 We define $K^1:= K^0$, we notice that $(h^1,K^1) \in \text{AHK}$ and $S_{h^1,K^1} \subset \overline{A^1}$, thus $(A^1, S_{h^1,K^1}) \in \B_\textbf{m}$ and since $L_j(h)$ is connected to $\partial A$ for every $j \in J_4\subset J_3$,  
    we deduce that $\partial A^1$ has at most $m_1$-connected components.     
    Finally, we observe that
    \begin{equation}
	    \label{eq:modvertical5}
	    {\mathcal{S}}_L (A,S_{h,K},R) - {\mathcal{S}}_L ( A^1,S_{h^1,K^1},R) 
     \ge -\sum_{j \in J_4} 2 \int_{L_j(h)}{\phi_{\text F}(\nu_{L_j (h)}) d \Hs^1 } 
     \ge  - 2 c_2 \widetilde{\eta}, 
	\end{equation}
    where we used the non-negativeness of $\phi$, $\phi_{\text FS}$ and \eqref{eq:H1}, and we observe that $( A^1,S_{h^1,K^1})  \in \B_\textbf{m}$ has a finite number of  substrate filaments, more precisely, we denote by $j' \in \N$ the cardinality of the  index set $J:= J_3 \setminus J_4 = \{j \in J_3: j \le j'_1\}$ and we have that $L(h^1)$ is the union of $j'$ filaments, i.e.,
 \begin{equation}\label{eqLs1}
 L(h^1) = \bigcup_{j \in J} L_j(h) ,
\end{equation}
where $L(h^1)$ is defined as in \eqref{eq:definitionLh} with respect to the substrate $S_{h^1,K^1}  \in \text{AS}(\Om)$.

{\itshape Step 2 (Modification of the substrate free boundary).} Without loss of generality in the following we assume that  $
\partial \Int(S_{h^1}) \cap \overline{L( h^1)} \subset R$. 
 Since $\Hs^1\mres\partial {S_{h^1,K^1}}$ is a finite Borel measure and 
 $$\Gamma^\text{A} := \left( \Gamma^{\text A}_\mathrm{ FS}(A^1,S_{h^1,K^1}) \setminus \Int(S_{h^1}) \right) \cap R$$ is a Borel set, by the outer regularity of measures (see \cite[Theorem 2.10]{M}), 
 there exists an open set $O=O(\widetilde{\eta})\subset R$ such that $ \Gamma^{\text A}
 \subset O\cap \partial {S_{h^1,K^1}}$ and
		\begin{equation}
			\label{eq:finiteness2}
			\begin{split}
			    \Hs^1 \left(\wh \Lambda \right) &= \Hs^1 \left( (O \cap \partial^* {S_{h^1,K^1}} )\setminus \Gamma^{\text A}
       \right)
                \le \Hs^1 \left( (O \cap \partial {S_{h^1,K^1}}) \setminus \Gamma^{\text A}
                \right) < \frac{ 2^{-5/2}}{j'+1} \widetilde{\eta},
			\end{split}  
		\end{equation}
	where $\wh \Lambda := (O \cap \partial \Int(\overline{S_{h^1,K^1}}))\setminus \Gamma^{\text A}
 $ and $j'$ is defined in the Step 1 as the number of filaments of $S_{h^1}$. Moreover, by using the notation introduced in \eqref{eq:uppergraphS} and the fact that $h^1\in \text{AH}(\Om)$ we conclude that
    \begin{equation}
    \label{eq:pathSh+}
    	\partial \Int\left(\overline{S_{h^1,K^1}}\right)=\partial \Int(S_{h^1})=\partial S_{{h^1}^+},
    \end{equation} 
    and hence, since $\partial \mathrm{ \Int}(\overline{S_{h^1,K^1}})$ is a connected and compact set in $\mathbb{R}^2$ with finite  $\mathcal{H}^1$-measure,  by \cite[Lemma 3.12]{F} there exists a parametrization $r:[0,1]\to\mathbb{R}^2$ of $\partial \mathrm{ \Int}(\overline{S_{h^1,K^1}})$  whose support $\gamma$ joins the points $(-l,{h^1}(-l^+))$ with $(l,{h^1}(l^-))$.

Notice that by Step 1, $
\partial \Int(S_{h^1}) \cap \overline{L( h^1)} $ is the union of $j'$-points  that we can order by labeling them with $p_1,\dots,p_{j'}$. Furthermore, 
 we denote with $p_0:=r(t_0)$ and $p_{j'+1}:= r(t_1)$, where $t_0 := \inf\{t \in [0,1]: r(t) \in \partial R\}$ and $t_1 := \sup\{t \in [0,1]: r(t) \in \partial R\}$, and we consider the family  $\{R^i\}_{i =1}^{j' +1}$ of the strips $R^i$ defined as  the open regions of $R$ contained between the vertical lines passing though the points $p_{i-1}$ and $p_i$. 

 Since $\wh \Lambda$ is a Borel measurable set, by \eqref{eq:equivalencE_hausdorffnetmeasure} and \eqref{eq:finiteness2} 
 we have that
	\begin{equation}
		\label{eq:finiteness3}
			\mathcal{N}^1(\wh\Lambda)\leq 2^{\frac52} \mathcal{H}^1(\wh\Lambda) < \frac {\widetilde{\eta}}{j'+1} ,
	\end{equation}
 where $\mathcal{N}^1$ is the net measure defined in \eqref{eq:netmeasure1}.
 Therefore,  there exists $\delta>0$ such that we can find a family of disjoint dyadic squares $\{U_n\}_{n \in \N} \subset \mathcal{Q}$  such that $\wh\Lambda \subset \bigcup_{n\in \N} U_n $, $\mathrm{ diam} (U_n) \le \delta$ for any ${n\in \N}$ and 
		\begin{equation}
			\label{eq:finiteness4}
			\sum_{n\in \N}{\mathrm{ diam} (U_n) } \le \mathcal{N}^1(\wh\Lambda) + \frac {\widetilde{\eta}} {j'+1} <\frac 2{j'+1} \widetilde{\eta}. 
		\end{equation}
Without loss of generality we assume that ($\wh \Lambda \neq \emptyset$ and that) $U_n \cap \wh \Lambda$ is non-empty for every $n \in \N$.
Let $\{U_{n}^i\} \subset \{U_n\}$ be the subfamily of dyadic squares such that  $U_{n}^i \cap R^i \cap \gamma \neq \emptyset$ for every $n \in \N$. Furthermore, we assume for simplicity that $\Int(U_{n}^i)\subset R^i$. 
We begin by modifying the pair $(A^1,S_{h^1,K^1})$ in the strip $R^1$, by denoting the modification by $(A^2_1,S^2_1)$. We characterize $A^2_1$ and $S^2_1$  as
\begin{equation}
 \label{eq:A21}
	A^2_1 : = \left(A^1 \cup \bigcup_{n\in \N}{U_{n}^1} \right) \setminus \bigcup_{n\in \N} ( \partial U_{n}^1 \setminus {U_{n}^1}) 
 \end{equation}
 and
\begin{equation}
 \label{eq:S21} S^2_1 : = \left(S_{h^1,K^1} \setminus \bigcup_{n \in \N} \Int(U_n^1)\right) \cup \bigcup_{n \in \N} (\partial U^1_n \cap (R^1 \setminus S_{h^1}) ),
 \end{equation}
respectively. 
By construction, it follows that $S^2_1 \subset \ov{A^2_1}$, $\partial A^2_1$ and  $\partial S^1_1$  have finite $\Hs^1$ measure and are $\Hs^1$-rectifiable, and $ \partial A^2_1 \cap \Int(\ov{S^2_1}) = \emptyset$ and hence,   
$(A^2_1,  S^2_1) \in \wt \B$. Furthermore, we have that
\begin{equation}
    \label{eq:finitenessfinal0bis}
    \begin{split}
       &{\mathcal{S}}_L\left(A^1,S_{h^1,K^1}, R^1  \right)  - {\mathcal{S}}_L\left(A_1^{2} , S_1^{2}, R^1  \right)\\
       & \ge - 2 \int_{\bigcup_{n \in \N} \partial U^1_n  }\phi_{\text{F}}   (\nu_{U^1_n}) + \phi  (\nu_{U^1_n})  + \phi_{\text{FS}}   (\nu_{U^1_n}) 
       \, d\Hs^1 \\
       &  \ge - \sum_{n\in\N} 2 \int_{ \partial U^1_n  }\phi_{\text{F}}   (\nu_{U^1_n}) + \phi  (\nu_{U^1_n})  + \phi_{\text{FS}}   (\nu_{U^1_n}) 
       \, d\Hs^1\\
       & \ge - \sum_{n\in\N} 24 \, c_2 \text{diam}(U^1_n) \ge -  \frac {48}{j'+1} c_2\widetilde{\eta},
    \end{split}
\end{equation}
where in the first inequality we used the non-negativeness of $\phi_{\mathrm{F}}, \phi_{\mathrm{FS}}$ and $\phi$, in the second inequality we used the subadditivity of measures, in the third inequality we used \eqref{eq:H1}, and in the last inequality we used \eqref{eq:finiteness4}.  
We notice that $\Gamma^{\text{A}}_{\text{FS}}(A^{2}_1, S^2_1) \cap \partial \Int(\ov{S^2_1}) \cap R^1 $ is a countable union of connected sets because by construction every connected component of $\Gamma^{\text{A}}_{\text{FS}}(A^{2}_1, S^2_1) \cap \partial \Int(\ov{S^2_1}) \cap R^1 $ is connected to an element in the family of sets $\{ \partial U^1_n \cap U^1_n \}_{n \in \N}$.

Now, we modify $(A_1^{2} , S_1^{2})$ in a new configuration $(A^3_1, S^3_1)$ in order to prove Assertion (ii). To this end let $\Lambda\subset\overline{\partial A\cap R}$ be a closed connected set such that $\Lambda\cap T_\ell\neq\emptyset$ for $\ell=1,2$. 
By \cite[Lemma 3.12]{F} there exists a parametrization $r_1 : [0,1] \to \R^2$ whose support $\gamma^1 \subset \Lambda$ joins $T_1$ with $T_2$. We define $\gamma^1_1: = \gamma^1 \cap \ov{R^1}$ and we observe that $(\gamma^1_1 \setminus \Int(U^1_n)) \cup \partial U^1_n$ is a connected set. Let $Z_1 \subset \N$ be the set of indexes $n$ such that  $\gamma_1 \cap (\partial U^1_n\cap U^1_n) 
\neq \emptyset$.
If $Z_1 = \emptyset$, then we define $A^3_1 := A^2_1$
and $S^3_1:= S^2_1$. 
If $Z_1\neq \emptyset$, then we modify $\gamma_1^1$ in 
$\bigcup_{n\in Z_1} U^1_n $ by defining a new set $\Lambda_1$. More precisely, 
by using the fact that  dyadic squares by definition do not intersect each other, we fix $n\in Z_1$ and we  replace with a set $\Lambda^1_n$ (see \eqref{eq:lambda1} below) the portion of $\gamma^1$ passing through $U^1_n$. 
 To this end, let us denote the closures of the left and bottom sides of $U^1_n$ by $L^1_n$ and $L^2_n$, respectively, and proceed by defining $\Lambda^1_n$ in  different way with respect to following three cases (see Figure \ref{fig:excludecantor}): 
\begin{description}
\item[\,\,\, Case 1] \textbf{$\bm{\gamma^1_1 \cap L^1_n \neq \emptyset}$ and $\bm{\gamma_1 \cap L^2_n = \emptyset}$.} Since $L^1_n$ is closed, we deduce that $\gamma^1_1 \cap L^1_n$ is closed. Therefore, there exist $  a^2_n : = \max\{a\in \R; (l_n^1, a) \in \gamma^1_1 \cap L^1_n \}$ and $b^2_n : = \min\{b\in \R; (l_n^1, b) \in \gamma^1_1 \cap L^1_n \}$, where $l^1_n$ is the element in the singleton $\pi_1(L^1_n)$.  Since $\gamma_1$ is parametrized by $r_1$, there exist $t_n^1, t_n^2 \in [0,1] $ such that  $p_{n,1}^1 : = (l^1_n,a^2_n ) = r_1(t_n^1)$ and $p_{n,1}^2 : = (l^1_n,b^2_n ) = r_1(t_n^2)$, 
and by the continuity of $r_1$ there exists  $q_{n,1}^1  \in \gamma^1_1 \setminus  U^1_n$ such that $\dist(p_{n,1}^1,q_{n,1}^1) \le \frac{\mathrm{diam}(U^1_n)}2 $. 
If $\pi_1(q_{n,1}^1) =  l^1_n$, then we define $\wt q_{n,1}^1:= ( l_n^1 - \ep, \pi_2(q_{n,1}^1) - \ep)$ for a $\ep>0$ small enough such that $\dist(q_{n,1}^1, \wt q_{n,1}^1)\le \frac{\mathrm{diam}(U^1_n)}2 $, 
otherwise we let $\wt q_{n,1}^1: = q_{n,1}^1$. 
 We denote by $\wt L_{n,1}^1$  the segment connecting  $q_{n,1}^1$ with $\wt q_{n,1}^1$, we denote by $\wt L_{n,1}^2$ the segment  connecting $\wt q_{n,1}^1$ with $p_{n,1}^2$, and we denote by $\wt L_{n,1}^3$ the segment connecting $p_{n,1}^1$ with the vertex of $U^1_n$ in $\partial U^1_n\setminus U^1_n$. Let $\Lambda^1_n:=  \wt L_{n,1}^{1} \cup \wt L_{n,1}^{2} { \cup \wt L_{n,1}^{3}}$ and observe that by construction it follows that
\begin{equation}
      \label{eq:lastcurve1}
    \Hs^1( \Lambda^1_n     ) = \Hs^1( \wt L_{n,1}^{1} \cup \wt L_{n,1}^{2}   { \cup \wt L_{n,1}^{3}}   )\le { 2 \,} \text{diam}(U^1_n).
\end{equation}

\item[\,\,\, Case 2] \textbf{$\bm{\gamma^1_1 \cap L^1_n = \emptyset}$ and $\bm{\gamma_1 \cap L^2_n \neq \emptyset}$.} By arguing analogously to  Case 1 there exist 
$t_n^1, t_n^2 \in [0,1] $ such that  $p_{n,2}^1 : = (a^1_n, l^2_n) = r_1(t_n^1)$ and $p_{n,2}^2 : = (b^1_n, l^2_n) = r_1(t_n^2)$, where $\pi_2(L^2_n)=\{l^2_n\}$,  $a^1_n : = \max\{a\in \R; (a, l_n^2) \in \gamma^1_1 \cap L^2_n \}$, and $b^1_n : = \min\{b\in \R; (b, l_n^2) \in \gamma^1_1 \cap L^2_n \}$.  By the continuity of $r_1$ there exists $q_{n,2}^1\in \gamma^1_1 \setminus  U^1_n$ such that $\dist(p_{n,2}^1,q_{n,2}^1) \le \frac{\mathrm{diam}(U^1_n)}2 $. 
If $\pi_2(q_{n,2}^1) =  l^2_n$, then we define $\wt q_{n,2}^1:= (\pi_1(q_{n,2}^1)-\ep, l_n^2 - \ep )$ for  $\ep>0$ small enough such that $\dist(q_{n,2}^1, \wt q_{n,2}^1)\le \frac{\mathrm{diam}(U^1_n)}2 $, otherwise we let $\wt q_{n,2}^1: = q_{n,2}^1$. We denote by $\wt L_{n,2}^1$ the segment connecting $\wt q_{n,2}^1$ with $p_{n,2}^1$,  
we denote by $\wt L_{n,2}^2$ the segment connecting $\wt q_{n,2}^1$ with $p_{n,2}^2$ and we denote by $\wt L_{n,2}^3$ the segment connecting $p_{n,2}^1$ with with the vertex of $U^1_n$ in $\partial U^1_n\setminus U^1_n$. Let $\Lambda^1_n:=  \wt L_{n,2}^{1} \cup \wt L_{n,2}^{2}{ \cup \wt L_{n,2}^{3}}$ and observe that by construction it follows that
\begin{equation}
    \label{eq:lastcurve2}
    \Hs^1( \Lambda^1_n 
    ) = \Hs^1( \wt L_{n,2}^{1} \cup \wt L_{n,2}^{2}  { \cup \wt L_{n,2}^{3}}    )\le { 2\,} \text{diam}(U^1_n).
\end{equation}

\item[\,\,\, Case 3] \textbf{$\bm{\gamma^1_1 \cap L^1_n \neq \emptyset}$ and $\bm{\gamma_1 \cap L^2_n \neq \emptyset}$.}  
We define $p^k_{n,\ell}, q^{1}_{n,\ell}, \wt q_{n,\ell}^1$, $\wt L_{n,\ell}^{\alpha}$ for $k = 1,2$, $\alpha = 1,2,3 $ as in Case 1 for $\ell = 1$ and as in Case 2 for $\ell = 2$.  Furthermore, we denote by ${ \wt L_n^ 4}$ the segment connecting $p^2_{n,1}$ with $p^2_{n,2}$. Let $\Lambda^1_n:=\wt L_{n,1}^{1} \cup \wt L_{n,1}^{2} \cup  \wt L_{n,2}^{1} \cup \wt L_{n,2}^{2} {\cup \wt L^3_{n,1} \cup  \wt L^3_{n,2} \cup \wt L_n^ 4}$ and observe that by construction it follows that
\begin{align}
    \label{eq:lastcurve3}
    \Hs^1( \Lambda^1_n     ) &\le \Hs^1(\wt L_{n,1}^{1} \cup \wt L_{n,1}^{2} { \cup  \wt L^3_{n,1}}) + \Hs^1(  \wt L_{n,2}^{1} \cup \wt L_{n,2}^{2} { \cup  \wt L^3_{n,2}}) + {  \Hs^1(\wt L^4_n) }\notag\\
    &\le { 5}\, \text{diam}(U^1_n).
\end{align}
\end{description}

\begin{figure}[ht]
\centering	\includegraphics{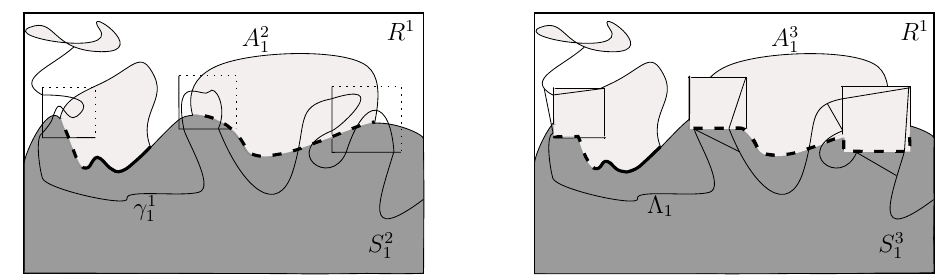}
		\caption{\small The three squares appearing in the illustration  represent dyadic squares $U^1_n$ of the first strip $R^1$ in the three different cases, namely, by moving from the left to the right, Cases 1, 2, and 3, that are considered in Step 2 of the proof of Lemma \ref{lem:finiteness}. On the left the initial pair $(A^2_1,S^2_1)$ is represented, while on the right the  pair $(A^3_1,S^3_1)$ that is obtained after the modification described in such step is depicted.} 
		\label{fig:excludecantor}
	\end{figure}

Let $$\Gamma_1:=\ov{\left(\Gamma^{\text{A}}_{\text{FS}}(\ov{A^{2}_1}, S^2_1) \cap \partial \Int(\ov{S^2_1}) \cap R^1 \right) \setminus \ov{\bigcup_{n \in Z_1} (U^1_n)}}.$$
We now observe that  the previous construction of $\bigcup_{n\in Z_1}\Lambda^1_n$ does not divide $\Gamma_1$ in an uncountable number of connected components. More precisely, we claim that for every given a connected component  $\wh \Gamma$ of $\Gamma_1$, $\wh \Gamma \setminus \ov{\bigcup_{n\in Z_1}\Lambda^1_n} $ is a countable union of disjoint connected sets. To prove this claim, we notice that, since $\ov{\wh \Gamma} \subset \gamma$ and $\gamma$ is parameterized by $r$, also $\ov{\wh \Gamma}$ is parametrizable and hence, there exists a continuous injective map $\wh r:[0,1]\to \R^2$ whose support coincides with $\ov{\wh \Gamma}$. This in particular proves that $\wh r$ is a homeomorphism between  $[0,1]$ and $\ov{\wh \Gamma}$. The claim then follows from the fact that $\ov{\wh \Gamma}  \setminus \ov{\bigcup_{n\in Z_1}\Lambda^1_n}$ is open with respect to the relative topology of $\ov{\wh \Gamma}$ and \cite[Proposition 8 in Part 1]{R}.

We are now in the position to define $A^3_1$ and $S^3_1$ as follows
\begin{multline}
\label{eq:defA^3_1}
    A^3_1:= \ov{A^2_1} \setminus \left( \left(\gamma^1 \setminus \bigcup_{n \in Z_1} U_n^1  \right) \cup \bigcup_{n\in \N} ( \partial U_{n}^1 \setminus {U_{n}^1})   \cup \bigcup_{n \in Z_1} (\ov{A^2_1} \cap \Lambda^1_n)\right) \\
    \cup  \bigcup_{n \in Z_1} ((R^1 \setminus \Int({A^2_1})) \cap \Lambda^1_n)
    \end{multline}
and
\begin{equation}
\label{eq:defS^3_1}
S^3_1 := S^2_1 \setminus \left( \bigcup_{n \in Z_1} (\ov{A^2_1} \cap  \Lambda^1_n ) \right) .
\end{equation}
Therefore,  we have that
\begin{equation}
    \label{eq:lastcorrectionbefore}
    \begin{split}
    {\mathcal{S}}_L\left(A_1^{2} , S_1^{2}, R^1  \right) - {\mathcal{S}}_L\left(A_1^{3} , S_1^{3}, R^1  \right)&\ge - 2 \int_{\bigcup_{n \in Z_1}  \Lambda^1_n 
       }\phi_{\text{F}}   (\nu_{A_3^{2}}) + \phi  (\nu_{A_3^{2}})  
       \, d\Hs^1 \\
       &  \ge - \sum_{n\in Z_1} 2 \int_{  \Lambda^1_n 
       }\phi_{\text{F}}   (\nu_{U^1_n}) + \phi  (\nu_{U^1_n})    \, d\Hs^1\\
       & \ge - \sum_{n\in Z_1} { 20} \, c_2 \text{diam}(U^1_n) \ge -  \frac {{ 40}}{j'+1} c_2\widetilde{\eta},
    \end{split}
\end{equation}
where in the first inequality we used the non-negativeness of $\phi_{\mathrm{F}}, \phi_{\mathrm{FS}}$ and $\phi$, in the second inequality we used the subadditivity of measures, in the third inequality we used \eqref{eq:H1} and \eqref{eq:lastcurve1}--\eqref{eq:lastcurve3}, and in the last inequality we used \eqref{eq:finiteness4}.

Moreover, by construction we have that 
\begin{equation}
    \label{eq:lambda1}    
\Lambda_1:=\left(\gamma^1 \setminus \left( \bigcup_{n \in Z_1}  U_n^1 \right) \right) \cup \bigcup_{n \in Z_1}  \Lambda^1_n \subset \partial A^3_1
\end{equation}
is closed, connected, and joins $T_1$ with $T_2$. 

We now modify the pair ${(A^3_1,S^3_1)}$ and $\Lambda_1$ in the strip $R^2$, first by employing the same construction of \eqref{eq:A21}  and \eqref{eq:S21} to obtain a configuration ${(A^2_2,S^2_2)}\in\wt\B$, and then by employing the same construction of  \eqref{eq:defA^3_1} and \eqref{eq:defS^3_1}  
to modify the pair ${(A^2_2,S^2_2)}$, by denoting the final modified pair with ${(A^3_2,S^3_2)}\in\wt\B$. We then define 
\begin{equation}
    \label{eq:lambda2}    
\Lambda_2:=\left(\Lambda_1 \setminus \left( \bigcup_{n \in Z_2}  U_n^2 \right) \right) \cup \bigcup_{n \in Z_2}  \Lambda^2_n \subset \partial A^3_2.
\end{equation}
 By iterating the same procedure on the strips $R^i$ for $i = 3, \ldots, j'+1$ we obtain the pair $(A^2,S^2) :=(A^3_{j'+1}, S^3_{j'+1})\in \wt \B$ and we define $\Lambda_{j'+1}$ as in \eqref{eq:lambda2} by replacing all the index 1 and 2 with $j'$ and $j'+1$, respectively. 
We observe that by \cite[Lemma 3.12]{F} there exists a support  $\wt \Lambda \subset \Lambda_{j'+1}$ of a  curve joining $T_1$ with $T_2$.

 Furthermore, as done in \eqref{eq:lastcorrectionbefore} for i=1,
\begin{equation}
    \label{eq:lastcorrection}
   {\mathcal{S}}_L\left(A_i^{2} , S_i^{2}, R^i  \right) - {\mathcal{S}}_L\left(A_i^{3} , S_i^{3}, R^i  \right) \ge -  \frac {{40}}{j'+1} c_2\widetilde{\eta}
\end{equation}
for every $i = 1, \ldots, j'+1$. 
Therefore, by iteration it follows that
\begin{equation}
    \label{eq:finitenessfinal0bis1}
    \begin{split}
       {\mathcal{S}}_L\left(A^1,S_{h^1,K^1}, R \right)  - {\mathcal{S}}_L\left(A^{2} , S^{2}, R \right) &\ge \sum_{i = 1}^{j'+1} \left( {\mathcal{S}}_L\left(A^1,S_{h^1,K^1}, R^i  \right) { - {\mathcal{S}}_L\left(A_i^{2} , S_i^{2}, R^i  \right) } \right. \\
       & \qquad  \left. { +{\mathcal{S}}_L\left(A_i^{2} , S_i^{2}, R^i  \right) - {\mathcal{S}}_L\left(A_i^{3} , S_i^{3}, R^i  \right)  }\right)\\
       &\ge   - \sum_{i = 1}^{j'+1}  \frac {{ 88}}{j'+1} c_2\widetilde{\eta}\\
       &\ge  -{ 88} c_2\widetilde{\eta},
    \end{split}
\end{equation}
where in the second inequality we used \eqref{eq:finitenessfinal0bis} and \eqref{eq:lastcorrection}.
We notice that $\Gamma^{\text{A}}_{\text{FS}}(A^{2}, S^2) \cap \partial \Int(\ov{S^2}) \cap R $ is a countable union of connected sets because by construction every connected component of $\Gamma^{\text{A}}_{\text{FS}}(A^{2}, S^2) \cap \partial \Int(\ov{S^2}) \cap R $ is connected to an element in the family of sets $\{ \partial U_n \cap U_n \}_{n \in \N}$. Therefore, $(\Gamma^{\text{A}}_{\text{FS}}(A^{2}, S^2)\setminus \Int(\ov{S^2}) )\cap R $ is equal to a countable union of connected sets.  More precisely, there exists  
    a family of connected and disjoint sets $\{\wt{\Gamma}_i\}_{i\in \N}$ such that
    \begin{equation}
        \label{eq:connected}
        (\Gamma_{\text{FS}}(A^{2}, S^2)\setminus \Int(\ov{S^2}))\cap R= \bigcup_{i \in  \N} \wt{\Gamma}_i,
    \end{equation}
and hence, 
\begin{equation}    
\label{eq:connected2} \Hs^1\left((\Gamma_{\text{FS}}(A^{2}, S^2)\setminus \Int(\ov{S^2})\cap R\right) = \sum_{i \in \N}\Hs^1\left(\wt{\Gamma}_i\right).
\end{equation}
We conclude this step by observing that Assertion (i) follows by the construction of $S^2$, while Assertion (ii) holds with the defined set  $\wt \Lambda$.

{\itshape Step 3 (From countable to a finite number of  components).} 
     Since $\Hs^1 ((\Gamma_{\text{FS}}(A^{2}, S^2)\setminus \Int(\ov{S^2}) )\cap R) \le \Hs^1(\partial S^2) < \infty$,  by \eqref{eq:connected2} 
	there exists $i_0 := i_0(\wt \eta) \in \N$ such that 
	\begin{equation}
	\label{eq:resto1}
	    \sum_{i  = i_0 +1 }^{\infty} \Hs^1 (\wt{\Gamma}_i) \le  \widetilde \eta.
	\end{equation}
Notice that 
$(\wt A,\wt S) \in  \wt \B$, where $\wt S : = S^2$ and $\wt A :=A^2 \setminus \bigcup_{i \in \wt I} \wt{\Gamma}_i $  with $\wt I:= \{ i \in  \N : i \ge i_0+1\}$. Furthermore, it follows from Steps 1 and 2 that 
\begin{equation}
        \label{eq:finitenesfinal}
        \begin{split}
            {\mathcal{S}}_L(A,S_{h,K},R) - {\mathcal{S}}_L(\wt A,\wt S,R) &= {\mathcal{S}}_L(A,S_{h,K},R) \pm {\mathcal{S}}_L(A^2,S^2,R) - {\mathcal{S}}_L(\wt A,\wt S,R)\\
            &\ge -{ 88}c_2 \wt \eta - 2\sum_{i  = i_0 +1 }^{\infty}\int_{\wt{\Gamma}_i} \phi_{\text{F}}(\nu_{\wt{\Gamma}_i}) + \phi(\nu_{\wt{\Gamma}_i})\, d\Hs^1 
            \\ 
            & \ge -{ 92} c_2\wt \eta,
        \end{split}
    \end{equation} 
where in the first inequality we used \eqref{eq:modvertical5} and \eqref{eq:finitenessfinal0bis1}, and  the definition of $\wt A$ and the non-negativeness of $\phi_{\text{FS}}$ and,  
in the second inequality we used \eqref{eq:H1} and \eqref{eq:resto1}.
We conclude this step by defining $M\in \N \cup \{0\}$ as the cardinality of $\N \setminus \wt I$, and we notice by construction that 
$(\Gamma^\text{A}_{\mathrm{ FS}} (\wt A, \wt S)\setminus \Int(\ov{ \wt S}) ) \cap R$ has at most $M$-connected components. 

Finally, we observe that Assertion (i) is a direct consequence of the construction in Steps 1 and 2, while Assertion (ii) follows from the definition of $\wt \Lambda$ in Step 2 (which is not modified in Step 3 since $\wt \Lambda\cap (\Gamma^{\text{A}}_{\text{FS}}(A^{2}, S^2)\setminus \Int(\ov{S^2}))=\emptyset$). The proof of  
this lemma is concluded 
by taking $\widetilde{\eta} := \frac{\eta}{{ 92}c_2 }$ in \eqref{eq:finitenesfinal} with  $\eta \in (0, \min\{1, { 92}c_2 \})$.
\end{proof}

We formalize below the notions   of \emph{film islands}, \emph{composite voids}, and \emph{substrate grains}  for any admissible pair $(A,S_{h,K}) \in \B$. 

  \begin{definition} \normalfont
      \label{def:notationislands2}
Let $R \subset \Om$ be an open rectangle and let $(A,S ) 
    \in \B$.  We refer to:
      \begin{itemize}

    \item any closed component $V{ \subset \overline{R}}$ of $\Om \setminus \Int(A)$ such that 
          $\partial V \cap (\Gamma_\mathrm{ S}(A, S 
          ) \cup \Gamma^{\text{D}}_\mathrm{ FS}(A, S 
          ) )$ is not empty and it consists in one and only one connected component as an \emph{extended}  (as we also include  ``connected delamination regions'') \emph{composite void} of the configuration $(A,S 
          )$ (or sometime for simplicity of the film region or of $A$).
 \item any open  connected  component $P { \subset R}$ of ${\rm Int}(A\setminus \ov{S}) 
 $ such that $\partial P \cap  \Gamma^{\text{A}}_\mathrm{ FS}(A, S )$  is not empty and it consists in one and only one connected component  as a \emph{film island} of the configuration $(A,S 
 )$ (or sometime for simplicity of the film region or of $A$), and to a film island $P={\rm Int}(A\setminus \ov{S})\cap R$ of $(A,S)$ as the \emph{full island} of $A$. 

           \item any open connected  component $G {\subset R}$ of ${\rm Int}(S) 
 $ such that $\partial G \cap  \Gamma^{\text{A}}_\mathrm{ FS}(A, S )$  is not empty and it consists in one and only one connected component  as a \emph{substrate grain} of the configuration $(A,S 
 )$ (or sometime for simplicity of the substrate region or of $S$), and to a substrate grain $G={\rm Int}(S)\cap R$ of $(A,S)$ as the \emph{full grain} of $S$. 

      \end{itemize}
  \end{definition}

	The following results can be seen as analogous of \cite[Lemmas 4.4 and 4.5]{KP} though in our  more involved setting of three free interfaces (see Table \ref{tab:my_label}), where we have to distinguish among the blow-ups at:
 \begin{itemize}
     \item[-]  the substrate free  boundary (Lemma \ref{lema:creationOneBdr}),
     \item[-]  the film-substrate incoherent (delaminated) interface (Lemma \ref{lema:creationdelamin}), 
      \item[-]  the substrate cracks in the film-substrate incoherent interface (Lemma \ref{lema:crackssubstrate}),
     \item[-] the  filaments of both the substrate and the film 
     (Lemma \ref{lema:filamentssubstrate}), 
\item[-]  the substrate filaments on the film free boundary (Lemma  \ref{lem:delaminationofS}),  
\item[-]  the delaminated substrate filaments in the film (Lemma \ref{lema:creationcrackdelamination}).
 \end{itemize}   
 The strategy employed in these proofs  is based on reducing to the situation of a finite  number of connected components for the film-substrate coherent interfaces by Lemma \ref{lem:finiteness} and then on designing induction arguments (with respect to the index of such components) in which  we ``shrink" islands,  ``fill"  voids,  and modify ``grains'' in new voids (see Figures \ref{figure:oneisland}, \ref{figure:replacementofvoids}, and \ref{figure:replacementofgrains}, respectively) by means of the \emph{minimality of segments} (see \cite[Remark 20.3]{M}). 
 
 We begin by addressing the setting of the substrate free boundary.

	\begin{lemma}
		\label{lema:creationOneBdr} 
Let $R \subset \R^2$ be an open rectangle with a side parallel to $\textbf{e}_\textbf{1}$ and let $T\subset \R^2$ be a line such that $T \cap R \neq \emptyset$. Let $\{\rho_k\}_{k\in\N} \subset [0,1] $ be such that $\rho_k \searrow 0$ and $R\subset \sigma_{\rho_{1}}(\Omega)$.  If $\{(A_k, { S_{ h_k,K_k}})\} \subset \B_\textbf{m}(\sigma_{\rho_{1}}(\Omega))$ is a sequence such that 
$\partial {S_{h_k,K_k}}\cap { \overline{R}} \cngK { T \cap  \overline{R}}$ in $\R^2$ and $(\overline{A_k} \setminus \Int \left(S_{h_k, K_k}\right))\cap{ \overline{R}} \cngK { T \cap  \overline{R}}$ in $\R^2$  as $k \to \infty$,  then for every $\delta \in (0,1)$ small enough, there exists $k_\delta\in\N$ such that 
	\begin{equation}
	\label{eq:creationexposedpartsubstrate}
		\mathcal{S}_L(A_k,S_{ h_k,K_k}, { {R}}) \ge \int_{T \cap { \overline{R}}}{\phi({ \nu_T}) \, d\Hs^1} - \delta
	\end{equation}
   for any $k \ge k_\delta$. 
	\end{lemma}	
\begin{proof}
We prove \eqref{eq:creationexposedpartsubstrate} in three steps. In the first step, we prove \eqref{eq:creationexposedpartsubstrate} for every $k \in \N$ such that $\Gamma^\text{A}_k := (\Gamma^\text{A}_{\text{FS}}(A_k,S_{h_k,K_k}) \setminus \Int(S_{h_k,K_k})) { \cap R}$ is $\Hs^1$-negligible by following the program of \cite[Lemma 4.4]{KP}.  In the second step, we consider those $k \in \N$ such that  $\Gamma^\text{A}_k$ has $\Hs^1$-positive measure and observe that in view of Lemma \ref{lem:finiteness} we can pass, up to a small error in the energy, to a triple $(\wt A_k , \wt S_k ) \in \B_\textbf{m}$ such that $\wt \Gamma^\text{A}_k:= (\Gamma^\text{A}_{\text{FS}}(\wt A_k , \wt S_k 
  )  \setminus \Int(\ov{\wt S_k
  })) { \cap R}$ has $M_k$ connected components, and which then is shown to always admit either 
an island or a void. Finally, in the third step, we apply the anisotropic minimality of segments to prove \eqref{eq:creationexposedpartsubstrate} by means of an induction argument based on shrinking the islands and/or filling the voids of the triple $(\wt A_k, \wt S_k )$.
  
If $T$ is a vertical segment we define $c_\theta :=1$, otherwise we define $c_{\theta}:=  (1/\sin \theta) + (1/\cos \theta)$, where $\theta < \pi/2$ is the smallest angle formed by the direction of $T$ with $\textbf{e}_\textbf{1}$. Since $ \overline{A_k} \setminus \Int ({S_{h_k,K_k}}) { \cap \overline R}  \cngK { T \cap \overline{R}}$, for every $\delta' \in (0,1)$ there exits $k_{\delta'} \in \N$ such that 
		\begin{equation}
			\label{eq:exposedpart0}
			\emptyset \neq \overline{A_k} \setminus \Int (S_{h_k,K_k}) { \cap \overline R} \subset { T^{\delta'}} 
		\end{equation}
		for any $k \ge k_{\delta'}$, where $T^{\delta'}:= \{x \in R: \text{dist}(x, T) < \delta'/(2 c_\theta) \}$ is the tubular neighborhood of $T$ in $R$. 
		
{\itshape Step 1.} Assume that $\Hs^1 \left(\Gamma^{\text{A}}_k \right)=0$ for a fix $k \ge k_{\delta'}$. Since 
\begin{equation}
	\label{eq:parametrizationk}
	\partial \Int\left(\overline{S_{h_k,K_k}}\right)=\partial \Int(S_{h_k})=\partial S_{h_k^+},
\end{equation}
by \cite[Lemma 3.12]{F} there exists a parametrization $r_k:[0,1]\to\mathbb{R}^2$ of ${\partial \mathrm{ \Int}\left(\overline{S_{h_k,K_k}}\right)}$,  whose support we denote by $\gamma_k$. 
Notice that by \eqref{eq:exposedpart0}, there exists 
$p_0 := r_k(t_0)$ and $p_1 := r_k(t_1)$, where $t_0 := \inf\{t\in [0,1]: r_k(t) \in \partial R  \}$ and $t_1 := \sup\{t\in [0,1]: r_k(t) \in \partial R  \}$, and also by \eqref{eq:parametrizationk} $t_0 < t_1$.  Let $\wt T_i \subset \partial R \cap \overline{T^{\delta'}}$  for $i = 0,1$ be the closed and connected set with minimal length connecting  $p_i$ with $T \cap \overline{R}$. Therefore, by trigonometric identities we obtain that $\Hs^1(\wt T_{i}) \le \frac{\delta'}{2c_\theta} c_\theta = \frac{\delta'}2$. 
We define $\Lambda_k:= (\partial R \cap \overline{T^{\delta'}}) \cap (\sigma_{\rho_k}(\Om) \setminus \Int(S_{h_k, K_k}) )$ and we observe that
		\begin{equation}
			\label{eq:exposedpart3}
			\begin{split}
				&\mathcal{S}_L (A_k, S_{h_k,K_k}, { R}) { + \int_{\Lambda_k} \phi (\nu_{\Lambda_k}) \, d\Hs^1 }\\
    &\quad \ge  \int_{\gamma_{k}{ \cap R}  \cap \partial^* {S_{h_k,K_k}} \cap  \partial^* A_k }{\phi (\nu_{A_k} )\, d\Hs^1} + \int_{\gamma_{k} { \cap R} \cap \partial^* S_{h_k, K_k} \cap \partial A_k \cap \Ako}{ \phi ( \nu_{A_k}  )\, d\Hs^1}  \\
    & \qquad \qquad { + \int_{\Lambda_k} \phi (\nu_{\Lambda_k}) \, d\Hs^1 } \\
				&\quad  = \int_{\Gamma_k}{ \phi(\nu_{\Gamma_k})\, d\Hs^1} - \sum_{i=0}^{1}{\int_{\wt T_i}{\phi( {\mathbf e}_{\mathbf 1} ) \, d\Hs^1 }},
			\end{split}
		\end{equation}
		where $\Gamma_k := \wt T_{0} \cup { (\gamma_{k} \cap \overline R) \cup \Lambda_k} \cup \wt T_1$. By the anisotropic minimality of segments (see \cite[Remark 20.3]{M}), it yields that
		\begin{equation}
			\label{eq:exposedpart4}
			\int_{\Gamma_k}{\phi(\nu_{\Gamma_k})\, d\Hs^1} \ge \int_{{ T \cap \overline{R}}}{\phi({ \nu_T}) \, d\Hs^1}, 
		\end{equation}
	and so, thanks to the facts that $\Hs^1(\wt T_{0} \cup \wt T_1)\le \delta'$ and $\Hs^1( \Lambda_k ) \le \delta'$, 
 and by \eqref{eq:H1}, \eqref{eq:exposedpart0}-\eqref{eq:exposedpart4}, we deduce that
	\begin{equation}
		\label{eq:exposedpart5}
		\mathcal{S}_L(A_k, S_{h_k,K_k},R)  \ge  \int_{{ T \cap \overline{R}}}{\phi({ \nu_T}) \, d\Hs^1} - { 2} c_2 {\delta'}.
	\end{equation}
	
	{\itshape Step 2.} Assume that $\Hs^1( \Gamma^{\text{A}}_k)> 0$ for a fixed $k \ge k_{\delta'}$. By applying Lemma \ref{lem:finiteness}, with $\Om = \B_{\textbf{m}} ( \sigma_{\rho_k} (\Om))$, there exist 
	$M_k:= M_k(\delta', A_k, h_k,K_k)\in \N$ and an admissible triple $({ \widetilde{A}^1_k}, { \wt S_k} 
	) \in \B ( \sigma_{\rho_k} (\Om))$ such that $\wt \Gamma^\text{A}_k$ 
 has $M_k$ connected components  and
	\begin{equation}
		\label{eq:exposedpart6bis}
		\mathcal{S}_L(A_k, S_{h_k,K_k},{R})  \ge \mathcal{S}_L({ \widetilde{A}^1_k}, { \wt S_k} 		, {R}) -  c_2 \delta'.
	\end{equation} 
We consider $\wt A_k := {\widetilde{A}^1_k} \cup ( (\partial \wt A_k^1\setminus \partial \wt S_k) \cap  (\wt A^1_k)^{(1)})$. By \eqref{eq:exposedpart6bis} and by the non-negativeness of $\phi_\mathrm{F}$ we deduce that
\begin{equation}
		\label{eq:exposedpart6}
		\mathcal{S}_L(A_k, S_{h_k,K_k},{R})  \ge \mathcal{S}_L({ \widetilde{A}_k}, { \wt S_k} 		, {R}) -  c_2 \delta'.
	\end{equation}

	We denote by $\{\Gamma^j\}_{j=1}^{M_k}$ the family of open connected components of $\wt \Gamma^\text{A}_k$.
	In this step we prove  that there exists at least an island or an extended void (see definition in \eqref{def:notationislands2}) in $\wt{A}_k$. More precisely, by arguing by contradiction we prove that one of the following two cases always applies: $M_k\geq1$ and there exists at least an island in $\wt{A}_k$, or $M_k\geq2$ and there exists at least an extended void in $\wt{A}_k$.

 To this end, assume that the admissible pair $(\widetilde{A}_k, {\wt S_k}  
 )$ does not present any island or extended void. We begin by observing that, since $\Hs^1(\wt  \Gamma^{\text{A}}_k)> 0$, there exist a  $j_0 \in \{1, \ldots, M_k\}$ and an open connected component $F_1$ of $\Int( \overline{ \widetilde{A}_k} \setminus \ov {\wt S_k})\cap R$ 
    such that $\Gamma^{j_0} \subset \partial F_1$. By the contradiction hypothesis, since 
    $F_1$ cannot be an island of $\wt A_k$,  there exists $j_1 \in  \{1, \ldots, M_k\} \setminus \{j_0\}$ such that $\Gamma^{j_1} \subset \partial F_1$ and $\Gamma^{j_0}\cap\Gamma^{j_1}=\emptyset$. 
 Furthermore, as by the contradiction hypothesis  there cannot be also an extended void between $\Gamma^{j_0}$ and  $\Gamma^{j_1}$, and by  using the fact that $\partial \ov {\wt S_k}$ contains $\Gamma^{j_0}$ and  $\Gamma^{j_1}$,  and  consists of only 
 one component, we conclude that there exist an open connected component  $F_2$ of $\Int( \overline{ \widetilde{A}_k} \setminus \ov {\wt S_k})\cap R$, which could coincide or not with $F_1$, and at least an extra $j_2 \in \{ 1, \ldots , M_k \} \setminus \{j_0, j_1\}$   such that $\Gamma^{j_2} \subset \partial F_2$.

 We now claim that there exist an open connected component  $F_3$ of $\Int( \overline{ \widetilde{A}_k} \setminus \ov {\wt S_k})\cap R$,  which could or not coincide with $F_2$, and at least an extra $j_3 \in \{1, \ldots, M \} \setminus \{j_0, j_1, j_2 \}$ such that $\Gamma^{j_3} \subset \partial F_3$. Indeed, if $F_1 \neq F_2$, the claim is a direct consequence of applying to $F_2$ the same argument applied to $F_1$ to find $\Gamma^{j_1}$, and we define $F_3=F_2$, 
while, if $F_2=F_1$, the claim is a direct consequence of applying to $F_2$ with the pair of components consisting, e.g., of $\Gamma^{j_0}$ and $\Gamma^{j_2}$, the same argument applied to $F_1$ with the pair of components  $(\Gamma^{j_0},\Gamma^{j_1})$ to find $\Gamma^{j_2}$, with $F_3$ being possibly, but not necessary, equal to $F_2$. 

Moreover, the same reasoning applied on $F_2$, can be implement also on $F_3$, yielding  an open connected component  $F_4$ of $\Int( \overline{ \widetilde{A}_k} \setminus \ov {\wt S_k})\cap R$ and at least an extra $j_4 \notin\{j_n\}_{n=0,\dots,3}$ such that $\Gamma^{j_4} \subset \partial F_4$. As such,  by keeping on iterating 
this reasoning we reach a contradiction with the fact that the family of connected components of $\wt \Gamma^\text{A}_k$ consists of at most $M_k<\infty$ elements.

	{\itshape Step 3.} In this step we prove \eqref{eq:creationexposedpartsubstrate} for those $k \ge k_{\delta'}$ such that  $\Hs^1( \Gamma^{\text{A}}_k)> 0$, which together with Step 1 concludes the proof of \eqref{eq:creationexposedpartsubstrate}. More precisely, we prove that 
 \begin{equation}			\label{eq:claimstep3old}
	\mathcal{S}_L(\wt A_k, \wt S_k, R 
	) \ge \int_{T \cap \overline{R}}{\phi(\nu_T) \, d\Hs^1} - 6c_2\delta',
\end{equation}
which, in view of \eqref{eq:exposedpart6},  yields Assertion (i)  by taking $\delta' := \frac{\delta}{7c_2}$ and $k_\delta:=k_{\delta'}$ for any $\delta \in (0,\min{\{ 7c_2, 1\}})$.

In order to prove \eqref{eq:claimstep3old} 
we consider an auxiliary energy $\mathcal{S}^1_L$ in $\B$  by defining
    \begin{equation}
        \label{eq:defS^1k}
        \mathcal{S}^1_L(\wt A_k, \wt S_k , R
	):= \mathcal{S}_L(\wt A_k, \wt S_k, R 
	) {  + \sum_{i = 1}^2  \int_{T^k_i} \phi_\mathrm{ F}(\nu_{\partial R}) + \phi(\nu_{\partial R}) \, d\Hs^1 } 
    \end{equation}
for every $(\wt A_k,\wt S_k 
)\in\B_\textbf{m}$, where $T^k_1$ and $T^k_2$ are the closed connected components of $\partial R \cap \overline {T^{\delta'}}$,
we recall that $\wt\Gamma^\text{A}_k$ has $M_k$ connected components, and we prove that
 \begin{equation}			
 \label{eq:claimstep3}
	\mathcal{S}^1_L(\wt A_k, \wt S_k , R
	) \ge \int_{T \cap \overline{R}}{\phi(\nu_T) \, d\Hs^1} - 2c_2\delta', 
\end{equation}
by proceeding by induction on the number $M_k\in\mathbb{N}$  
of connected components of $\wt\Gamma^\text{A}_k$  in three steps. Notice that \eqref{eq:claimstep3old} directly follows from \eqref{eq:claimstep3}, since  $$
 \Hs^1(T^k_i) \le \delta'
 $$
 by \eqref{eq:exposedpart0} and the definition of $T^k_i$ and hence,
  \begin{equation}
      \label{eq:differenceSandS1}
      \mathcal{S}_L(\wt A_k, \wt S_k , R
	) \ge \mathcal{S}^1_L(\wt A_k, \wt S_k , R
	) - 4 c_2 \delta',
  \end{equation} 
  by \eqref{eq:H1}. 
 
In Substeps 3.1 and 3.2 we prove the basis of the induction by proving it in both the two cases provided by Step 2, i.e., if $M_k=1$ and $\wt{A}_k$ presents an island, and if $M_k=2$ and $\wt{A}_k$ presents   an extended void, respectively. Finally in Substep 3.3 we prove the induction and obtain \eqref{eq:claimstep3}.

{\itshape Step 3.1}. We consider the basis of the induction in the  case in which $M_k=1$ and there is an island  $P_1{ \subset R}$ of $\widetilde A_k$ such that $\Gamma^1 \subset \partial P_1$. 
 Let $p_1, p_2 \in \partial P_1$ be two different \emph{triple junctions} of $P_1$ (see Definition \ref{def:notationislands}) such that $p_1$ and $p_2$ belong to the relative boundary of $\Gamma^1$ in $\partial { \ov{\wt S_k}}
 $ and let 
$L_1$ by the closed segment connecting $p_1$ with $p_2$. 
It follows that
 $$
    {\wh \Gamma } := (\partial  { \ov{\wt S_k}}\setminus \Gamma^1) \cup L_1  $$
 is connected and closed. We consider by $\{P^1_n\}_{n\in \N}$ the family of open connected components enclosed by $L_1$ and $\Gamma^1$. We are now in the position to characterize the modification $\wh S_k$ of $\wt S_k$ by 
 $$
 	\wh S_k := \left(\wt S_k \setminus \bigcup_{n \in \N, P^1_n \subset \ov{\wt S_k}} \ov{P^1_n} \right)   \cup  \bigcup_{n \in \N, P^1_n \subset R \setminus { \Int(\wt S_k)}} \ov{ P^1_n} .
 $$
 Furthermore, we consider by $\{P^2_n\}_{n\in \N}$ the family of open connected components enclosed by $L_1$ and $\partial P_1 \setminus \Gamma^1  $ such that for every $n \in \N$, $\partial P^2_n \cap \Gamma^1$ and  $\partial P^2_n \setminus \Gamma^1$ have one non-empty connected component. We define $$
 \wh A_k := \left( \wt A_k \setminus \bigcup_{n \in \N, P^2_n \subset \ov{\wt A_k}} \ov{P^2_n}  \right)  \cup \bigcup_{n \in \N, P^2_n \subset R \setminus \Int(\wt A_k)} \ov{P^2_n}
 $$
  (see Figure \ref{figure:oneisland}). By applying the anisotropic minimality of segments (see \cite[Remark 20.3]{M}), it yields that
\begin{equation}
	\begin{split}
	\label{eq:exposedpart10}
	& \int_{(\partial P_1 \setminus \Gamma^1) \cap \partial^* {\wt A_k} \setminus { \partial \wt S_k} }{\phi_{\mathrm{ F}}( \nu_{\wt A_k}  )\, d\Hs^1} +  \int_{(\partial P_1 \setminus \Gamma^1) \cap\left(\partial \wt A_k \setminus { \partial \wt S_k}\right)\cap \left({\wt A_k}^{(0)} \cup {\wt A_k}^{(1)}\right)}{2 \phi_{\mathrm{ F}}(\nu_{\wt A_k}  )\, d\Hs^1}  \\
	& \quad + \int_{\Gamma^1 \cap \left(\partial ^*{  \wt S_k} \setminus \partial {\wt A_k}\right)\cap {\wt A_k}^{(1)}}{ \phi_{\text{FS}} ( \nu_{{ \wt S_k}}  )\, d\Hs^1}  + \int_{\Gamma^1 \cap \left({ \partial \wt S_k} \setminus \partial {\wt A_k}\right) \cap  {  \wt S_k^{(0)}} \cap {\wt A_k}^{(1)} }{ 2 \phi_{\text{FS}} ( \nu_{{\wt S_k}}  ) \, d\Hs^1} \\
	& \quad + \int_{(\partial P_1 \setminus \Gamma^1) \cap  { \partial \wt S_k}\cap \partial ^* \wt A_k \cap { \wt S_k^{(0)}} }{ { \phi_{\mathrm{ F}} } (\nu_{\wt A_k} )  \, d\Hs^1   }   + \int_{(\partial P_1 \setminus \Gamma^1) \cap (T^k_1 \cup T^k_2) }{\phi_{\mathrm{ F}}( \nu_{\partial P_1} )\, d\Hs^1}\\ 
  & 
  \ge \int_{L_1}{\phi_{\mathrm{ F}}(\nu_{L_1}) + \phi_{\mathrm{ FS}}(\nu_{L_1}) \, \, d\Hs^1} \ge \int_{{L_1}}{ \phi(\nu_{L_1}) \, \, d\Hs^1},
	\end{split}
\end{equation}where in the last inequality we used the definition of $\phi$. 
We notice that  the last term in the left side of the previous inequality is needed to include in the analysis the situation in which $\overline{P_1}\cap\partial { R}\neq\emptyset$. 
From \eqref{eq:exposedpart10}, the inequality   \eqref{eq:claimstep3} directly follows as 
\begin{equation} 
	\label{eq:exposedpart11}
		 \mathcal{S}^1_L(\wt A_k, { \wt S_k} 
		 ,  R)  \ge \mathcal{S}^1_L( \wh A_k, {  \wh S_k}, 
		 {R}) \ge  \int_{T \cap \overline{R}}{\phi(\nu_T) \, d\Hs^1} - { 2} c_2 \delta',   
\end{equation}
where  in the last inequality we used the non-negativeness of $\phi_\mathrm{ F}$ and we proceeded by applying Step 1 to the configuration $\left(\wh A_k,\wh S_k\right)$ for  which by construction it holds that $(\Gamma^{\text{A}}_\mathrm{ FS}\left(\wh A_k,\wh S_k\right) \setminus \Int(\ov{\wh S_k}))\cap R$ is $\Hs^1$-negligible.

\begin{figure}[ht]
\centering
		\includegraphics{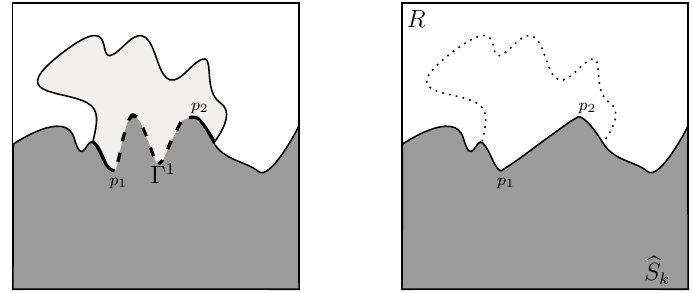}
		\caption{\small The two illustrations above represent, passing from the left to the right, the construction that consists in ``shrinking'' a film island, which is contained in Step 3.1 of the proof of  Lemma \ref{lema:creationOneBdr} for the basis of the induction in the case with $M=1$ and with $\wt A_k $  presenting an island $P_1$.} 
		\label{figure:oneisland}
	\end{figure}

{\itshape Step 3.2}.  To conclude the basis of the induction, we consider the case with  $M=2$ and the presence of an extended void  $V_1{ \subset \overline{R}}$ of $\widetilde A_k$.  Let $p_1$ and $p_2$ be the two  triple junctions such that $p_1, p_2 \in \partial V_1$ and $p_i \in \Gamma^i$, for $i = 1,2$. By \cite[Lemma 3.12]{F}  there exists a curve with support  $\gamma^1 \subset \partial V_1 \cap \partial{ \ov{\wt S_k}}
$ connecting $p_1$ with $p_2$. Furthermore, by \cite[Lemma 4.3]{KP}, since $\partial V_1$ is connected, $\Hs^1$-finite and $V_1$ is bounded,  there exists a curve with 
 support  $\gamma^2 \subset \overline{\partial V_1 \setminus (\gamma^1 \cap \partial \Int(\overline{V_1}) }$. Notice that $\gamma^1$ and $\gamma^2$ can intersect only in the delamination area, or more precisely  $\gamma^1 \cap \partial ^* { \ov{\wt S_k}}
 \cap \partial ^*\wt A_k$ and $\gamma^2 \cap \partial ^*\wt A_k \setminus \partial ^*{ \ov{\wt S_k}}
$ are disjoint up to a 
$\Hs^1$-negligible set.   We denote by $L_1$ the closed segment connecting $p_1$ with $p_2$, 
notice that
$$
    {\wh \Gamma } := \overline{(\partial { \ov{\wt S_k}}
  \setminus \gamma^1) }\cup L_1. 
 $$
 is connected and closed. We consider by $\{V^1_n\}_{n\in \N}$ the family of open connected components enclosed by $L_1$ and $\partial V_1 \cap \partial \ov{\wt S_k}$ such that for every $n \in \N$, $\partial V^1_n \cap \partial \ov{\wt S_k}$ and  $\partial V^1_n \cap L_1$ have one non-empty connected component. 
 We characterize the modification $\wh S_k$ of $\wt S_k$ as  
 $$
 	\wh S_k := \left(\wt S_k \setminus \bigcup_{n \in \N, V^1_n \subset \ov{\wt S_k}} { \ov{V^1_n}} \right)   \cup  \bigcup_{n \in \N, V^1_n \subset R \setminus \Int(\wt S_k)}V^1_n . 
 $$
  Furthermore, we define $$
 \wh A_k := \wh A_k 
  \cup (\Int{(\wh S_k)} \setminus \ov{\wt S_k} ). 
 $$
We notice that by construction  
$(\wh A_k, { \wt S_k} 
) \in \B
(\Om_k)$ and $\Gamma^{\text{A}}_\mathrm{ FS}\left( \wh A_k,{ \wt S_k}  \right) \setminus \Int(\ov{\wh S_k}) = \Gamma_1 \cup L_1 \cup \Gamma_2$ is a connected set (see Figure \eqref{figure:replacementofvoids}). Moreover, by the anisotropic minimality of segments (see \cite[Remark 20.3]{M}), it follows that
\begin{equation}
	\label{eq:exposedpart13}
	\begin{split}
		& \int_{\gamma^2 \cap \partial^* {\wt A_k} \setminus {\partial { \wt S_k} } }{\phi_{\mathrm{ F}}( \nu_{\wt A_k} )\, d\Hs^1}   + \int_{\gamma^1 \cap  \partial ^*{ \wt S_k}  \cap \partial^* {\wt A_k} }{\phi ( \nu_{\wt A_k} )\, d\Hs^1}  + \int_{\gamma^2 \cap \partial ^* {\wt A_k} \cap { \wt S_k^{(0)}}}{  \phi_{\mathrm{ F}} ( \nu_{\wt A_k} ) \, d\Hs^1}  \\ 
        & \quad + \int_{(\gamma^1 \cup \gamma^2 ) \cap  \partial ^*{{ \wt S_k} } \cap \partial \wt A_k  \cap {\wt A_k}^{(1)}}{ (\phi_{\mathrm{ F}} + \phi )( \nu_{\wt A_k} ) \, d\Hs^1}  
			\\
		& \ge \int_{L_1}{\phi_{\mathrm{ F}}(\nu_{L_1}) + \phi(\nu_{L_1})  \,  d\Hs^1} \ge  \int_{L_1}{ \phi_{\mathrm{ FS}}(\nu_{L_1})  \,  d\Hs^1} ,
	\end{split}
\end{equation}
where in the last inequality we used \eqref{eq:H2}. We now obtain  \eqref{eq:claimstep3}   by observing that 
\begin{equation} 
	\label{eq:exposedpart2}
   \mathcal{S}^1_L(\wt A_k, { \wt S_k} 
		 , { R})  \ge \mathcal{S}^1_L( \wh A_k, {  \wh S_k}, 
		 { R}) \ge \int_{T \cap \overline{R}}{\phi(\nu_T) \, d\Hs^1} - { 2}c_2\delta',
\end{equation}
where in the first inequality we used 
\eqref{eq:exposedpart13} and in the second inequality we proceed by applying Step 3.1 to the configuration $\left(\wh A_k, {  \wh S_k}\right)$,  which by construction presents a island and is such that $(\Gamma^{\text{A}}_\mathrm{ FS}\left(\wh A_k, {  \wh S_k}\right) \setminus \Int(\ov{\wh S_k})){ \cap R}$ consists of one and only component.

\begin{figure}[ht]
	\centering
 \includegraphics{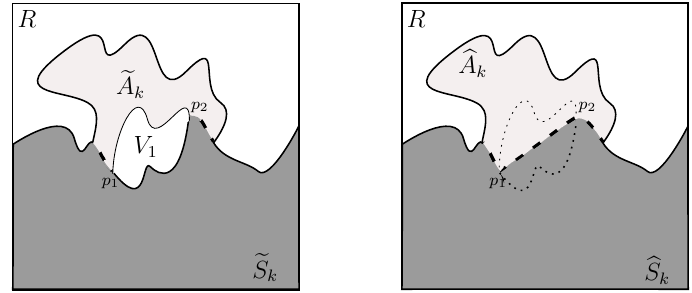}
	\caption{\small 
 The two illustrations above represent, passing from the left to the right, the construction that consists in ``filling'' a void, which is contained in Step 3.2 of the proof of the Lemma \ref{lema:creationOneBdr} for the basis of the induction in the case with $M=2$ and with $\wt A_k $  presenting a void $V_1$.}	\label{figure:replacementofvoids}
\end{figure}


{\itshape Step 3.3}.  
Now we make the inductive hypothesis that  \eqref{eq:claimstep3} holds true  if 
$\wt \Gamma^\text{A}_k$ has $M_k=j-1$ connected components, and we prove  that  \eqref{eq:claimstep3} holds also if $\wt \Gamma^\text{A}_k$ has $M_k=j$ connected components. We observe that by Step 2 we have the two cases: 
\begin{itemize}
    \item[(a)] $j\geq1$ and there exists at least an island $P{ \subset R}$ in $\wt{A}_k$; 
    \item[(b)]  $j\geq2$ and there exists at least an extended void $V{ \subset \overline{R}}$ in $\wt{A}_k$. 
\end{itemize}
In the case (a) we proceed by applying the same construction done in Step 3.1 with respect to the island $P$ instead of $P_1$ obtaining the configuration  $\left(\wh A_k, {  \wh S_k}\right) \in \B
(\Om_k)$. We observe that by construction  
$(\Gamma^{\text{A}}_\mathrm{ FS}\left(\wh A_k, {  \wh S_k}\right) \setminus \Int(\ov{\wh S_k})){ \cap R}$ presents $j-1$ connected components (since a component is canceled) and hence, we obtain that
\begin{equation} 
	\label{eq:inductionfirst}
 \mathcal{S}^1_L(\wt A_k, {  \wt S_k},{ R })
 \ge \mathcal{S}^1_L (\wh A_k, {  \wh S_k}, { R } ) \ge \int_{T \cap \overline{R}}{\phi(\nu_T) \, d\Hs^1} - { 2} c_2\delta',
\end{equation}
where  we used \eqref{eq:exposedpart11} in the first inequality  and  we applied  the induction hypothesis on $(\wh A_k, {  \wh S_k})$ in the second.

In the case (b) we proceed by applying the same construction done in Step 3.2 with respect to the extended void  $V$ instead of $V_1$ obtaining the configuration  $\left(\wh A_k, {  \wh S_k}\right) \in \B_{\textbf{m}}$. We observe that by construction  $\left(\wh A_k, {  \wh S_k}\right) \in \B
(\Om_k)$ and $(\Gamma^{\text{A}}_\mathrm{ FS}\left(\wh A_k, {  \wh S_k}\right) \setminus \Int(\ov{ \wh S_k})){ \cap R}$ presents $j-1$ connected components (since two components are connected in one). Finally, we have that  
\begin{equation} 
	\label{eq:induction2}
	 \mathcal{S}^1_L(\wt A_k, {  \wt S_k}, { R})
 \ge \mathcal{S}^1_L (\wh A_k, {  \wh S_k} , { R}) 
 \ge \int_{T \cap \overline{R}}{\phi(\nu_T) \, d\Hs^1} - { 2}c_2\delta'
\end{equation}
where we used  \eqref{eq:exposedpart2} in the first inequality  and we applied  the induction hypothesis on $(\wh A_k, {  \wh S_k})$ in the second. 
	\end{proof}

We continue with the setting of the film-substrate incoherent delaminated interface.

\begin{lemma}
	\label{lema:creationdelamin} 
	Let $R \subset \R^2$ be an open rectangle with a side parallel to $\textbf{e}_\textbf{1}$ and let $T\subset \R^2$ be a line such that $T \cap R \neq \emptyset$  and let $x\in T$. Let $\{\rho_k\}_{k\in\N} \subset [0,1] $ be such that $\rho_k \searrow 0$ and  $R\subset \sigma_{\rho_{1}}(\Omega)$.
	If $\{(A_k, S_{h_k,K_k})\} \subset \B_\textbf{m}(\sigma_{\rho_{1}}(\Omega))$ is a sequence such that $ {S_{h_k,K_k}} \cap \overline{R} \cngK { H_{x, \nu_T} } \cap \overline R$ in $\R^2$ and $  \overline{R} \setminus A_k  \cngK T \cap \overline R$ in $\R^2$ as $k \to \infty$, then for every $\delta \in (0,1)$ small enough, there exists $k_\delta\in \N$ such that 
		\begin{equation}
		\label{eq:delaminationeq}
		\mathcal{S}_L(A_k,S_{h_k,K_k},R) \ge \int_{T \cap \overline R}{\phi_{\mathrm{ F}}(\nu_T) + \phi(\nu_T) \, d\Hs^1} - \delta.
	\end{equation}
	for any $k \ge k_\delta$.
The same statement remains true if we replace $H_{x, \nu_T}$ with $H_{x, -\nu_T}$. 
\end{lemma}
\begin{proof} 
	Without loss of generality, we assume that $\sup_{k \in \N} { \mathcal{S}_L }(A_k,S_{h_k,K_k},R) < \infty$. 
	We prove \eqref{eq:delaminationeq} in two steps. In the first step, we prove \eqref{eq:delaminationeq} for every $k \in \N$ such that 	$\Gamma^\text{A}_k := (\Gamma^\text{A}_{\text{FS}}(A_k,S_{h_k,K_k}) \setminus \Int(S_{h_k,K_k})) { \cap R}$ is $\Hs^1$-negligible by repeating the same arguments of Step 1 in the proof of Lemma \ref{lema:creationOneBdr}.  In the second step, by arguing as in \cite[Lemma 4.4]{KP} we prove \eqref{eq:delaminationeq} for those $k \in \N$ such that  $\Hs^1(\Gamma^\text{A}_k)$ is positive. 
		
	If $T$ is a vertical segment we define $c_\theta :=1$, otherwise we define $c_{\theta}:= ( 1/\sin \theta) + (1/\cos \theta)$, where $\theta < \pi/2$ is the smallest angle formed by the direction of $T$ with $\textbf{e}_\textbf{1}$. Since $\partial {S_{h_k,K_k}} \cap \overline{R} \cngK T \cap \overline R$ and  $  \overline{R} \setminus A_k  \cngK T \cap \overline R$ in $\R^2$ as $k \to \infty$, for every $\delta' \in (0,1)$ there exits $k_{\delta'} \in \N$ such that  
	\begin{equation}
		\label{eq:del0}
		\partial S_{h_k, K_k} \cap \overline{R},\, \overline R \setminus A_k  \subset { T^{\delta'}} 	\end{equation}
	for every $k \ge k_{\delta'}$, where $T^{\delta'}:= \{x \in R: \text{dist}(x, T) < \delta'/(2 c_\theta) \}$ is a tubular neighborhood of $T$ in $R$. By arguing as in \eqref{eq:parametrizationk}
	there exists a parametrization $r_k:[0,1]\to\mathbb{R}^2$ of ${\partial \mathrm{ \Int}\left(\overline{S_{h_k,K_k}}\right)}$,  whose support we denote by $\gamma_k$. 	
	Finally, let $T^{\delta'}_1$ and $T^{\delta'}_2$ be the connected components of $\partial R \cap \overline{T^{\delta'}}$.

	{\itshape Step 1.} Assume that $\Hs^1 \left(\Gamma^{\text{A}}_k \right)=0$ for a fix $k \ge k_{\delta'}$. Notice that by \eqref{eq:del0}, there exists 
	$p_1 := r_k(t_1)$ and $p_2 := r_k(t_2)$, where $t_1 := \inf\{t\in [0,1]: r_k(t) \in \partial R  \}$ and $t_2 := \sup\{t\in [0,1]: r_k(t) \in \partial R  \}$, and without loss of generality we assume that $t_1 < t_2$. Let $\wt T^1_i \subset \partial R \cap {T^{\delta'}_i}$  for $i = 1,2$ be the closed and connected set with minimal length connecting  $p_i$ with $T \cap \overline{R}$. Therefore, by trigonometric identities we obtain that $\Hs^1(\wt T_{i}) \le \frac{\delta'}{2c_\theta} c_\theta = \frac{\delta'}2$. Since  $\Hs^1 \left(\Gamma^{\text{A}}_k \right)=0$ and by \cite[Lemma 4.3]{KP}, there exists a curve with 
	support  $\wt \gamma_k \subset \overline{\partial A_k \setminus (\gamma^1 \cap \partial \Int(\overline{A_k}) }$ such that $\gamma_k$ and $\wt \gamma_k$ can intersect only in the delamination area, more precisely,  $\gamma_k \cap \partial ^* S_{h_k} \cap \partial ^* A_k$ and $\wt \gamma_k \cap \partial ^* A_k \setminus \partial S_{h_k}$ are disjoint up to a 
	$\Hs^1$-negligible set. We define $\Lambda_k:= (\partial R \cap \overline{T^{\delta'}}) \cap (\Om_k \setminus \Int(S_{h_k, K_k}) )$ and 
	we denote by $\wt T^2_i \subset \partial R \cap {T^{\delta'}_i}$  for $i = 1,2$ be the closed and connected set with minimal length connecting  $\wt \gamma_k$ with each point of $T \cap T^{\delta'}_i$. It yields that
		\begin{equation}
			\label{eq:del1}
			\begin{split}
				&\mathcal{S}_L (A_k, S_{h_k,K_k}, { R}) { + \int_{\Lambda_k}  \phi (\nu_{\Lambda_k}) \, d\Hs^1 }\\
				&\quad  \ge \int_{\wt \gamma_k \cap \partial^* { A_k} \setminus {\partial {S_{ h_k, K_k}}}}{\phi_{\mathrm{ F}}( \nu_{A_k} )\, d\Hs^1}   + \int_{\gamma_k \cap  \partial ^*S_{h_k, K_k} \cap \partial^* {A_k} }{\phi ( \nu_{ A_k} )\, d\Hs^1} \\  & \qquad
				 + \int_{(\gamma_k \cup \wt  \gamma_k ) \cap  \partial ^*{S_{h_k, K_k}} \cap \partial A_k  \cap { A_k}^{(1)}}{ (\phi_{\mathrm{ F}} + \phi )( \nu_{ A_k} ) \, d\Hs^1}  
				+ \int_{\Lambda_k}  \phi (\nu_{\Lambda_k}) \, d\Hs^1  \\
	&\quad  =  \int_{{\wt \Gamma_k}}{\phi_{\mathrm{ F}}(\nu_{{\wt \Gamma_k}})\, d\Hs^1} + \int_{\Gamma_k}{\phi(\nu_{\Gamma_k})\, d\Hs^1}  - \sum_{i=1}^2 {\int_{\wt T^1_i}{\phi(\nu_{\wt T^j_i}) \, d\Hs^1 }}  - \sum_{i=1}^2 {\int_{\wt T^2_i}{\phi_{\mathrm{F}}(\nu_{\wt T^j_i}) \, d\Hs^1 }} , 
			\end{split}
		\end{equation}
where $\Gamma_k := \wt T^1_{1} \cup { (\gamma_{k} \cap \overline R) \cup \Lambda_k} \cup \wt T^1_2$ and $\wt \Gamma_k := \wt T^2_{1} \cup \wt \gamma_{k} \cup \wt T^2_2$. By the anisotropic minimality of segments (see \cite[Remark 20.3]{M}), it yields that
		\begin{equation}
			\label{eq:del2}
			\int_{{\wt \Gamma_k}}{\phi_{\mathrm{ F}}(\nu_{{\wt \Gamma_k}})\, d\Hs^1} + \int_{\Gamma_k}{\phi(\nu_{\Gamma_k})\, d\Hs^1} \ge \int_{{ T \cap \overline{R}}}{ \phi_{\text{F}}({ \nu_T}) + \phi({ \nu_T}) \, d\Hs^1}, 
		\end{equation}
		and so, thanks to the facts that $\Hs^1(\wt T^j_{0} \cup \wt T^j_1)\le \delta'$ and $\Hs^1( \Lambda_k ) \le \delta'$ for $j=1,2$,  
		and by \eqref{eq:H1}, \eqref{eq:del0}-\eqref{eq:del2}, we deduce that
		\begin{equation}
			\label{eq:del3}
			\mathcal{S}_L(A_k, S_{h_k,K_k},R)  \ge  \int_{{ T \cap \overline{R}}}{ \phi_{\text{F}}({ \nu_T}) + \phi({ \nu_T}) \, d\Hs^1} - { 3} c_2 {\delta'}.
		\end{equation}
		
		{\itshape Step 2.}
  Since $(A_k, S_{h_k, K_k})\in \B_{\textbf{m}}(\sigma_k (\Om))$ we can find an enumeration $\{\Lambda^n_k\}_{n=1, \ldots, m_k^1}$ of the connected components $\Lambda^n_k$ of $\partial {A}_k$ lying strictly inside of $R$ such that $m_k^1 \le m_1$.
Moreover, thanks to the fact that ${ \mathcal{S}_L }(A_k, S_{ h_k, K_k},Q) < \infty$ for each $k \in \N$,  
the family $\{\Lambda^{\alpha}_k\}_{\alpha\in\N}$ of connected components $\Lambda^{\alpha}_k$ that intersect $T_1^{\delta'}$ or $T_2^{\delta'}$  of  $\partial {A}_k \cap  \overline{Q_1}$, respectively, are at most countable. 
Furthermore, we define $\Lambda^{m_k +i}$ for $i=1,2$ by
$$
\Lambda^{m_k +i}:=
\left(\bigcup_{\alpha\in\mathbb{N},\, \Lambda^{\alpha} \cap T^{\delta'}_i \neq \emptyset} \Lambda^{\alpha} \right) \cup T^{\delta'}_i
$$
We denote by $\pi_T:\R^2 \to \R$ the orthogonal projection of $\R^2$ onto $T$ and since for every $n = 1,\ldots, m_k+2$, $\Lambda^n$ is a connected set we have that 
			$ \pi_T(\Lambda^n)$ is a homeomorphic to a closed interval in $\R$. More precisely, $\bigcup_{n =1}^{m_k+2} \pi _T(\Lambda^j)$ is equal to a finite family of closed segments.
Thanks to the facts that $\overline{R} \setminus A_k \cngK T\cap \overline{R}$ in $\R^2$ as $k \to \infty$, and  $m_k \le m_1$ for every $k \in \N$, we have that
	$$
	\lim_{k \to \infty} \Hs^1 \left( (T\cap \ov R) \setminus \bigcup_{n =1}^{m_k+2} \pi _T(\Lambda^n)   \right)=0
	$$		
	and hence, there exists $k^1_{\delta'} > k_{ \delta'}$ such that 
	\begin{equation}
	\label{1eq:continouspart11}
	\Hs^1 \left( (T\cap \ov R) \setminus \bigcup_{n =1}^{m_k+2} \pi _T(\Lambda^n) \right) < (m _1+2) \delta' 
	\end{equation}
	for every $k \ge k^1_{\delta'}$. 
	We denote by $a_n, b_n \in T \cap \overline{R}$ for every $n =1, \ldots, m_k+2$ the initial and final point of each $\pi _T(\Lambda^j)$, respectively (notice that $a_{m_k+1} \in T \cap T^{\delta'}_1$ and $b_{m_k+2} \in T \cap T^{\delta'}_2$). 
	We decompose $\bigcup_{n=1}^{m_k+2} \pi_T (\Lambda^n)$ as the finite union of disjoint open connected sets $\{C_j\}_{j\in J}$, where the endpoints of $C_j$ are denoted by $a'_j, b'_j \in \bigcup_{j =1}^{m_k+2} \{a_j, b_j\}$ for every $j \in J$, and $\bigcup_{j \in J}\{a'_j, b'_j\}= \bigcup_{n =1}^{m_k+2} \{a_n, b_n\}$. Therefore, by definition the cardinality of $J$ is bounded by $2m_k+3$, 
	$$
	\bigcup_{n=1}^{m_k+2} \pi_T (\Lambda^n) = \bigcup_{j \in J} C_j,
	$$
	and also by \eqref{1eq:continouspart11} we have that
			\begin{equation}
				\label{eq:dela1}
				\Hs^1 \left( (T\cap \ov R) \setminus \bigcup_{j \in J} C_j \right) < (m _1+2)  \delta'
			\end{equation}
	for every $k \ge k^1_{\delta'}$. 
Let $T_{a'_j}$ and $T_{b'_j}$ be the lines parallel to $\nu_T$ 
	and passing through  $a'_j$ and $b'_j$, respectively.  Finally, we denote by $S^j$  the intersection of the strip  between $T_{a'_j}$ and $T_{b'_j}$ and $R \cap T^{\delta'}$ for every $j \in J$. If $j \in J$ is such that $C_j \cap T^{\delta}_1 $, we define $T'_{a'_j} := T^{\delta'}_1$ and $T'_{b'_j} =  T_{b'_j}\cap \overline{S^j}  $, analogously, if $j \in J$ is such that $C_j \cap T^{\delta}_2 $ we define $T'_{a'_j} =  T_{a'_j}\cap \overline{S^j}  $ and $T'_{b'_j} := T^{\delta'}_2$, otherwise, if $ C_j \cap \bigcup_{i=1}^2 T^{\delta}_i = \emptyset $,  $T'_{a'_j} =  T_{a'_j}\cap \overline{S^j}$ and $T'_{b'_j} =  T_{b'_j}\cap \overline{S^j}$.
	It follows that 
	\begin{equation}
		\label{eq:delaminationsplit}
		\Hs^1(T'_{a'_j} \cup T'_{b'_j}) \le 2 \delta'.
	\end{equation}
  From now on we fix $ j \in J$ and we consider a fixed $k \ge k_{\delta'}^1$ such that  $\Hs^1( \Gamma^{\text{A}}_k\cap S^j)> 0$. For simplicity in the following part of this step we denote $\Lambda^n:=\Lambda^n_k$ for every $n=1,\ldots, m_k+2$. 
By applying Lemma \ref{lem:finiteness} (with $R$, as from the notation of Lemma \ref{lem:finiteness}, coinciding with $S^j$)  there exist $M_k:= M_k(A_k,S_{h_k,K_k},\delta') \in \N \cup \{0\}$  and $(\widetilde{A}_k , \widetilde{S}_k) \in \B_\textbf{m}$ such 
that  $\wt \Gamma^{\text{A}}_k: = (\Gamma^\mathrm{{A}}_\mathrm{ FS}(\widetilde{A}_k, \widetilde{S}_k  ) \setminus \Int(\overline{\wt{S}_k}))\cap S^j$ has $M_k$ connected components and 
		\begin{equation}
			\label{eqlem:finitenessbis1}
			{\mathcal{S}}_L(A_k,S_{h_k,K_k},S^j) \ge {\mathcal{S}}_L (\widetilde{A}_k, \widetilde{S}_k,S^j ) - c_2 \delta', 
		\end{equation} 
and there exists a path $\wt\Lambda^j\subset\partial \wt{A}_k$ such that  $\wt\Lambda^j\cap T'_{c_j}\neq\emptyset$ and $\partial \widetilde S_k\cap \overline{S^j}\cap T'_{c_j}\neq\emptyset$ for $c_j\in\{a'_j,b'_j\}$.  
Let $\{\Gamma_\ell\}_{\ell =1}^{M_k}$ be the family of connected components of $\wt \Gamma^{\text{A}}_k$. Without loss of generality, we assume that $\wt\Lambda^j$  intersects all islands and voids of $\wt A_k$ that are not full ones (see Definition \ref{def:notationislands2}), because otherwise we can always reduce to this situation by repeating Steps 3.1 and 3.2 of Lemma \ref{lema:creationOneBdr}. 
If $M_k =0$, by repeating the same arguments of Step 1 we obtain that 
\begin{equation}
    \label{cardinalityzero}
\mathcal{S}_L(\wt A_k, \wt S_k , S^j) \ge \int_{C_j}{ \phi_{\text{F}}({ \nu_T}) + \phi({ \nu_T}) \, d\Hs^1} - { 3} c_2 {\delta'}.
\end{equation}
Therefore, by \eqref{eqlem:finitenessbis1} 
and \eqref{cardinalityzero}, we deduce that
\begin{equation}
    \label{cardinalityzero_bis}
\mathcal{S}_L(A_k, S_k , S^j) \ge \int_{C_j}{ \phi_{\text{F}}({ \nu_T}) + \phi({ \nu_T}) \, d\Hs^1} - 4 c_2 {\delta'}.
\end{equation}
In the remaining of this step we assume that $M_k \ge 1$ 
by proving the following claim: If $M_k \ge 1$, 
then there exists 
an island of $\wt A_k$ or a substrate grain of $\wt{S}_k$ (see Definition \ref{def:notationislands2}). In order to prove this claim we proceed by contradiction. 
Therefore, let us assume that  $(\wt A_k,\wt{S}_k)$ does not contain any island and substrate grain. Since $M_k \ge 1$, 
there exists $\ell_1 \in \{1, \ldots,M_k\}$ 
and, since the endpoints of $\Gamma_{\ell_1}$ 
must be connected through $\wt \gamma \cup \partial \Int(\ov{\wt{A}_k})$, then there exists an 
open connected set $D_1\subset \Int(\wt A_k)$ enclosed by $\wt \Gamma^{\text{A}}_k 
$ and $\wt \gamma \cup \partial 
\Int(\ov{\wt{A}_k})$ such that $\Gamma_{\ell_1} 
\subset  \partial D_1 $. By assumption, since $D_1$ cannot be an island and cannot be a grain  of $(\wt A_k,\wt{S}_k)$, there must exist $\ell_2 \in \{1, \ldots,M_k\} \setminus \{\ell_1\} 
$ such that 
$\Gamma_{\ell_2} \subset \ov{D_1}$.  Since the 
endpoints of $\Gamma_{\ell_2}$ must be connected through $\wt \gamma \cup \partial \Int(\ov{\wt{A}_k})$, then there exists an open connected set $D_2\subset \Int(\wt A_k)$ enclosed by $\wt \Gamma^{\text{A}}_k$  
and $\wt \gamma \cup \partial \Int(\ov{\wt{A}_k})$ such that $\Gamma_{\ell_2} \subset  \partial D_2 $, and we notice that $D_2$ cannot coincide with $D_1$ since $\partial D_1\cap (\wt \gamma \cup \partial 
\Int(\ov{\wt{A}_k}))$ is not connecting the endpoints of $\Gamma_{\ell_2}$  (besides not connecting also the endpoints of $\Gamma_{\ell_1}$). By keeping on iterating this reasoning we reach a 
contradiction with the fact that the family $M_k < \infty$, 
and hence the claim holds true.

\begin{figure}[ht]
	\centering
 \includegraphics{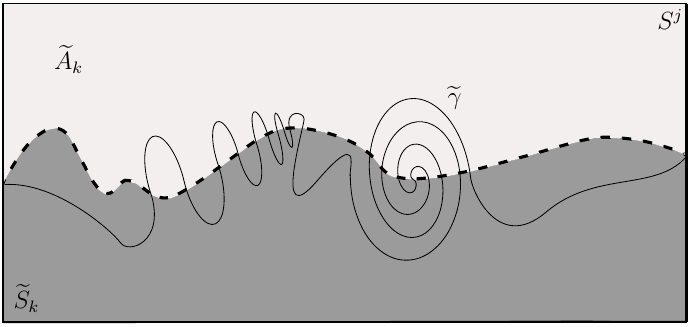}
	\caption{\small 
 The illustration represents the configuration of the admissible pair $(\wt A_k, \wt S_k)$ that is modified in Step 2 of Lemma \ref{lema:creationdelamin} to obtain the pair $(\breve A_k,\breve S_k)$ for which in the same step is proved the existence of at least a film island or a substrate grain in accordance with Definition \ref{def:notationislands}. 
 }	
 \label{figure:spiral}
\end{figure}

{\it Step 3.} 
In this step we assume that $M_k \ge 1$ 
(for the case $M_k =0$ 
see Step 2) and we
prove that 
\begin{equation}
    \label{eq:claimstep3last}
    \mathcal{S}_L(\wt A_{k},\wt S_{k} , S^j) \ge  \int_{C_j}{ \phi_{\text{F}}({ \nu_T}) + \phi({ \nu_T}) \, d\Hs^1}  -8
    c_2 \delta', 
\end{equation}
from which it easily follows from \eqref{eq:H1}, \eqref{eq:delaminationsplit} and \eqref{eqlem:finitenessbis1}, 
that
 \begin{equation}
     \label{missing1}
 \mathcal{S}_L(A_{k},S_{h_k,K_k} , S^j) \ge  \int_{C_j}{ \phi_{\text{F}}({ \nu_T}) + \phi({ \nu_T}) \, d\Hs^1}  -9c_2 \delta'. 
 \end{equation}
In order to prove \eqref{eq:claimstep3last} we consider an auxiliary energy $\mathcal{S}^1_L$ given by $\mathcal{S}_L$ with an extra term, namely defined by 
    \begin{equation}
        \label{eq:defS^1klemma5.8}
        \mathcal{S}^1_L(\wt A_k, \wt S_k , R
	):= \mathcal{S}_L(\wt A_k, \wt S_k, R 
	)  \,+ \sum_{c_j \in \{a'_j, b'_j\}} 2  \int_{T'_{c_j}}  (\phi_\mathrm{ F} + \phi
 )  \, d\Hs^1      \end{equation}
for every $(\wt A_k,\wt S_k 
)\in\B$, and we prove that 
\begin{equation}
    \label{eq:claimstep3prelast}
    \mathcal{S}^1_L(\wt A_{k},\wt S_{k} , S^j) \ge  \int_{C_j}{ \phi_{\text{F}}({ \nu_T}) + \phi({ \nu_T}) \, d\Hs^1},
\end{equation}
since \eqref{eq:claimstep3last} directly follows from \eqref{eq:claimstep3prelast} by \eqref{eq:H1} and \eqref{eq:delaminationsplit}. To prove \eqref{eq:claimstep3prelast} we proceed by induction on the number $M_k\in\mathbb{N}$  
of connected components of $\wt\Gamma^\text{A}_k$ in three steps. 
In Steps 3.1 and 3.2 we show the basis of the induction for $M_k=1$, by considering the two cases provided by Step 2, i.e., the case in which $\wt{A}_k$ presents an island in Step 3.1 and the case in which $\wt{S}_k$ presents a substrate grain in Step 3.2. We conclude then the induction in Step 3.3.

{\itshape Step 3.1} 
Assume that $M_k=1$  
and that there exists an island $P_1 \subset \ov{\wt A_k} \setminus \Int(\wt S_k)$ of $\wt A_k$ such that $P_1$ is enclosed by $\Gamma_1 
\cup \wt \gamma \cup \partial \Int(\ov{\wt A_k})) $ with $\Gamma_1
\subset \partial P_1$. Let $p_1$ and $p_2$ be the endpoints of  $\Gamma_1 
$, and let $L_1$ be the segment connecting $p_1 $ with $p_2$. We denote by $P_1^1$ the open set enclosed by $L_1$ and $\partial P_1 \cap \Gamma_1 
$ and we denote by $P_1^2$ the open set enclosed by $L_1$ and $\partial P_1 \setminus \Gamma_1
$. 

We define a modification of $\wt S_k$, denoted by $ \wh S_{k}$,  and a modification of $\wt A_k$, denoted by $\wh A_{k}$, as
$$
	\wh S_{k} := (\wt S_k \setminus (P^1_1 \cap \Int( \wt S_k )) )\cup (P^1_1 \cap (S^j \setminus \ov{\wt S_k}) ),
$$
and
$$
	\wh A_{k}: = \wt A_k \setminus \left(P_1^2 \cap (\Int(\wt A_k) \setminus \ov{\wt S_k}) \right) \cup \left(P_1^2 \cap (S^j \setminus \ov{\wt A_k}) \right),
$$
respectively. Notice that $(\wh A_k, \wh S_k) \in \B(\sigma_k(\Om))$. 
By the anisotropic minimality of segments, it follows that
\begin{equation}
\label{minimalitycase1.1bis}
\begin{split}
	\int_{\partial P_1 \setminus \wt C_1 } \phi_{\text{F}}(\nu_{\partial P_1}) \, d\Hs^1 +\int_{\partial P_1 \cap \wt C_1} \phi_{\text{FS}}(\nu_{\partial P_1}) \, d\Hs^1 & \ge \int_{L_1}  \phi_{\text{F}}(\nu_{L_1}) + \phi_{\text{FS}}(\nu_{L_1}) \, d\Hs^1 \\
 &\ge \int_{L_1} \phi(\nu_{L_n}) \, d\Hs^1,
	\end{split}
\end{equation}
where in the second inequality we used the definition of $\phi$, and hence, by  \eqref{minimalitycase1.1bis} we obtain that 
\begin{equation}
\label{eq:delamination4}
	\mathcal{S}^1_L(\wt A_k, \wt S_k, S^j) \ge \mathcal{S}^1_L(\wh A_k, \wh S_k, S^j). 
\end{equation}
Moreover,  we observe that by construction $\wh \gamma := (\wt \gamma \setminus ( \wt \gamma \cap \partial \Int(\ov{\wt A_k}) )) \cup L_1$ is path connected and it joins $ T \cap T'_{a'_j} $ with $ T \cap T'_{b'_j} $,   and  $ (\Gamma^\text{A}_{\text{FS}}(\wh A_k,\wh S_k) \setminus \Int(\ov{ \wh S_k})) { \cap S^j}$ is $\Hs^1$-negligible.  Thus, by repeating the same arguments of Step 1, we deduce that 
\begin{equation}
	\label{eq:inductionbis1}
	\mathcal{S}^1_L(\wh A_k, \wh S_k, S^j) \ge  \int_{C_j}{ \phi_{\text{F}}({ \nu_T}) + \phi({ \nu_T}) \, d\Hs^1}. 
\end{equation}
By \eqref{eq:delaminationsplit}, \eqref{eqlem:finitenessbis1}, 
\eqref{eq:defS^1klemma5.8}, \eqref{eq:delamination4}, and \eqref{eq:inductionbis1} we obtain \eqref{eq:claimstep3last}. 

{\itshape Step 3.2}
Assume that $M_k=1$ 
and that there exists %
a substrate grain   $G_1 \subset \ov{\wt S_k}$ of $\wt S_k$ such that  $G_1$ is enclosed by $\Gamma_1\cup \wt \gamma \cup \partial \Int(\ov{\wt A_k})) $ with $\Gamma_1\subset \partial G_1$. Let $p_1$ and $p_2$ be the endpoints of  $\Gamma_1$, and let $L_1$ be the segment connecting $p_1 $ with $p_2$. We denote by $G_1^1$ the open set enclosed by $L_1$ and $\partial G_1 \cap \Gamma_1$ and we denote by $G_1^2$ the open set enclosed by $L_1$ and $\partial G_1 \setminus \Gamma_1$. 

We define a modification of $\wt S_k$, denoted by $ \wh S_{k}$,  and a modification of $\wt A_k$, denoted by $\wh A_{k}$, as
$$
	\wh S_{k} := \wt S_k \setminus ((G^2_1 \cap \Int( \wt S_k ) )\cup (G^1_1 \cap \Int( \wt S_k )  )),
$$
and
$$
	\wh A_{k}: = \wt A_k \setminus \left(G^2_1 \cap \Int( \wt A_k ) \right) \cup \left(G_1^1 \cap \Int( \wt A_k ) \right),
$$
respectively (see Figure \ref{figure:replacementofgrains}). Notice that $(\wh A_k, \wh S_k) \in \B(\sigma_k(\Om))$.

By the anisotropic minimality of segments, it follows that 
\begin{equation}
\label{minimalitycase1.2bis}
\begin{split}
	\int_{\partial G_1 \setminus \Gamma_1  } \phi(\nu_{\partial G_1}) \, d\Hs^1 +\int_{\partial G_1 \cap  \Gamma_1  } \phi_{\text{FS}}(\nu_{\partial G_1}) \, d\Hs^1 & \ge \int_{L_1}  \phi(\nu_{L_1}) +\phi_{\text{FS}}(\nu_{L_1}) \, d\Hs^1 \\
 & \ge \int_{L_1}\phi_{\text{F}}(\nu_{L_1}) \, d\Hs^1,
	\end{split}
\end{equation}
where in the second inequality we used \eqref{eq:H2} and hence,  by \eqref{minimalitycase1.2bis} we obtain that 
\begin{equation}
\label{eq:delamination5}
	\mathcal{S}^1_L(\wt A_k, \wt S_k, S^j) \ge \mathcal{S}^1_L(\wh A_k, \wh S_k, S^j).
\end{equation}

\begin{figure}[ht]
	\centering
 \includegraphics{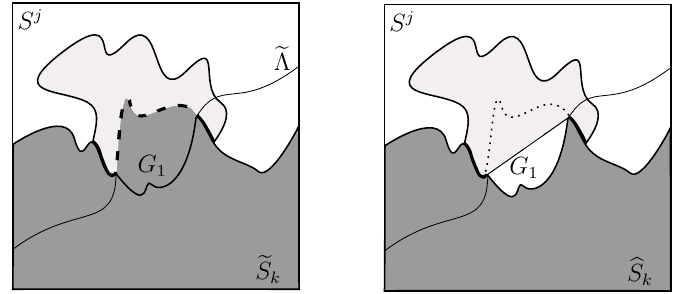}
	\caption{\small 
 The two illustrations above represent, passing from the left to the right, the construction that consists in ``modifying grains in new voids'', which is contained in Step 3.2 of the proof of the Lemma \ref{lema:creationdelamin} for the modification of the grain in a new void.}	\label{figure:replacementofgrains}
\end{figure}

Moreover, we observe that by construction $\wh \gamma := (\wt \gamma \setminus ( \wt \gamma \cap \partial \Int(\ov{\breve A_k}) )) \cup L_1$ is path connected and it joins $ T \cap T'_{a'_j} $ with $ T \cap T'_{b'_j} $,   and  $ (\Gamma^\text{A}_{\text{FS}}(\wh A_k,\wh S_k) \setminus \Int(\ov{ \wh S_k})) { \cap S^j}$ is $\Hs^1$-negligible. Thus, by repeating the same arguments of Step 1, we deduce that 
\begin{equation}
	\label{eq:inductionbis12}
	\mathcal{S}^1_L(\wh A_k, \wh S_k, S^j) \ge  \int_{C_j}{ \phi_{\text{F}}({ \nu_T}) + \phi({ \nu_T}) \, d\Hs^1}.
\end{equation}
By \eqref{eq:H1}, \eqref{eq:delaminationsplit}, \eqref{eqlem:finitenessbis1}, 
\eqref{eq:defS^1klemma5.8}, \eqref{eq:delamination5}, and \eqref{eq:inductionbis12} we obtain \eqref{eq:claimstep3last}. 

{\itshape Step 3.3}.  Assume that 
\eqref{eq:claimstep3last} 
holds true if $M_k = i-1$. 
We need to show that \eqref{eq:claimstep3last} holds also if $M_k =i$. 
By Step 2 we have two cases: 
\begin{itemize}
    \item[(a)] $i\geq1$ and there exists at least an island $P \subset S^j$ of $\wt A_k$; 
    \item[(b)]  $i\geq1$ and there exists at least a grain substrate $G\subset S^j$ of $\wt{S}_k$. 
\end{itemize}

In the case (a) we proceed by applying the same construction done in Step 3.1 with respect to the island $P$ instead of $P_1$ obtaining the configuration  $\left(\wh A_k, {  \wh S_k}\right) \in \B
(\Om_k)$. We observe that by construction  
the set 
$\wh \Gamma_k^\text{A}: = (\Gamma^\mathrm{{A}}_\mathrm{ FS}(\wh{A}_k, \wh{S}_k  ) \setminus \Int(\overline{\wh{S}_k}))\cap S^j$ has cardinality $i-1$ (since the island $P$ of $\breve A_k$ is shrunk in $\wh A_k$) and hence, we obtain that
\begin{equation} 
	\label{eq:inductionfirstbis2}
 \mathcal{S}^1_L(\wt A_k, {  \wt S_k},{ R })
 \ge \mathcal{S}^1_L (\wh A_k, {  \wh S_k}, { R } ) \ge \int_{C_j}{ \phi_{\text{F}}({ \nu_T}) + \phi({ \nu_T}) \, d\Hs^1}  
\end{equation}
where we used the inductive hypothesis in the last inequality. 

In the case (b) we proceed by applying the same construction done in Step 3.2 with respect to the substrate grain  $G$ instead of $G_1$ obtaining the configuration  $\left(\wh A_k, {  \wh S_k}\right) \in \B$. We observe that by construction   the set 
$\wh \Gamma_k^\text{A}: = (\Gamma^\mathrm{{A}}_\mathrm{ FS}(\wh{A}_k, \wh{S}_k  ) \setminus \Int(\overline{\wh{S}_k}))\cap S^j$ has cardinality $i-1$ (since the grain $G$ of $\breve S_k$ is opened in a void) and hence, we obtain that
\begin{equation} 
	\label{eq:induction2bis}
 \mathcal{S}^1_L(\wt A_k, {  \wt S_k},{ R })
 \ge \mathcal{S}^1_L (\wh A_k, {  \wh S_k}, { R } ) \ge \int_{C_j}{ \phi_{\text{F}}({ \nu_T}) + \phi({ \nu_T}) \, d\Hs^1}  
\end{equation}
where we used the inductive hypothesis in the last inequality.

	{\itshape Step 4.} Let  $k \ge k^1_{\delta'}$, which was defined in Step 2.  
 If  $\Hs^1( \Gamma^{\text{A}}_k)>0$, we consider $j \in J$. 
  By repeating the arguments of Step 1 in $S^j$ if $\Hs^1( \Gamma^{\text{A}}_k  \cap S^j)= 0$ (see \eqref{eq:del3}), and by Steps 2 and 3 (see \eqref{cardinalityzero_bis} and \eqref{eq:claimstep3last}) we obtain that 
		\begin{equation}
			\label{finaldela1}
			\begin{split}
				\mathcal{S}_L &( A_k,  S_{h_k,K_k},S^j) 
				 \ge \int_{C_j }{\phi_\mathrm{ F}(\nu_T) + \phi(\nu_T)\, d \Hs^1} -9 c_2 \delta'.
 			\end{split}
		\end{equation}
		
Therefore, in view of the fact that  the cardinality of $J$ is bounded by $2m_k+3$, by  \eqref{eq:dela1} and \eqref{finaldela1} it follows that
		\begin{equation}
			\label{eq:finaldelamination2}
			\begin{split}
				\mathcal{S}_L ( A_k,  S_{h_k,K_k},R) & \ge \sum_{j\in J} \mathcal{S}_L ( A_k,  S_{h_k,K_k},S^j) \\
    & \ge   \int_{\bigcup_{j\in J}C_j} \phi_{\mathrm{ F}} (\nu_T) +  \phi (\nu_T) \, d\Hs^1 - 9(2m_k+3) c_2 \delta '\\
				& \ge  \int_{T\cap \overline{R}} { \phi_{\mathrm{ F}} (\nu_T) +  \phi (\nu_T)\, d\Hs^1} - 2 (m_1 + 2) c_2 \delta' - 9(2m_k+3)c_2 \delta '\\
				& =  \int_{T\cap \overline{R}} { \phi_{\mathrm{ F}} (\nu_T) +  \phi (\nu_T) \, d\Hs^1} - (20 m_1 +31)c_2 \delta',
			\end{split}
		\end{equation}
		where in the last inequality we used that $m_k \le m_1$.
By recalling once again \eqref{eq:del3} of Step 1, we observe that by \eqref{eq:finaldelamination2} we obtain that 
		\begin{equation}
			\label{eq:finalfinal}
			\begin{split}
				\mathcal{S}_L ( A_k,  S_{h_k,K_k},R) \ge   \int_{T\cap \overline{R}} { \phi_{\mathrm{ F}} (\nu_T) +  \phi (\nu_T) \, d\Hs^1} - (20 m_1 +31)c_2 \delta',
			\end{split}
		\end{equation}
  for both the case with $\Hs^1( \Gamma^{\text{A}}_k)> 0$ and the case with $\Hs^1( \Gamma^{\text{A}}_k)= 0$.  Finally, follows \eqref{eq:delaminationeq}  from choosing $k_\delta := k^1_{\delta'}$ and $\delta' = \frac{\delta}{(20 m_1 +31)c_2 }$ for  $\delta \in (0, \min \{(20 m_1 +31)c_2 , 1\})$  in \eqref{eq:finalfinal}. This completes the proof.
	\end{proof}

 We continue with  the situation of the substrate cracks  in the film-substrate incoherent interface. 

\begin{lemma}
\label{lema:crackssubstrate} 
Let $T_0  $ be the $x_2$-axis. Let $\{\rho_k\}_{k\in\N} \subset [0,1] $ be such that $\rho_k \searrow 0$ and $Q_1 \subset \sigma_{\rho_1} (\Om) $. 
If $\{(A_k, S_{h_k,K_k})\} \subset \B_\textbf{m}(\sigma_{\rho_1} (\Om))$ is a sequence such that $\overline{Q_1} \setminus A_k \cngK T_0 \cap \ov{Q_1}  $ and $\overline{Q_1} \setminus S_{h_k,K_k} \cngK T_0 \cap \ov{Q_1} $ in $\R^2$ as $k \to \infty$,  then for every $\delta \in (0,1)$ small enough, there exists $k_\delta \in \N$ such that 
	\begin{equation}
	\label{eq:cracksubstrate}
		\mathcal{S}_L(A_k,S_{h_k,K_k}, Q_1) \ge 2\int_{T_0\cap \ov{Q_1}}{\phi(\bm{e_1}) \, d\Hs^1} - \delta
	\end{equation}
   for any $k \ge k_\delta$. 
	\end{lemma}
 \begin{proof}
 Without loss of generality we assume that $\sup_{k \in \N} \mathcal{S}_L(A_k,S_{h_k,K_k}, Q_1) < \infty$.  
 Since $\overline{Q_1} \setminus A_k \cngK T_0 \cap \ov{Q_1}  $ and $\overline{Q_1} \setminus S_{h_k,K_k} \cngK T_0 \cap \ov{Q_1} $ in $\R^2$ as $k \to \infty$, for every $\delta' \in (0,1)$ there exits $k_{\delta'} \in \N$ such that  
	\begin{equation}
		\label{eq:crack0}
		\overline{Q_1} \setminus A_k ,\, \overline{Q_1} \setminus S_{h_k,K_k}  \subset { T^{\delta'}} 
	\end{equation}
	for every $k \ge k_{\delta'}$, where $T^{\delta'}:= \{x \in Q_1: \text{dist}(x, T_0) < \delta'/2\}$ is the tubular neighborhood with thickness $\delta'$ of $T_0$ in $Q_1$. Let $T^{\delta'}_1$ and $T^{\delta'}_2$ be the connected components of $\partial Q_1 \cap \overline{T^{\delta'}}$. Since $(A_k, S_{h_k, K_k})\in \B_{\textbf{m}}(\sigma_k (\Om))$ we can find an enumeration 
 $\{\Gamma^n_k\}_{1, \ldots,m_k^0}$
 of connected components $\Gamma^n_k$ of $\partial S_{h_k,K_k}$ lying strictly inside of $Q_1$, and an enumeration $\{\Lambda^n_k\}_{1,\ldots,m_k^1}$ of the connected components $\Lambda^n_k$ of $\partial {A}_k$ lying strictly inside of $Q_1$,   such that $m_k^\ell \le m_\ell$ for $\ell=0,1$. Moreover, 
thanks to the fact that ${ \mathcal{S}_L }(A_k, S_{ h_k, K_k},Q) < \infty$ for each $k \in \N$,  
the families $\{\Gamma^\alpha_k\}_{\alpha\in\N}$ and   $\{\Lambda^{\alpha}_k\}_{\alpha\in\N}$ of connected components $\Gamma^\alpha_k$ and $\Lambda^{\alpha}_k$ that intersect $T_1^{\delta'}$ or $T_2^{\delta'}$  of $\partial S_{h_k,K_k} \cap  \overline{Q_1} $ and of $\partial {A}_k \cap  \overline{Q_1}$, respectively, are at most countable. Furthermore, we define $\Gamma^{m_k +i}$ and $\Lambda^{m_k +i}$ for $i=1,2$ by
$$
\Gamma^{m_k +i}:=
\left(\bigcup_{\alpha\in\mathbb{N},\, \Gamma^{\alpha} \cap T^{\delta'}_i \neq \emptyset} \Gamma^{\alpha} \right) \cup T^{\delta'}_i \quad \text{and} \quad
 \Lambda^{m_k +i}:=
\left(\bigcup_{\alpha\in\mathbb{N},\, \Lambda^{\alpha} \cap T^{\delta'}_i \neq \emptyset} \Lambda^{\alpha} \right) \cup T^{\delta'}_i .
 $$

Thanks to the Kuratowski convergences of $\overline{Q_1} \setminus A_k$ and $\overline{Q_1} \setminus S_{h_k,K_k}$ to $T_0\cap \overline{Q_1}$ in $\R^2$ as $k \to \infty$, the fact that $m^\ell_k \le m_\ell$ for $\ell = 0,1$ and for every $k \in \N$, we have that 
$$
	\lim_{k \to \infty} \Hs^1 \left( (T_0\cap \ov{Q_1}) \setminus \bigcup_{n =1}^{m^0_k+2} \pi _2(\Gamma^n)   \right)=0 \quad \text{and} \quad 
	\lim_{k \to \infty} \Hs^1 \left( (T_0\cap \ov{Q_1}) \setminus \bigcup_{n =1}^{m^1_k+2} \pi _2(\Lambda^n)   \right)=0
	$$	
Hence, there exists $k^1_{\delta'} \geq k_{ \delta'}$ such that 
	\begin{equation}
	\label{eq:crackbis1}
	\Hs^1 \left( (T_0\cap \ov{Q_1}) \setminus \bigcup_{n =1}^{m^0_k+2} \pi _2(\Gamma^n) \right) < (m _0+2)  \delta' 
	\end{equation}
	for every $k \ge k^1_{\delta'}$, and there exists $k^2_{\delta'} \geq k_{ \delta'}$ such that
\begin{equation}
	\label{eq:crackbis2}
	\Hs^1 \left( (T_0\cap \ov{Q_1}) \setminus \bigcup_{n =1}^{m^1_k+2} \pi _2(\Lambda^n) \right) < (m _1+2)  \delta' 
	\end{equation}	
for every $k \ge k^2_{\delta'}$. Let $k^3_{\delta'} := \max\{k^1_{\delta'},k^2_{\delta'}\}$. Similarly to Step 2 of Lemma \ref{lema:creationdelamin}, we can decompose $\bigcup_{n=1}^{m_k+2} \pi_2 (\Gamma^n)$ and $\bigcup_{n=1}^{m_k+2} \pi_2 (\Lambda^n)$ as the finite union of disjoint open connected sets $\mathscr{C}^0:=\{C^0_j\}_{j\in J_0}$ and $\mathscr{C}^1:=\{C^1_j\}_{j\in J_1}$, respectively. Notice that the cardinality of $J_\ell$ is bounded by $2m_k^\ell +3$ for $\ell =0,1$. Therefore
	$$
	\bigcup_{n=1}^{m^0_k+2} \pi_2 (\Gamma^n) = \bigcup_{j \in J_0} C^0_j, \quad \text{and} \quad 	\bigcup_{n=1}^{m^1_k+2} \pi_2 (\Lambda^n) = \bigcup_{j \in J_1} C^1_j,
	$$
and also by \eqref{eq:crackbis1} and \eqref{eq:crackbis2} we have that
\begin{equation}
	\label{eq:crackbis3}
	\Hs^1 \left( (T_0\cap \ov{Q_1}) \setminus \bigcup_{j \in J_\ell} C^\ell_j \right) < (m _\ell+2)  \delta'
\end{equation}
for every $k \ge k^3_{\delta'}$ and $\ell = 0,1$. We observe that 
$$\mathscr{C}:=\{C: \text{$\emptyset\neq C=C^0\cap C^1$ with $C^\ell\in\mathscr{C}^\ell$ for $\ell=0,1$}\}$$
is a family of pairwise disjoint sets and has cardinality $m^k$ bounded by $(2m_k^0 +3) (2m_k^1 +3)$, i.e., $\mathscr{C}:=\{C_j\}_{j\in J}$
for an index set $J$ with cardinality $m^k\leq (2m_k^0 +3) (2m_k^1 +3)$. We observe that 
\begin{equation}
	\label{eq:crackbis4}
\Hs^1 \left( (T_0\cap \ov{Q_1}) \setminus \bigcup_{j \in J} C_j \right)\leq \sum_{\ell\in\{0,1\}}	\Hs^1 \left( (T_0\cap \ov{Q_1}) \setminus \bigcup_{j \in J_\ell} C^\ell_j \right) < (m _0 + m_1+4)  \delta',
\end{equation}
where we used \eqref{eq:crackbis3} in the last inequality. Finally, let $S^j:= ([-1,1] \times_{\R^2} C_j) \cap T^{\delta'}$ and let $T^j_1$ and $T^j_2$ the portions of the boundary of $ S^j$ parallels to the $x_1$-axis for every $j \in J$.

We now prove \eqref{eq:cracksubstrate} in three steps. In the first step, we prove \eqref{eq:cracksubstrate} for  $k \ge k^3_{\delta'}$ such that 
	$\Gamma^\text{A}_k := \Gamma^\text{A}_{\text{FS}}(A_k,S_{h_k,K_k})	{ \cap S^j}$ is $\Hs^1$-negligible by repeating the same arguments of Step 1 in the proof of Lemma \ref{lema:creationdelamin}.  In the second step, by arguing as in Steps 2 and 3 in the proof of \ref{lema:creationdelamin} we prove \eqref{eq:delaminationeq} for those $k \ge k^3_{\delta'}$ such that  $\Hs^1(\Gamma^\text{A}_k)$ is positive. In the last step, by arguing as in Step 4 in the proof of Lemma \ref{lema:creationdelamin} we obtain \eqref{eq:cracksubstrate}.

{\itshape Step 1.} Assume that $\Hs^1 \left(\Gamma^{\text{A}}_k \right)=0$ for a fix $k \ge k^3_{\delta'}$. By construction and by  \cite[Lemma 3.12]{F} there exists  a curve with  support  $\gamma_1 \subset \partial S_{h_k,K_k}\cap (S^j \setminus \Int(S_{h_k,K_k}))$ connecting $T^{j}_1$ with $T^{j}_2$, and hence, 
by \cite[Lemma 4.3]{KP}, there exists also a curve with support  $\gamma_2 \subset \overline{\partial S_{h_k,K_k} \setminus (\gamma_1 \cap \partial \Int(\overline{S_{h_k,K_k}}) } $ such that $\gamma_1 \cap \partial ^* S_{h_k,K_k}$, $\gamma_2\cap \partial ^* S_{h_k,K_k}$ are disjoint up to an $\Hs^1$-negligible set and $\gamma_2$ joins $T^{j}_1$ with $T^{j}_2$. Since $\Hs^1 \left(\Gamma^{\text{A}}_k \right)=0$, it yields that
\begin{equation}
			\label{eq:crack1}
			\begin{split}
				\mathcal{S}_L (A_k, S_{h_k,K_k}, S^j) &+ \sum_{i = 1}^2 2\int_{T^j_i} \phi(\nu_{T^j_i}) \, d\Hs^1 \\
				  &\ge  \int_{(\gamma_1 \cup \gamma_2) \cap  \partial ^*S_{h_k, K_k} \cap \partial^* {A_k} }{\phi ( \nu_{ A_k} )\, d\Hs^1} \\
				& \quad  + \int_{(\gamma_1 \cup \gamma_2) \cap ( \partial S_{h_k, K_k} \cap \partial{A_k}) \cap (S^{(1)}_{h_k,K_k}  \cup A^{(0)}_k )  }{2\phi ( \nu_{ A_k} )\, d\Hs^1} \\
				 & \quad + \int_{(\gamma_1 \cup \gamma_2) \cap  \partial ^*{S_{h_k, K_k}} \cap \partial A_k  \cap { A_k}^{(1)}}{ (\phi_{\mathrm{ F}} + \phi )( \nu_{ A_k} ) \, d\Hs^1} \\
				& \quad + \sum_{i = 1}^2 2\int_{T^j_i} \phi(\nu_{T^j_i}) \, d\Hs^1 \\
	& =  \int_{{\Gamma_1}}{\phi (\nu_{{ \Gamma_1}})\, d\Hs^1} + \int_{\Gamma_2}{\phi(\nu_{\Gamma_2})\, d\Hs^1},  
			\end{split}
		\end{equation}
where $\Gamma_1 := T^j_{1} \cup { \gamma_1  } \cup  T^j_2$ and $\Gamma_2:= T^j_{1} \cup { \gamma_2  } \cup  T^j_2$. From the anisotropic minimality of segments (see \cite[Remark 20.3]{M}) it follows that
\begin{equation}
			\label{eq:crack2}
			\int_{{\Gamma_1}}{\phi (\nu_{{ \Gamma_1}})\, d\Hs^1} + \int_{\Gamma_2}{\phi(\nu_{\Gamma_2})\, d\Hs^1} \ge 2\int_{C_j}{  { \phi}(\textbf{e}_\textbf{1}) \, d\Hs^1}, 
		\end{equation}
and so, thanks to the facts that $\Hs^1( T^j_{1} \cup  T^j_2)\le 2 \delta'$,  and by \eqref{eq:H1}, \eqref{eq:crack0}, \eqref{eq:crack1} and \eqref{eq:crack2}, we deduce that
\begin{equation}
			\label{eq:crack3}
			\mathcal{S}_L(A_k, S_{h_k,K_k}, S^j)  \ge 2 \int_{C_j}{  { \phi}(\textbf{e}_\textbf{1}) \, d\Hs^1} - {4} c_2 {\delta'}.
		\end{equation}

{\itshape Step 2.} Assume that $\Hs^1 \left(\Gamma^{\text{A}}_k \right)>0$ for a fix $k \ge k_{\delta'}$. By construction and by \cite[Lemma 3.12]{F}, there exists a curve with support $ \wt \gamma_1 \subset \partial A_k$ connecting $T^{j_1}_1$ with $T^{j_1}_2$. If $ \Hs^1 (\wt \gamma_1 \cap  (\partial A_k \setminus \partial S_{h_k})) =0$, by repeating the same arguments of Step 1, we can deduce that 
\begin{equation}
			\label{eq:crack4}
			\mathcal{S}_L(A_k, S_{h_k,K_k}, S^j)  \ge 2 \int_{C_j}{  { \phi}(\textbf{e}_\textbf{1}) \, d\Hs^1} - {4} c_2 {\delta'}.
\end{equation}
If $ \Hs^1 ( \wt \gamma_1 \cap (\partial A_k \setminus \partial S_{h_k}) ) > 0$, then by reasoning as it was done in the proof of Lemma \ref{lema:creationdelamin} to reach \eqref{missing1} we obtain 
\begin{equation}
     \label{eq:crack6}
 \mathcal{S}_L(A_{k},S_{h_k,K_k} , S^j) \ge  2\int_{C_j}{  \phi(\textbf{e}_\textbf{1}) \, d\Hs^1}  -9c_2 \delta'. 
 \end{equation}
 More precisely, we reach \eqref{eq:crack6} by noticing that $\nu_{T_0}=\textbf{e}_\textbf{1}$ and by following the proof of Step 2 (from after equation \eqref{eq:delaminationsplit}) and of Step 3 in Lemma \ref{lema:creationdelamin} with the only difference that we replace the reference to Step 1 of Lemma \ref{lema:creationdelamin} with the reference to Step 1 of the current lemma and the extra term appearing in \eqref{eq:defS^1klemma5.8} with 
 $$
 \sum_{i = 1,2}   \int_{T^j_i}  (\phi_\mathrm{ F} + 3\phi)(\bm{e_2})  \, d\Hs^1. 
$$.

{\itshape Step 3.} Let  $k \ge k^3_{\delta'}$ and let $j \in J$. By applying  \eqref{eq:crack3} if $\Hs^1( \Gamma^{\text{A}}_k  \cap S^j)= 0$ and  \eqref{eq:crack6} if $\Hs^1( \Gamma^{\text{A}}_k  \cap S^j)> 0$  we obtain that
\begin{equation}
\label{genralcrackineq}
			\mathcal{S}_L(A_k, S_{h_k,K_k},\wt S^j)  \ge  2 \int_{C_j}{  { \phi}(\textbf{e}_\textbf{1}) \, d\Hs^1} - 9 c_2 {\delta'}.
\end{equation}
Therefore, since  the cardinality of $J$ is bounded by $(2m_k^0 +3)\cdot (2m_k^1 +3)$, \eqref{genralcrackineq} yields 
\begin{equation}
			\label{eq:finalcrackbis1}
			\begin{split}
				\mathcal{S}_L &( A_k,  S_{h_k,K_k},Q_1)  \ge \sum_{j\in J} \mathcal{S}_L ( A_k,  S_{h_k,K_k},S^j) \\
    &\ge   2\int_{\bigcup_{j\in J}C_j} \phi (\textbf{e}_\textbf{1})\, d\Hs^1 - 9(2m_k^0 +3)\cdot (2m_k^1 +3) c_2 \delta '\\
				& \ge 2 \int_{T_0\cap \overline{Q_1}} { \phi (\textbf{e}_\textbf{1})\, d\Hs^1} -(m _0 + m_1+4) c_2  \delta' - 9(2m_k^0 +3)\cdot (2m_k^1 +3) c_2 \delta '\\
				& =  2\int_{T\cap \overline{R}} { \phi (\textbf{e}_\textbf{1}) \, d\Hs^1} - \alpha \delta',
			\end{split}
		\end{equation}
		where $ \alpha := (m _0 + m_1+4) c_2 + 9(2m_0 +3)\cdot (2m_1 +3) $ and in the last inequality we used that $m^\ell_k \le m_\ell$ for $\ell = 0,1$.
Finally, \eqref{eq:cracksubstrate}  follows  from choosing $k_\delta := k^3_{\delta'}$ and $\delta' = \frac{\delta}{\alpha}$ for  $\delta \in (0, \min \{\alpha, 1\})$  in \eqref{eq:finalcrackbis1}. This completes the proof.
\end{proof}

We continue by treating   the situation related to the blow up at a point in the filaments of both the substrate and the film.

 \begin{lemma}
		\label{lema:filamentssubstrate} 
Let $T_0  $ be the $x_2$-axis. Let $\{\rho_k\}_{k\in\N} \subset [0,1] $ be such that $\rho_k \searrow 0$ and $Q_1 \subset \sigma_{\rho_1} (\Om) $. 
If $\{(A_k, S_{h_k,K_k})\} \subset \B_\textbf{m}(\sigma_{\rho_1} (\Om))$ is a sequence such that  
$Q_1 \cap A_k \cngK  T_0 \cap \ov{Q_1} $ in $\R^2$ and $Q_1 \cap S_{h_k,K_k} \cngK  T_0 \cap  \ov{Q_1}$ in $\R^2$ as $k \to \infty$,  then for every $\delta \in (0,1)$ small enough, there exists $k_\delta\in \N$ such that 
	\begin{equation}
	\label{eq:filamentsubstrate}
		\mathcal{S}_L(A_k,h_k,K_k,Q_1) \ge 2\int_{T_0}{\phi'(\bm{e_1}) \, d\Hs^1} - \delta
	\end{equation}
   for any $k \ge k_\delta$. 
	\end{lemma}
 \begin{proof}
 The situation is symmetric to the situation of  Lemma \ref{lema:crackssubstrate} and the proof is the same since $\phi' \le \phi$. \end{proof}

We continue with  the situation in the blow up of    the substrate filaments on the film free boundary.

	\begin{lemma}
		\label{lem:delaminationofS}
Let $T_0  $ be the $x_2$-axis. Let $\{\rho_k\}_{k\in\N} \subset [0,1] $ be such that $\rho_k \searrow 0$ and $Q \subset \sigma_{\rho_1} (\Om) \cap H_{0, - \bm{e_2}}$ be an open square whose sides are either parallel or perpendicular to $\bm{e_1}$ and $T_0 \cap Q \neq \emptyset$. 
If $\{(A_k, S_{h_k,K_k})\} \subset \B_\textbf{m}(\sigma_{\rho_1} (\Om))$ is a sequence such that  $\overline{Q} \cap  S_{h_k,K_k}  \cngK T_0 \cap \overline{Q} $ and $\overline {Q} \cap  A_\kn \cngK H_{0, \textbf{e}_\textbf{1}} \cap  \overline {Q}  $, then for every $\delta \in (0,1)$, there exists $k_\delta\in \N$ such that for any $k \ge k_\delta$,
\begin{equation}
	\label{eq:delaminationS}
	\mathcal{S}_L(A_k, S_{h_k,K_k},Q) \ge  \int_{T_0 \cap \overline{Q} }{ \phi_{\mathrm{ F}}({\bm e}_{{\bm 1}}) \, d\Hs^1} - \delta.
		\end{equation}
	\end{lemma}

	\begin{proof}
	Without loss of generality we assume that $\sup_{k \in \N} \mathcal{S}_L(A_k,S_{h_k,K_k}, Q) < \infty$. 
Since $\overline{Q} \cap  S_{h_k,K_k}  \cngK T_0 \cap \overline{Q_1}$ and $\overline {Q} \cap  A_\kn \cngK H_{0, \textbf{e}_\textbf{1}} \cap  \overline {Q}$ in $\R^2$ as $k \to \infty$ for every $\delta' \in (0,1)$ there exists $k_{\delta'} \in \N$ such that 
\begin{equation}
    \label{eq:filamentsdelamin1}
    \overline{Q}\cap   S_{h_k,K_k} , \, Q\cap \partial A_k \subset T^{\delta'},
\end{equation}
where $T^{\delta'} := \{x \in Q : \dist(x , T_0) < \frac{\delta'}2\}$.
Let $T_1$ be the upper side of $Q$ and let $T_1^{\delta'} : = \{x \in Q: \text{dist}(x, T_1) < \delta'/2\} $. By the Kuratowski convergence of $ \overline{Q}\cap   S_{h_k,K_k}$, there exists $k^1_{\delta'} \ge k_{\delta'}$ such that $S_{h_k} \cap \ov{T_1^{\delta'} } \neq \emptyset$ for every $k \ge k^1_{\delta'}$. Let $R := (T^{\delta'} \setminus \ov{T_1^{\delta'}})$ and  
let $T'_1, T'_2 \subset \partial R$ be the upper and lower  side of the rectangle $R$, respectively. 

Since $(A_k, S_{h_k, K_k})\in \B_{\textbf{m}}(\sigma_k (\Om))$ we can find an enumeration $\{\Lambda^n_k\}_{n=1,\ldots,m_k^1}$ of the connected components $\Lambda^n_k$ of $\partial {A}_k$ lying strictly inside of $R$,   such that $m_k^1 \le m_1$. Moreover, thanks to the fact that ${ \mathcal{S}_L }(A_k, S_{ h_k, K_k},Q) < \infty$ for each $k \in \N$, the family $\{\Lambda^{\alpha}_k\}_{\alpha\in\N}$ of connected components $\Lambda^{\alpha}_k$ of $\partial {A}_k \cap  \overline{R}$ that intersect $T_1'$ or $T_2'$, respectively, are at most countable. Furthermore, we define $\Lambda^{m_k +i}$ for $i=1,2$ by
$$
 \Lambda^{m_k +i}_k:=
\left(\bigcup_{\alpha\in\mathbb{N},\, \Lambda^{\alpha}_k \cap T^{\delta'}_i \neq \emptyset} \Lambda^{\alpha}_k \right) \cup T'_i .
 $$
Thanks to the Kuratowski convergences of $\overline{R} \cap \partial  A_k$ to $T_0\cap \overline{R}$ in $\R^2$ as $k \to \infty$, the fact that $m^1_k \le m_1$ for every $k \in \N$, we have that 
$$
	\lim_{k \to \infty} \Hs^1 \left( (T_0\cap \ov{R}) \setminus \bigcup_{n =1}^{m^1_k+2} \pi _2(\Lambda^n_k)   \right)=0
$$
Hence, there exists $k^2_{\delta'} \geq k^1_{\delta'}$ such that 
\begin{equation}
	\label{eq:filamentnfree6}
	\Hs^1 \left( (T_0\cap \ov{R}) \setminus \bigcup_{n =1}^{m^1_k+2} \pi _2(\Lambda^n_k) \right) < (m _1+2)  \delta' 
\end{equation}
for every $k \ge k^2_{\delta'}$. Similarly to Step 2 of Lemma \ref{lema:creationdelamin}, we can decompose $\bigcup_{n=1}^{m_k+2} \pi_2 (\Lambda^n_k)$ as the finite union of disjoint open connected sets $\mathscr{C}:=\{C_j\}_{j\in J}$. Notice that the cardinality of $J$ is bounded by $2m_k^1 +3$. Therefore 	
$$
\bigcup_{n=1}^{m^1_k+2} \pi_2 (\Lambda^n_k) = \bigcup_{j \in J} C_j,
$$
and also by \eqref{eq:filamentnside6} we have that
\begin{equation}
	\label{eq:filamentnsfree7}
	\Hs^1 \left( (T_0\cap \ov{R}) \setminus \bigcup_{j \in J} C_j \right) < (m _1+2) c_2  \delta'
\end{equation}
for every $k \ge k^2_{\delta'}$. Finally, let $S^j:= (\pi_1(R) \times_{\R^2} C_j) \cap T^{\delta'}$ and let $T^j_1$ and $T^j_2$ be the upper and lower sides of the boundary of each rectangle $ S^j$, respectively.
	
    \textit{Step 1.} Let $k \ge k^2_{\delta'}$ and $ j \in J$ be such that $\Hs^1(\Gamma_{\mathrm{ FS}}^\text{A} (A_k,S_{h_k,K_k}) \cap S^j)=0$. 
    In view of the  construction of $S^j$, by \cite[Lemma 3.12]{F} 
    there exists  a curve with support $\gamma_k \subset \partial (A_k \setminus \Int(S_{h_k}))
    $ joining $T^j_1$ with $T^j_2$.  
    It follows that
    \begin{equation}
    \label{eq:delaminatedfilament1}
    \begin{split}
    \mathcal{S}_L(A_k, S_{h_k,K_k},R)  &+ \sum_{i=1}^2  \int_{T^j_i} \phi_{\mathrm{F}}(\textbf{e}_\textbf{2})  \, d\Hs^1  \\
    &\ge  \int_{ \gamma_k \cap \partial^* A_k \setminus \partial {S_{h_k,K_k}}}{\phi_{\mathrm{ F}}(\nu_{A_k})\, d\Hs^1} 
    \\
    & \quad + \int_{\gamma_k \cap (\partial A_k \setminus \partial S_{h_k,K_k})\cap (A_k^{(0)} \cup A_k^{(1)})}{2 \phi_{\mathrm{ F}}(\nu_{A_k} )\, d\Hs^1} \\
&\quad + \int_{\gamma_k \cap  \partial ^*{S_{h_k,K_k}} \cap \partial A_k  \cap A_k^{(1)}}{ (\phi_{\mathrm{ F}} + \phi )( \nu_{A_k}   ) \, d\Hs^1} \\
& \quad   + \int_{\gamma_k \cap \partial{S_{h_k,K_k}} \cap \partial^* A_k \cap S_{h_k,K_k}^{(0)}}{\phi_{\mathrm{F}} ( \nu_{A_k}   )\, d\Hs^1} + \sum_{i=1}^2  \int_{T^j_i} \phi_{\mathrm{F}}(\textbf{e}_\textbf{2})  \, d\Hs^1  \\
& \ge \int_{\wh \gamma_k} \phi_{\mathrm{F}} (\nu_{\wt \gamma_k}) \, d\Hs^1, 
\end{split}
    \end{equation}
    where $\wh \gamma_k:=T^j_1\cup\gamma_k \cup T^j_2$.
By the anisotropic minimality of segments, we deduce that
\begin{equation}
    \label{eq:delaminatedfilament2}    
    \begin{split}
    	\int_{\wh \gamma_k} \phi_{\mathrm{F}} (\nu_{\wt \gamma_k}) \, d\Hs^1 
	& \ge \int_{C_j} \phi_{\mathrm{F}} (\textbf{e}_\textbf{1})	\, d\Hs^1. 
    \end{split}
    \end{equation}
By \eqref{eq:H1}, \eqref{eq:delaminatedfilament1} and \eqref{eq:delaminatedfilament2}, we conclude that 
\begin{equation}
\label{eq:delaminatedfilamentbis3} 
 \mathcal{S}_L(A_k,S_{ h_k,K_k},S^j) \ge \int_{C_j} \phi_{\mathrm{F}} (\textbf{e}_\textbf{1})	\, d\Hs^1 - 2 c_2 \delta'. 
\end{equation}

{\itshape Step 2.} Let $k \ge k^2_{\delta'}$ and $ j \in J$ be such that $\Hs^1(\Gamma_{\mathrm{ FS}}^\text{A} (A_k,S_{h_k,K_k}) \cap S^j)>0$. 
By reasoning as  in the proof of Lemma \ref{lema:creationdelamin} to reach \eqref{missing1} we obtain 
\begin{equation}
     \label{eq:filamentnsfree8}
 \mathcal{S}_L(A_{k},S_{h_k,K_k} , S^j) \ge  \int_{C_j}{  \phi_{\mathrm{F}}(\textbf{e}_\textbf{1}) \, d\Hs^1}  -7c_2 \delta'. 
 \end{equation}
 More precisely, we reach \eqref{eq:filamentnsfree8} by noticing that $\nu_{T_0}=\textbf{e}_\textbf{1}$ and by following the proof of Step 2 (from after equation \eqref{eq:delaminationsplit}) and of Step 3 of Lemma \ref{lema:creationdelamin} with the only difference that we replace the reference to Step 1 of Lemma \ref{lema:creationdelamin} with the reference to Step 1 of the current lemma, and the extra term appearing in \eqref{eq:defS^1klemma5.8} with 
 \begin{equation}
     \label{eq:error511}
 \sum_{i = 1,2}   \int_{T^j_i}  (2\phi_\mathrm{ F} + \phi)(\bm{e_2})  \, d\Hs^1. 
 \end{equation}
Notice that the difference of two units more in the error term of \eqref{eq:filamentnsfree8} with respect to \eqref{missing1} is due to the fact that the extra term \eqref{eq:error511} contributes to the final error with exactly two units more.

 {\itshape Step 3.} Fix $k \ge k^2_{\delta'}$ and let $j \in J$. By \eqref{eq:delaminatedfilamentbis3} if $\Hs^1( \Gamma^{\text{A}}_k  \cap S^j)= 0$ and by \eqref{eq:filamentnsfree8} if $\Hs^1( \Gamma^{\text{A}}_k  \cap S^j)> 0$  we obtain that
 \begin{equation}
	\label{eq:filamentinredbdbis2}
	 \mathcal{S}_L( A_k,S_k,S^j) \ge \int_{C_j}    \phi_{\mathrm{F}} (\textbf{e}_\textbf{1}) \, d\Hs^1 - 7 c_2 \delta'.
\end{equation}
Therefore, since  the cardinality of $J$ is bounded by $2(m_1 +3)$, \eqref{eq:filamentinredbdbis2} yields 
\begin{equation}
			\label{eq:finaldelaminatiofilamentbis1}
			\begin{split}
				\mathcal{S}_L ( A_k,  S_{h_k,K_k},R) & \ge \sum_{j\in J} \mathcal{S}_L ( A_k,  S_{h_k,K_k},S^j) \ge   \int_{\bigcup_{j\in J}C_j} \phi_{\mathrm{F}} (\textbf{e}_\textbf{1}) \, d\Hs^1 - 7 (2m_1 +3) c_2 \delta '\\
				& \ge  \int_{T_0\cap \overline{R}} {  \phi_{\mathrm{F}}(\textbf{e}_\textbf{1})\, d\Hs^1} -(m _1+2) c_2  \delta' - 7 (2m_1 +3) c_2 \delta '\\
				& =  \int_{T_0\cap \overline{R}} { \phi_{\mathrm{F}}(\textbf{e}_\textbf{1}) \, d\Hs^1} - (15m_1 +24) c_2 \delta'.
			\end{split}
		\end{equation}

 By the non-negativeness of $\phi_\mathrm{F}, \phi$ and $\phi_\mathrm{FS}$ we have that
\begin{equation}
\label{eq:filamentnsfree11} 
\begin{split}
	\mathcal{S}_L(A_k,S_{ h_k,K_k},Q) &\ge \mathcal{S}_L(A_k,S_{ h_k,K_k},R) \ge \int_{T_0\cap \overline{R}} { \phi_{\mathrm{F}} (\textbf{e}_\textbf{1}) \, d\Hs^1} - (15m_1 +24) c_2 \delta' \\
	& \ge  \int_{T_0\cap \overline{Q}} { \phi_{\mathrm{F}} (\textbf{e}_\textbf{1}) \, d\Hs^1} - (15m_1 +25) c_2 \delta',
\end{split}
\end{equation}
where in the second inequality we used \eqref{eq:filamentnsfree11} and in the last inequality we added and subtracted $\int_{T_0 \cap \ov{T_{1}^{ \delta'}}}  \phi_{\mathrm{F}}(\bm{e_1})\, d\Hs^1$, and we used \eqref{eq:H1} and the fact that $\Hs^1(T_0 \cap \ov{T_{1}^{ \delta'}}) \le \delta'/2$. Finally, \eqref{eq:delaminationS} follows  from choosing $k_\delta := k^2_{\delta'}$ and $\delta' = \frac{\delta}{(15m_1 +25)  c_2}$ for  $\delta \in (0, \min \{(15m_1 +25)  c_2, 1\})$  in \eqref{eq:filamentnside11}. This completes the proof.	
\end{proof}

 We conclude these list of estimates by addressing the setting of the delaminated substrate filaments in the film.

	\begin{lemma}
 \label{lema:creationcrackdelamination}
Let $T_0  $ be the $x_2$-axis. Let $\{\rho_k\}_{k\in\N} \subset [0,1] $ be such that $\rho_k \searrow 0$ and $Q \subset \sigma_{\rho_1} (\Om) \cap H_{0, - \bm{e_2}}$ be an open square whose sides are either parallel or perpendicular to $\bm{e_1}$ and $T_0 \cap Q \neq \emptyset$.
If $\{(A_k, S_{h_k,K_k})\} \subset \B_\textbf{m}(\sigma_{\rho_1} (\Om))$ is a sequence such that $\overline{Q} \cap S_{h_k,K_k}  \cngK T_0 \cap \ov{Q}$ and $ \overline{Q}  \setminus A_k  \cngK T_0 \cap \ov{Q}$, then for every $\delta \in (0,1)$, there exists $k_\delta\in \N$ such that for any $k \ge k_\delta$,
\begin{equation}
	\label{eq:fildelamination0}
	\mathcal{S}_L(A_k,S_{ h_k,K_k},Q) \ge  \int_{T_0 \cap \ov{Q}}{ 2 \phi_{\mathrm{ F}}({\bm e}_{{\bm 1}}) \, d\Hs^1} - \delta.
		\end{equation}
	\end{lemma} 
\begin{proof}
Without loss of generality we assume that $\sup_{k \in \N} \mathcal{S}_L(A_k,S_{h_k,K_k}, Q) < \infty$. Since $\overline{Q} \cap S_{h_k,K_k}  \cngK T_0 \cap \ov{Q}$ and $ \overline{Q}  \setminus A_k  \cngK T_0 \cap \ov{Q}$ in $\R^2$ as $k \to \infty$ for every $\delta' \in (0,1)$ there exists $k_{\delta'} \in \N$ such that
\begin{equation}
    \label{eq:filamentinside1}
   \overline{Q} \cap S_{h_k,K_k} , \, \overline{Q}  \setminus A_k \subset T^{\delta'},
\end{equation}
where $T^{\delta'} := \{x \in Q : \dist(x , T_0) < \frac{\delta'}2\}$. Let $T_1$ be the upper side of $Q$ and let $T_1^{\delta'} : = \{x \in Q: \text{dist}(x, T_1) < \delta'/2\} $. By the Kuratowski convergence of $ \overline{Q}\cap   S_{h_k,K_k}$, there exists $k^1_{\delta'} \ge k_{\delta'}$ such that $S_{h_k} \cap \ov{T_1^{\delta'} } \neq \emptyset$ for every $k \ge k^1_{\delta'}$. Let $R := (T^{\delta'} \setminus \ov{T_1^{\delta'}})$ and let $T'_1, T'_2 \subset \partial R$ be the upper and lower side of the rectangle $R$, respectively.  

Since $(A_k, S_{h_k, K_k})\in \B_{\textbf{m}}(\sigma_k (\Om))$ we can find an enumeration $\{\Lambda^n_k\}_{n=1,\ldots,m_k^1}$ of the connected components $\Lambda^n_k$ of $\partial {A}_k$ lying strictly inside of $R$,   such that $m_k^1 \le m_1$. Moreover, thanks to the fact that ${ \mathcal{S}_L }(A_k, S_{ h_k, K_k},Q) < \infty$ for each $k \in \N$, the family $\{\Lambda^{\alpha}_k\}_{\alpha\in\N}$ of connected components $\Lambda^{\alpha}_k$ of $\partial {A}_k \cap  \overline{R}$ that intersect $T_1'$ or $T_2'$, respectively, are at most countable. Furthermore, we define $\Lambda^{m_k +i}$ for $i=1,2$ by
$$
 \Lambda^{m_k +i}_k:=
\left(\bigcup_{\alpha\in\mathbb{N},\, \Lambda^{\alpha}_k \cap T^{\delta'}_i \neq \emptyset} \Lambda^{\alpha}_k \right) \cup T'_i .
 $$
Thanks to the Kuratowski convergences of $\overline{Q} \setminus A_k$ to $T_0\cap \overline{R}$ in $\R^2$ as $k \to \infty$, the fact that $m^1_k \le m_1$ for every $k \in \N$, we have that 
$$
	\lim_{k \to \infty} \Hs^1 \left( (T_0\cap \ov{R}) \setminus \bigcup_{n =1}^{m^1_k+2} \pi _2(\Lambda^n_k)   \right)=0
$$
Hence, there exists $k^2_{\delta'} \geq k^1_{\delta'}$ such that 
\begin{equation}
	\label{eq:filamentnside6}
	\Hs^1 \left( (T_0\cap \ov{R}) \setminus \bigcup_{n =1}^{m^1_k+2} \pi _2(\Lambda^n_k) \right) < (m _1+2)  \delta' 
\end{equation}
for every $k \ge k^2_{\delta'}$. Similarly to Step 2 of Lemma \ref{lema:creationdelamin}, we can decompose $\bigcup_{n=1}^{m_k+2} \pi_2 (\Lambda^n_k)$ as the finite union of disjoint open connected sets $\mathscr{C}:=\{C_j\}_{j\in J}$. Notice that the cardinality of $J$ is bounded by $2m_k^1 +3$. Therefore 	
$$
\bigcup_{n=1}^{m^1_k+2} \pi_2 (\Lambda^n_k) = \bigcup_{j \in J} C_j,
$$
and also by \eqref{eq:filamentnside6} we have that
\begin{equation}
	\label{eq:filamentnside7}
	\Hs^1 \left( (T_0\cap \ov{R}) \setminus \bigcup_{j \in J} C_j \right) < (m _1+2) c_2  \delta'
\end{equation}
for every $k \ge k^2_{\delta'}$. Finally, let $S^j:= (\pi_1(R) \times_{\R^2} C_j) \cap T^{\delta'}$ and let $T^j_1$ and $T^j_2$ be the upper and lower sides of the boundary of each rectangle $ S^j$, respectively.


\textit{Step 1.} Let $k \ge k^2_{\delta'}$ and $ j \in J$ be such that $\Hs^1(\Gamma_{\mathrm{ FS}}^\text{A} (A_k,S_{h_k,K_k}) \cap S^j)=0$. 
 In view of the  construction of $S^j$, by \cite[Lemma 3.12]{F} 
    we can find  a curve with support $\gamma_k^1 \subset \partial (A_k \setminus \Int(S_{h_k}))
    $ joining $T^j_1$ with $T^j_2$, and there exists only one connected component of 
    $\partial (A_k \setminus \Int(S_{h_k}))$ in $S^j$. Therefore, by applying \cite[Lemma 4.3]{KP} there exists a curve connecting $T^j_1$ and $T^j_2$ with support $\gamma_k^2 \subset S^j\cap\ov{\partial (A_k \setminus \Int(S_{h_k})) \setminus (\gamma_k^1 \cap \partial \Int(\ov{A_k \setminus \Int(S_{h_k})}))}$. 
    It follows that
    \begin{equation}
    \label{eq:filamentnside2}
    \begin{split}
    \mathcal{S}_L(A_k, S_{h_k,K_k},S^j)  &+ 2\sum_{i=1}^2  \int_{T_i'} \phi_{\mathrm{F}}(\textbf{e}_\textbf{2})  \, d\Hs^1  \\
    &\ge  \int_{ (\gamma^1_k \cup \gamma^2_k ) \cap (\partial^* A_k \setminus \partial {S_{h_k,K_k}})}{\phi_{\mathrm{ F}}(\nu_{A_k})\, d\Hs^1} 
    \\
    & \quad + \int_{(\gamma^1_k \cup \gamma^2_k ) \cap (\partial A_k \setminus \partial S_{h_k,K_k})\cap A_k^{(0)} }{2 \phi_{\mathrm{ F}}(\nu_{A_k} )\, d\Hs^1} \\
&\quad + \int_{(\gamma^1_k \cup \gamma^2_k ) \cap  \partial ^*{S_{h_k,K_k}} \cap \partial A_k  \cap A_k^{(1)}}{ \phi_{\mathrm{ F}}  ( \nu_{A_k}   ) \, d\Hs^1} \\
& \quad   + \int_{(\gamma^1_k \cup \gamma^2_k ) \cap \partial{S_{h_k,K_k}} \cap \partial^* A_k \cap S_{h_k,K_k}^{(0)}}{\phi_{\mathrm{F}} ( \nu_{A_k}   )\, d\Hs^1} + 2\sum_{i=1}^2  \int_{T_i'} \phi_{\mathrm{F}}(\textbf{e}_\textbf{2})  \, d\Hs^1   \\
& \ge \sum_{i= 1}^2 \int_{\wt \gamma^i_k} \phi_{\mathrm{F}} (\nu_{\wt \gamma^i_k}) \, d\Hs^1, 
\end{split}
    \end{equation}
where $\wt \gamma^i_k : = T^j_1 \cup \gamma_k^i \cup T^j_2$ for $i = 1,2$. 
By the anisotropic minimality of segments, we deduce that
\begin{equation}
    \label{eq:filamentnside3}   
    \begin{split}
    	\sum_{i= 1}^2 \int_{\wt \gamma^i_k} \phi_{\mathrm{F}} (\nu_{\wt \gamma^i_k}) \, d\Hs^1
	& \ge 2 \int_{C_j}  \phi_{\mathrm{F}} (\textbf{e}_\textbf{1})	\, d\Hs^1. 
    \end{split}
    \end{equation}
By \eqref{eq:H1}, \eqref{eq:filamentnside2} and \eqref{eq:filamentnside3}, we conclude that 
\begin{equation}
\label{eq:filamentnside4}  
 \mathcal{S}_L(A_k,S_{ h_k,K_k},S^j) \ge 2\int_{C_j} \phi_{\mathrm{F}} (\textbf{e}_\textbf{1})	\, d\Hs^1 - 4 c_2 \delta'. 
\end{equation}

{\itshape Step 2.} Let $k \ge k^2_{\delta'}$ and $ j \in J$ be such that $\Hs^1(\Gamma_{\mathrm{ FS}}^\text{A} (A_k,S_{h_k,K_k}) \cap S^j)>0$. 
By reasoning as  in the proof of Lemma \ref{lema:creationdelamin} to reach \eqref{missing1} we obtain 
\begin{equation}
     \label{eq:filamentnside8}
 \mathcal{S}_L(A_{k},S_{h_k,K_k} , S^j) \ge  2\int_{C_j}{  \phi_{\mathrm{F}}(\textbf{e}_\textbf{1}) \, d\Hs^1}  -9c_2 \delta'. 
 \end{equation}
 More precisely, we reach \eqref{eq:filamentnside8} by noticing that $\nu_{T_0}=\textbf{e}_\textbf{1}$ and by following the proof of Step 2 (from after equation \eqref{eq:delaminationsplit}) and of Step 3 of Lemma \ref{lema:creationdelamin} with the only difference that we replace the reference to Step 1 of Lemma \ref{lema:creationdelamin} with the reference to Step 1 of the current lemma, and the extra term appearing in \eqref{eq:defS^1klemma5.8} with 
 $$
 \sum_{i = 1,2}   \int_{T^j_i}  (3\phi_\mathrm{ F} + \phi)(\bm{e_2})  \, d\Hs^1. 
$$

 {\itshape Step 3.} Fix $k \ge k^2_{\delta'}$ and let $j \in J$. By \eqref{eq:filamentnside4} if $\Hs^1( \Gamma^{\text{A}}_k  \cap S^j)= 0$ and by \eqref{eq:filamentnside8} if $\Hs^1( \Gamma^{\text{A}}_k  \cap S^j)> 0$  we obtain that
\begin{equation}
\label{eq:filamentnside9}
			\mathcal{S}_L(A_k, S_{h_k,K_k}, S^j)  \ge  2 \int_{C_j}{  { \phi_{\mathrm{F}}}(\textbf{e}_\textbf{1}) \, d\Hs^1} - 9 c_2 {\delta'}.
\end{equation}
Therefore, the same reasoning of Step 3 of Lemma \ref{lem:delaminationofS} yields that  
\begin{equation}
\label{eq:filamentnside11} 
\begin{split}
	\mathcal{S}_L(A_k,S_{ h_k,K_k},Q) &\ge \mathcal{S}_L(A_k,S_{ h_k,K_k},R) \ge 2\int_{T\cap \overline{R}} { \phi_{\mathrm{F}} (\textbf{e}_\textbf{1}) \, d\Hs^1} - (35m_1 + 56) c_2 \delta' \\
	& \ge  2\int_{T\cap \overline{Q}} { \phi_{\mathrm{F}} (\textbf{e}_\textbf{1}) \, d\Hs^1} - (35m_1 + 57) c_2 \delta',
\end{split}
\end{equation}
where as a difference with Step 3 of Lemma \ref{lem:delaminationofS}  we added and subtracted $2\int_{T_0 \cap \ov{T_{1}^{ \delta'}}}  \phi_{\mathrm{F}}(\bm{e_1})\, d\Hs^1$ in the last inequality. Finally, \eqref{eq:fildelamination0} follows  from choosing $k_\delta := k^2_{\delta'}$ and $\delta' = \frac{\delta}{(35m_1 + 57) c_2}$ for  $\delta \in (0, \min \{(35m_1 + 57) c_2, 1\})$  in \eqref{eq:filamentnside11}. This completes the proof.	 	 
\end{proof}
	
	We are now in the position to prove that the surface energy $\mathcal{S}$ is lower semicontinuous in $\B_\textbf{m}$ with respect to the $\tau_\B$-convergence.
	
 \begin{theorem}[Lower semicontinuity of $\mathcal{S}$]
		\label{thm:lowersemicontinuityS}
		Let $(A_k,S_{ h_k,K_k})_{k \in \N} \subset \mathcal{B}_\mathbf{m}  $ and $(A,S_{h,K}) \in \mathcal{B}_\mathbf{m} $ such that $(A_k, S_{h_k,K_k}) \xrightarrow{\tau_\mathcal{B}}(A,S_{h,K})$ as $k \rightarrow \infty$. Then, 
		\begin{equation}
			\label{eq:lowersemicontinuity0}
			\mathcal{S}(A,S_{h,K}) \le \liminf_{k \rightarrow \infty}{\mathcal{S}(A_k, S_{h_k,K_k})} .
		\end{equation}
	\end{theorem}
	\begin{proof}
	Without loss of generality, we assume that the \emph{liminf} in the right side of \eqref{eq:lowersemicontinuity0} is reached and finite in $\R$. For every $k \in \N$ we denote $S_k := S_{h_k,K_k} \in \AS$ for simplicity and we define $\mu_k$  as the Radon measure  associated to ${\mathcal{S}}(A_k, S_k
 )$, i.e., the measure $\mu_k$ given by
$$		
\mu_k (B) := \int_{B \cap \Om \cap (\partial A_k \cup \partial S_k)}{\psi_k(x, \nu_k \xp)\, \, d\Hs^1}
$$
  for every Borel set $B \subset \R^2$, 
   where $\nu_k:=\nu_{\partial A_k \cup \partial S_k}$ and the surface tension $\psi_k$ is defined by 
  \begin{equation}
      \label{eq:psik}
      \psi_k(x, \nu_k(x)) :=  \begin{cases} \varphi_{\mathrm{ F}}(x, \nu_{A_k} (x))  &\text{if } x \in \partial^* A_k \setminus \partial {S_k}  \\ \varphi (x, \nu_{A_k}(x)) &\text{if } x \in \partial^* A_k \cap \partial ^*{S_k}, 
			\\ 
   { 2 \varphi_{\mathrm{ F}}(x, \nu_{A_k} (x) ) }& \text{if }x \in (\partial A_k \setminus \partial {S_k})\cap A^{(1)},
			\\ 
   { 2 \varphi'(x, \nu_{A_k} (x) )}& \text{if }x \in (\partial A_k \setminus \partial {S_k})\cap  A^{(0)},
			\\ 
   \varphi_{\mathrm{ FS}} (x, \nu_{S_k} (x) ) &\text{if } x \in (\partial ^*{S_k} \setminus \partial A_k)\cap A_k^{(1)},
			\\ 
   {2 \varphi (x, \nu_{A_k} (x) ) }& \text{if } x \in (\partial {S_k} \cap \partial A_k) \cap S^{(1)}_{k} , 
			\\ 
   { 2 \varphi' (x, \nu_{A_k} (x) )} & \text{if } x \in (\partial {S_k} \cap \partial A_k) \cap  A^{(0)}, 
			\\ 
   { \phi 
   (x, \nu_{A_k} (x) )}  &\text{if } x \in \partial {S_k} \cap \partial ^* A_k \cap S ^{(0)}_{k}, 
			\\ 2 \varphi_{\mathrm{ FS}} (x, \nu_{S_k} (x)) &\text{if } x \in (\partial {S_k} \setminus \partial A_k) \cap (S^{(1)}_{k} \cup S^{(0)}_{k}) \cap A_k^{(1)},
			\\ (\varphi_{\mathrm{ F}} + \varphi)(x, \nu_{A_k} (x) )  &\text{if } x \in \partial A_k \cap \partial ^*{S_k}  \cap A^{(1)}, 
			\\ { 2 \varphi_{\mathrm{ F}} 
   (x, \nu_{S_k} (x) )} &\text{if } x \in (\partial {S_k} \cap \partial A_k) \cap S^{(0)}_{k} \cap A_k
			^{(1)}.\end{cases} 
  \end{equation}
 Furthermore, we denote by $\mu$ the Radon measure  associated to ${\mathcal{S}}(A,h,K)$, i.e., the measure $\mu$ given by
$$		
\mu (B) := \int_{B \cap \Om \cap (\partial A \cup \partial S_{h,K})}{\psi(x, \nu_{\partial A \cup \partial S_{h,K}} \xp)\, \, d\Hs^1}
$$
   for every Borel set $B \subset \R^2$, where $\psi$ is defined analogously to $\psi_k$ in \eqref{eq:psik}, but with the sets $A_k$ and $S_k$ replaced with $A$ and $S_{h,K}$, respectively.  
  
  We observe that  by (H1) there exists $c:= c(c_2) >0$ such that
		$$ \mu_k (\R^2) = \mathcal{S}(A_k, S_k 
  ) \le c \left( \Hs^1(\partial \Ak) + \Hs^1(\partial \Sk)  \right) 
  $$
and since $(A_k, S_k
) \xrightarrow{\tau_\B} (A,S_{h,K}) $, we obtain that  $\sup_k {\mu_k (\R^2)} < +\infty$. It follows that $\{ \mu_k \} $ is a sequence of bounded Radon measures and hence, owing to the weak* compactness of Radon measures (see \cite[Theorem 4.33]{M}), there exist a not relabeled subsequence $\{ \mu_k \}$ and a Radon measure $\mu_0$ such that $\mu_{k} \stackrel{\ast}{\rightharpoonup} \mu_0$ as $k \rightarrow \infty$. 
		The purpose of this proof is to show the following inequality in the sense of measures
		\begin{equation}
			\label{eq:lowersemicontinuity1}
			\mu_0 \ge \mu.
		\end{equation}		
		Since $\mu_0$ and $\mu$ are non-negative measures and $\mu << \Hs^1 \mres (\partial A \cup \partial {S_{h,K}} )$, to obtain \eqref{eq:lowersemicontinuity1}, 
   it is enough to prove that the surface tension $\psi$ of $\mu$ on each subset of $\partial A \cup \partial {S_{h,K}}$ on which it is uniquely defined, is bounded from above by the Radon-Nikodym  
   derivative of $\mu_0$ with respect to  the $\Hs^1$-measure of the corresponding subset, namely the following 12 inequalities:
\begin{alignat}{3}
& \frac{d \mu_0}{\, d\Hs^1 \mres ( \partial^* {A} \setminus { \partial {S_{h,K}}})} \xp  \ge \varphi_{\mathrm{ F}}(x , \nu_A \xp ) & & \textrm{ for $\Hs^1$-a.e. $x \in  \partial^* {A} \setminus { \partial {S_{h,K}}}$}, \label{eq:blowupa} \\
&  \frac{d \mu_0}{\, d\Hs^1 \mres (\partial ^* S_{h,K} \cap\partial ^* A )} \xp \ge \varphi (x, \nu_{A} \xp) & & \textrm{ for $\Hs^1$-a.e. $x \in \partial ^* S_{h,K} \cap\partial ^* A$, } \label{eq:blowupb}
\end{alignat}
 \begin{alignat}{3}
& \frac{d \mu_0}{\, d\Hs^1 \mres ((\partial A \setminus \partial {S_{h,K}})\cap  A^{(1)}  )} \xp \ge 2 \varphi_{\mathrm{ F}}(x, \nu_A \xp ) & \quad & \textrm{ for $\Hs^1$-a.e. $x \in (\partial A\setminus \partial {S_{h,K}}) \cap  A^{(1)}   , $}	\label{eq:blowupc1} \\
& \frac{d \mu_0}{\, d\Hs^1 \mres ((\partial A \setminus \partial {S_{h,K}})\cap   A^{(0)} )} \xp \ge { 2 \varphi' (x, \nu_A \xp ) 
}&  & \textrm{ for $\Hs^1$-a.e. $x \in (\partial A\setminus \partial {S_{h,K}}) \cap   A^{(0)}   ,  $}	\label{eq:blowupc_2} 
\end{alignat}
 \begin{multline}
  \label{eq:blowupd}
\frac{d \mu_0}{\, d\Hs^1 \mres ((\partial ^*{S_{h,K}} \setminus \partial A)\cap A^{(1)}) } \xp \ge  \varphi_{\mathrm{ FS}} (x, \nu_{S_{h,K}} \xp )\\   \textrm{ for $\Hs^1$- a.e. $x \in (\partial ^*{S_{h,K}} \setminus \partial A)\cap A^{(1)},$ }  
		\end{multline}
 \begin{alignat}{3}
& \frac{d \mu_0}{\, d\Hs^1 \mres ( (\partial {S_{h,K}} \cap \partial A) \cap S^{(1)}_{h,K})} \xp \ge  2 \varphi (x, \nu_A \xp ) & & \textrm{ for $\Hs^1$-a.e. $x \in (\partial {S_{h,K}} \cap \partial A) \cap S^{(1)}_{h,K}, $} \label{eq:blowupe}  \\
& \frac{d \mu_0}{\, d\Hs^1 \mres ( (\partial {S_{h,K}} \cap \partial A) \cap A^{(0)})} \xp \ge  { 2 \varphi'(x, \nu_A \xp ) }& & \textrm{ for $\Hs^1$-a.e. $x \in (\partial {S_{h,K}} \cap \partial A) \cap A^{(0)}, $} \label{eq:blowupea1}
 \end{alignat}
\begin{multline}
   \label{eq:blowupf} 
    \frac{d \mu_0}{\, d\Hs^1 \mres  (\partial {S_{h,K}} \cap \partial^* A \cap  \Sz) } \xp \ge  {  \varphi_{\mathrm{ F}} } (x, \nu_A \xp ) \\  \textrm{for $\Hs^1$-a.e. $x \in \partial {S_{h,K}} \cap \partial^* A \cap  \Sz, $} 
\end{multline}
\begin{multline}
	\label{eq:blowupg} 		
	\frac{d \mu_0}{\, d\Hs^1 \mres ((\partial {S_{h,K}} \setminus \partial A) \cap S^{(1)}_{h,K} \cap A^{(1)})} \xp \ge 2 \varphi_{\mathrm{ FS}} (x, \nu_{S_{h,K}} \xp) \\ \textrm{for $\Hs^1$-a.e. }
	x \in (\partial {S_{h,K}} \setminus \partial A)  \cap S^{(1)}_{h,K} \cap A^{(1)},	 
 \end{multline}
\begin{multline}
	\label{eq:blowupga}
	\frac{d \mu_0}{\, d\Hs^1 \mres ((\partial {S_{h,K}} \setminus \partial A) \cap S^{(0)}_{h,K} \cap A^{(1)})} \xp \ge 2 \varphi_{\mathrm{ FS}} (x, \nu_{S_{h,K}} \xp)\, \\
	\textrm{for $\Hs^1$-a.e. $  x \in (\partial {S_{h,K}} \setminus \partial A)  \cap S^{(0)}_{h,K} \cap A^{(1)},	 $}
\end{multline} 
	\begin{multline}
 \label{eq:blowuph}
	 \frac{d  \mu_0}{\, d\Hs^1 \mres (\partial ^*{S_{h,K}} \cap \partial A   \cap A^{(1)})} \xp \ge (\varphi_{\mathrm{ F}} +  \varphi)(x, \nu_A \xp ) \\ \textrm{for $ \Hs^1$-a.e. $x \in \partial ^*{S_{h,K}} \cap \partial A   \cap A^{(1)}$} 
		\end{multline}
		and
  \begin{multline}
  \label{eq:blowupi}
\frac{d \mu_0}{\, d\Hs^1 \mres ((\partial {S_{h,K}} \cap \partial A)\cap S^{(0)}_{h,K} \cap A^{(1)})} \xp \ge    { {  2 \varphi_{\mathrm{ F}}}  } (x, \nu_{S_{h,K}} \xp )  \\
\textrm{for $\Hs^1$-a.e. $x \in (\partial {S_{h,K}} \cap \partial A)  \cap S^{(0)}_{h,K} \cap A^{(1)}	
			.$ } 
		\end{multline}

The rest of the proof is devoted to establish the previous 12 lower-bound estimates. In order to do that we fix $\ep> 0$ small enough and  we  recall that from  the uniform continuity of $\varphi_{\mathrm{ F}}, \varphi_{\mathrm{ S}}, \varphi_{\mathrm{ FS}}$ it follows that   there exists a $\delta_\ep>0$ such that 
		\begin{equation}
			\label{eq:continuitynorms}
			\varphi_{\mathrm{ F}}(y, \xi) \ge \varphi_{\mathrm{ F}}(x_0, \xi) - \ep, \quad \varphi_{\mathrm{ S}}(y, \xi) \ge \varphi_{\mathrm{ S}}(x_0, \xi) - \ep, \text{ and } \varphi_{\mathrm{ FS}}(y, \xi) \ge \varphi_{\mathrm{ FS}}(x_0, \xi) - \ep,
		\end{equation}
for every $y \in Q_{\delta_\ep}(x_0)\subset\Omega$,  $x_0\in\Omega$ and $\abs{\xi}=1$. The proofs of (\ref{eq:blowupa})  and \eqref{eq:blowupd} are based on \cite[Theorem 20.1]{M}, 
the proofs of \eqref{eq:blowupb}, \eqref{eq:blowupe}, \eqref{eq:blowupea1}, 
\eqref{eq:blowupf}, \eqref{eq:blowuph} and \eqref{eq:blowupi} are based on Lemmas \ref{lema:creationOneBdr}, \ref{lema:creationdelamin},
\ref{lema:crackssubstrate}, \ref{lema:filamentssubstrate}, 
\ref{lem:delaminationofS} and \ref{lema:creationcrackdelamination}, respectively. Finally, the proofs of \eqref{eq:blowupc_2} and \eqref{eq:blowupga} are based on \cite[Lemma 4.4]{KP}, and the proofs of \eqref{eq:blowupc1} and \eqref{eq:blowupg} are based on \cite[Lemma 4.5]{KP} (See Table \ref{tab:my_label}).

 \begin{table}[h!]
    \centering
    \begin{tabular}{p{4.5cm}|p{5.9cm}|p{1.4cm}|p{2.9cm}}
 \textbf{Sets}& \textbf{Conditions} & \textbf{Surf. t.} & \textbf{Assertions}\\
 \hline\hline
$\partial^* {A} \setminus { \partial {S_{h,K}}}$  & $\nu_{A_k } \Hs^1 (\partial ^* A_k) \stackrel{\ast}{\rightharpoonup} \nu_{A} \Hs^1 (\partial ^* A )$& $\varphi_{\text{F}}$ & \small \cite[Theorem 20.1]{M}  \\ \hline
$\partial ^* S_{h,K} \cap \partial^* {A}$ & $\overline{R_{\nu_A}} \cap (\overline{  A_{k_n}} \setminus \Int(  S_{k_n} ) ) \cngK 
    \overline {R_{\nu_A}} \cap T_{0,\nu_A}$ &$\varphi$ &\small Lemma \ref{lema:creationOneBdr} \\
    \hline
$(\partial A\setminus \partial {S_{h,K}}) \cap  A^{(1)}$  & $\overline { Q_1} \setminus A_\kn \cngK 
    \overline {Q_1} \cap T_{0,\textbf{e}_\textbf{1}}$ &$2\varphi_{\text{F}}$ & \small\cite[Lemma 4.5]{KP} \\
    \hline
$(\partial A\setminus \partial {S_{h,K}}) \cap  A^{(0)}$  & $\overline { Q_1} \cap A_\kn \cngK 
    \overline {Q_1} \cap T_{0,\textbf{e}_\textbf{1}}$ &${ 2\varphi'}
    $ & \small\cite[Lemma 4.4]{KP} \\ 
    \hline
$\partial^* {S_{h,K}} \setminus { \partial A}$  & $\nu_{S_k } \Hs^1 (\partial ^* S_k) \stackrel{\ast}{\rightharpoonup} \nu_{S_{h,K}} \Hs^1 (\partial ^* S_{h,K} )$ & $\varphi_{\text{FS}}$ & \small\cite[Theorem 20.1]{M}  \\
\hline
$(\partial {S_{h,K}} \cap \partial A) \cap S^{(1)}_{h,K}$  & $\begin{matrix}
    \overline {Q_1} \setminus  A_\kn \cngK 
    \overline {Q_1} \cap T_{0,\textbf{e}_\textbf{1}}, \\ \overline {Q_1} \setminus  S_\kn \cngK 
    \overline {Q_1} \cap T_{0,\textbf{e}_\textbf{1}}
\end{matrix}$ & $2\varphi$& \small Lemma \ref{lema:crackssubstrate} \\ 
    \hline
$(\partial {S_{h,K}} \cap \partial A) \cap A^{(0)}$ & $
\begin{matrix}
    \overline {Q_1} \cap  A_\kn\cngK 
    \overline {Q_1} \cap T_{0,\textbf{e}_\textbf{1}} , \\ \overline {Q_1} \cap  S_\kn \cngK 
    \overline {Q_1} \cap T_{0,\textbf{e}_\textbf{1}}
\end{matrix}$ & ${ 2\varphi'}$&\small Lemma \ref{lema:filamentssubstrate} \\ \hline
$\partial {S_{h,K}} \cap \partial^* A \cap  \Sz$ & $
\begin{matrix}
    \overline {Q_1} \cap  A_\kn \cngK \overline {Q_1} \cap H_{0, \textbf{e}_\textbf{1}}, \\
    \overline {Q_1} \cap  S_\kn \cngK 
    \overline {Q_1} \cap T_{0,\textbf{e}_\textbf{1}}
\end{matrix}$ & ${ \varphi_{\mathrm{F}}} 
$ &\small Lemma \ref{lem:delaminationofS} \\ \hline
$(\partial {S_{h,K}} \setminus \partial A)  \cap S^{(1)}_{h,K} \cap A^{(1)}$&  $\overline {Q_1} \setminus  S_\kn \cngK 
    \overline {Q_1} \cap T_{0,\textbf{e}_\textbf{1}}$ &  $2 \varphi_{\text{FS}}$ &\small\cite[Lemma 4.5]{KP} \\ \hline
$(\partial {S_{h,K}} \setminus \partial A)  \cap S^{(0)}_{h,K} \cap A^{(1)}$ &  $\overline {Q_1} \cap  S_\kn \cngK 
    \overline {Q_1} \cap T_{0,\textbf{e}_\textbf{1}}$ & $2 \varphi_{\text{FS}}$ &\small\cite[Lemma 4.4]{KP} \\ \hline
$\partial ^*{S_{h,K}} \cap \partial A   \cap A^{(1)}$ &  
$ \begin{matrix}
\overline {R_{\nu_{S_{h,K}}}} \cap  S_\kn \cngK \overline {R_{\nu_{S_{h,K}}}} \cap H_{0,\nu_{S_{h,K}}}, \\ \overline {R_{\nu_{S_{h,K}}}} \setminus  A_\kn \cngK 
    \overline {R_{\nu_{S_{h,K}}}} \cap T_{0,\nu_{S_{h,K}}}    
\end{matrix} $ &  $ \varphi_{\text{F}} + \varphi$ &\small Lemma \ref{lema:creationdelamin} \\ \hline
$(\partial {S_{h,K}} \cap \partial A)  \cap S^{(0)}_{h,K} \cap A^{(1)}$ &  $
\begin{matrix}
\overline {Q_1} \cap  S_\kn \cngK \overline {Q_1} \cap T_{0,\textbf{e}_\textbf{1}}, \\ 
\overline {Q_1} \setminus  A_\kn \cngK \overline {Q_1} \cap T_{0,\textbf{e}_\textbf{1}}
\end{matrix}$ & $ 
{ 2\varphi_{\text{F}} }
$  &\small Lemma \ref{lema:creationcrackdelamination} \\
	\hline \hline
\end{tabular}
    \caption{\small Sketch of the proof of \eqref{eq:blowupa}--\eqref{eq:blowupi} for Theorem \ref{thm:lowersemicontinuityS}: in the blow-ups centered at a point of the sets listed in the first column, the corresponding conditions listed in the second column are proven to hold and  the lower bounds of the localized  surface energy are  reached with surface tensions given in the third column by means of the assertions listed in the fourth column. Note that $R_{\nu_U}$ for $U = A, S_{h,K}$  is defined as $R_{\nu_U}:= Q_1$ if $\nu_U  = \textbf{e}_i$ for $i =\textbf{1},\textbf{2}$,  or, otherwise,  $R_{\nu_U}:=(-\cos \theta_{\nu_U }, \cos \theta_{\nu_U } ) \times_{\R^2} (-\sin \theta_{\nu_U }, \sin \theta_{\nu_U } )$, where $\theta_{\nu_U}$ is the angle formed between  $T_{0,\nu_U}$ and the $x_1$-axis.}
    \label{tab:my_label}
\end{table}

		\begin{proof}[Proof of (\ref{eq:blowupa})]
	We begin by observing that by the definition of $\partial ^* A$, the continuity of $\varphi$, the Borel regularity of $y \in \partial ^* A \mapsto \varphi_{\mathrm{ F}}(y, \nu_A (y))$ and the Besicovitch Derivation Theorem  (see \cite[Theorem 1.153]{FL}), the set of points in $\partial^* {A} \setminus { \partial {S_{h,K}}}$ not satisfying the following 3 conditions:
   \begin{itemize}
	\item[(a1)] $ \nu_A \xp $ exists,
	\item[(a2)] $x$ is a Lebesgue point of $y \in  \partial^* {A} \setminus { \partial {S_{h,K}}} \mapsto \varphi_{\mathrm{ F}}(y, \nu_A (y)),$ i.e.,
	$$ \lim_{r \rightarrow 0}{\frac{1}{2r}{\int_{Q_r \cap  \partial^* {A} \setminus { \partial {S_{h,K}}}}{\abs{\varphi_{\mathrm{ F}}(y, \nu_A (y)) - \varphi_{\mathrm{ F}}(x, \nu_A \xp)}\, \, d\Hs^1(y)}}}=0, $$
	\item[(a3)] $ \frac{d \mu_0}{\, d\Hs^1 \res ( \partial^* {A} \setminus  \partial {S_{h,K}})} \xp $ exists and it is finite,
	\end{itemize}
   is $\Hs^1$-negligible. Therefore, we prove \eqref{eq:blowupa} for a fixed $x \in \partial^* {A} \setminus { \partial {S_{h,K}}}$ satisfying (a1)-(a3). Without loss of generality, we consider $x=0$ and  $\nu_A(0) = \mathbf{e}_\textbf{1}$, where we used (a1). 
   
   By \cite[Lemma 3.2-(b)]{KP} and $\tau_\B$-convergence we have that $A_k  \rightarrow A$ and $S_k  \rightarrow S$  in $L^1 (\R^2)$ and hence,  $D \mathbbm{1}_{A_k}  \stackrel{\ast}{\rightharpoonup} D \mathbbm{1}_{A}$ and $D \mathbbm{1}_{ S_k}  \stackrel{\ast}{\rightharpoonup} D \mathbbm{1}_{S}$, hence by the Structure Theorem for sets of finite perimeter  (see \cite[Theorem 5.15]{EG}), it holds that
	\begin{equation}
	    \label{eq:reshetnyakcond}
        \nu_{A_k } \Hs^1 (\partial ^* A_k) \stackrel{\ast}{\rightharpoonup} \nu_{A} \Hs^1 (\partial ^* A ) .
	\end{equation}
Furthermore, by Remark \eqref{remark:1}-(i) and again the  the $\tau_\B$-convergence we obtain  that  $\partial S_k \cngK \partial S_{h,K}$, from which it follows that for any $\eta >0$  there exists $k_\eta \in \N$ such that $\partial S_k \subset S^\eta$ for every $k \ge k_\eta$, where $S^\eta: = \{ x \in \overline{\Om}: \, \dist(x, \partial S_{h,K}) \le \eta \}.$

 We observe that from the properties of Radon measures there exists a sequence $\rho_n \searrow 0$ such that $Q_{\rho_n} \subset \subset \Om$, $\mu_0 (\partial  Q_{\rho_n} )=0$, 
 \begin{equation}
	\label{eq:beforebesicovitch89}
 \mu_0 (Q_{\rho_n}) = \lim_{k \rightarrow +\infty}{\mu_k (\overline{Q_{\rho_n}})} 
 	\end{equation}
 and
\begin{equation}
\label{eq:besicovitch89}
\frac{d \mu_0}{\, d\Hs^1 \res ( \partial^* {A} \setminus  \partial {S_{h,K}})}(0) = \lim_{n \rightarrow \infty}{\frac{\mu_0 (Q_{\rho_n})}{2 \rho_n}},
	\end{equation}
where we also used (a3) and Besicovitch   Derivation Theorem (see, e.g., \cite[Theorem 1.153]{FL}). Therefore, by \eqref{eq:beforebesicovitch89} we deduce that
\begin{equation}
\label{eq:blowupa2}
    \begin{split}
\mu_0 (Q_{\rho_n})& = \lim_{k \rightarrow +\infty}{\mu_k (\overline{Q_{\rho_n}})} \ge \liminf_{k \rightarrow \infty}{\int_{Q_{\rho_n} \cap  \partial^* {A_k} \setminus { \partial {S_k}}}{\varphi_{\mathrm{ F}}(y, \nu_{A_k})\, d\Hs^1}} \\ 
        & \ge \liminf_{k \rightarrow \infty}{\int_{Q_{\rho_n} \cap  \partial^* {A_k} \setminus S^\eta}{\varphi_{\mathrm{ F}}(y, \nu_{A_k})\, d\Hs^1}} \ge {\int_{Q_{\rho_n} \cap  \partial^* {A} \setminus S^\eta}{\varphi_{\mathrm{ F}}(y, \nu_{A})\, d\Hs^1}},
    \end{split}
\end{equation}
where in the first inequality we used the non-negativeness of $\psi_k$, in the second inequality we used the fact that $\partial S_k  \subset S^\eta$ for every $k \ge k_\eta$ and in the last inequality, by using the fact that $A_k \to A$ in $L^1(\R^2)$ and \eqref{eq:reshetnyakcond}, we apply (see, e.g.,  \cite[Theorem 20.1]{M}). Moreover, taking $\eta \to 0$ in \eqref{eq:blowupa2}  we obtain that
\begin{equation}
\label{eq:blowupa2bis}
    \begin{split}
	\mu_0 (Q_{\rho_n}) & \ge \lim_{\eta \to 0} {\int_{Q_{\rho_n} \cap  \partial^* {A} \setminus S^\eta}{\varphi_{\mathrm{ F}}(y, \nu_{A})\, d\Hs^1}} =  {\int_{Q_{\rho_n} \cap  \partial^* {A} \setminus \partial S_{h,K}}{\varphi_{\mathrm{ F}}(y, \nu_{A})\, d\Hs^1}}
    \end{split}
\end{equation}
by Lebesgue monotone convergence theorem \cite[Theorem 1.79]{FL}.
Finally, by \eqref{eq:besicovitch89} and (a2) we conclude that 
			\begin{align*}
				\frac{d \mu_0}{\, d\Hs^1 \mres ( \partial^* {A} \setminus { \partial {S_{h,K}}})}(0) &= \lim_{n \rightarrow \infty}{\frac{\mu_0(Q_{\rho_n})}{2\rho_n}} \ge \liminf_{n \rightarrow \infty}{\frac{1}{2\rho_n} \int_{Q_{\rho_n} \cap  \partial^* {A} \setminus { \partial {S_{h,K}}}}{\varphi_{\mathrm{ F}}(y, \nu_{A})\, d\Hs^1} }\\
    &= \varphi_{\mathrm{ F}}(0, \mathbf{e}_\textbf{1}).
			\end{align*}
		\end{proof}
		\begin{proof}[Proof of (\ref{eq:blowupb})]
By the definition of $\partial^* A$ and $\partial^* {S_{h,K}}$, by \cite[Proposition A.4]{KP} (applied with $K$ taken as first $\partial A$ and then $\partial {S_{h,K}}$), and by the Besicovitch Derivation Theorem the set of points $x \in \partial ^* A \cap \partial^*{S_{h,K}}$ not satisfying the following 3 conditions:
	\begin{enumerate}
	\item[(b1)] $\nu_A \xp$, $\nu_{S_{h,K}} \xp$ exist and, either $\nu_A \xp = \nu_{S_{h,K}} \xp$ or $\nu_A \xp = - \nu_{S_{h,K}} \xp$,
	\item[(b2)]  
for every open rectangle $R$ containing $x$ with sides parallel or perpendicular to $\textbf{e}_\textbf{1}$ we have that  $\overline{R} \cap  \partial {\sigma}_{\rho, x} ( A) \cngK \overline{R} \cap T_{x,\nu_A(x)}$ and $ \overline{R} \cap  \partial {\sigma}_{\rho, x} ( {S_{h,K}}) \cngK  \overline{R} \cap T_{x,\nu_A(x)}$ as $\rho \to 0$, where $T_{x,\nu_A(x)}$ is the approximate tangent line at $x$ of $\partial A$ (or of $\partial {S_{h,K}}$), 
	\item[(b3)] $ \frac{d \mu_0}{\, d\Hs^1 \mres (\partial ^* S_{h,K} \cap \partial ^* A )} \xp$ exists and it is finite,
	\end{enumerate}   
   is $\Hs^1$-negligible. Therefore, we prove \eqref{eq:blowupb} for any fixed $x \in  \partial ^* A \cap \partial^*{S_{h,K}}$ satisfying (b1)-(b3). Without loss of generality we assume that $x=0$ and we denote $T_0 = T_{0,\nu_A(0)}$. Furthermore, by using (b1) we choose in (b2) the rectangle $R_{\nu_A}:= Q_1$ if $\nu_A(0)  = \textbf{e}_i$ for $i =\textbf{1},\textbf{2}$,  or $R_{\nu_A} :=(-\cos \theta_{\nu_A }, \cos \theta_{\nu_A } ) \times_{\R^2} (-\sin \theta_{\nu_A }, \sin \theta_{\nu_A } )$, where $\theta_{\nu_A }$ is the angle formed between the tangent line $T_0$ and the $x_1$-axis, otherwise. For any $\rho >0$, we write $R_\rho := \rho R_{\nu_A}$.
      
In view of the definition of $R_{\nu_A}$ and again by using also the Besicovitch Derivation Theorem (see \cite[Theorem 1.153]{FL}) there exists a subsequence $\rho_n \searrow 0$ such that 
\begin{equation}
	\label{eq:blowupb1}
	\mu_0 (\partial  R_{\rho_n}  
    )=0, \qquad  \lim_{k \rightarrow +\infty}{\mu_k (\overline{ R_{\rho_n}})} = \mu_0 ( R_{\rho_n})
\end{equation} 
and
\begin{equation}
\label{eq:blowupb2}
	\frac{d \mu_0}{\, d\Hs^1 \mres (\partial ^* S_{h,K} \cap \partial ^* A )}(0) = \lim_{n \rightarrow \infty}{\frac{\mu_0 ( R_{\rho_n})}{2 \rho_n}}.
\end{equation}
We now claim that 
\begin{equation}
\label{eq:diag1}
\sdist (\cdot, \partial \sigma_{\rho_n} ( A )) \rightarrow \sdist (\cdot, \partial H_0)\quad\text{and}\quad\sdist (\cdot,\partial \sigma_{\rho_n} ( {S_{h,K}})) \rightarrow \sdist (\cdot, \partial H_0)
\end{equation}
uniformly in $\overline{R_{\nu_A}}$ as $n\to\infty$, where $H_0$ is the half space centered in $0$ with respect to the vector $\nu_A$. To prove the claim we can for example observe that by \cite[Proposition A.4]{KP} we have (not only (b2), but also) that $\overline{Q_r} \cap  \partial {\sigma}_{\rho, x} ( A) \cngK \overline{Q_r} \cap T_x \quad \text{and} \quad  \overline{Q_r} \cap  \partial {\sigma}_{\rho, x} ( {S_{h,K}}) \cngK  \overline{Q_r} \cap T_x \quad \text{as }\rho \to 0$ for any square $Q_{r}$ such that $R_{\nu_A} \subset Q_r$
and hence, by Proposition \ref{prop:aux1}-(c) applied to $Q_{r}$, $\sdist (\cdot, \partial \sigma_{\rho_n} ( A )) \rightarrow \sdist (\cdot, \partial H_0)$ and $\sdist (\cdot,\partial \sigma_{\rho_n} ( {S_{h,K}})) \rightarrow \sdist (\cdot, \partial H_0)$ uniformly in ${ \overline{Q_{r}}}\supset R_{\nu_A}$.  

Furthermore, from the $\tau_\B$-convergence it follows that 
\begin{equation}
\label{eq:diag2}
\sdist (\cdot, \partial A_k ) \rightarrow \sdist(\cdot, \partial A )\quad\text{and}\quad\sdist (\cdot, \partial {S_k} ) \rightarrow \sdist(\cdot, \partial {S_{h,K}} )
\end{equation}
uniformly in $\overline{R_{\nu_A}}$ as $k \rightarrow \infty$ and hence, by \eqref{eq:diag1} and \eqref{eq:diag2}, a standard diagonalization argument yields that there exists a subsequence $\left\{ (A_{k_n}, h_{k_n},K_{k_n} ) \right\}$ such that
			\begin{equation}
				\label{eq:blowupb3}
				\sdist (\cdot , \partial \sigma_{\rho_n} ( A_{k_n})) \rightarrow \sdist (\cdot, \partial H_0), \quad \sdist (\cdot , \partial \sigma_{\rho_n} ( S_{k_n} )) \rightarrow \sdist (\cdot, \partial H_0)
			\end{equation}
uniformly in $\overline{R_{\nu_A}}$ as $n \to \infty$ and by \eqref{eq:blowupb1} such that 
\begin{equation}
\label{eq:blowupb4}
	\mu_\kn (\overline{R_{\rho_n}}) \le \mu_0 (R_{\rho_n}) + \rho_n^2,
\end{equation}
for every $n \in \N$. We also observe that 
	\begin{equation}
	\label{eq:blowupb6}
		\begin{split}
		\frac{d \mu_0}{\, d\Hs^1 \mres (\partial ^* S_{h,K} \cap  \partial ^* A )}(0)  & \ge  \limsup_{n \to \infty}{\frac{\mu_\kn (\overline{R_{\rho_n}} )}{2 \rho_n}  } \\
  & \ge c_1 \limsup_{n \to \infty}{ \frac{\Hs^1 (\overline{R_{\rho_n}} \cap \partial A_{k_n}) + \Hs^1 (\overline{R_{\rho_n}} \cap \partial S_{k_n} \setminus \partial A_{k_n}) }{2 \rho_n}  }, 
		\end{split}
		\end{equation}
where in the first inequality we used \eqref{cmpct1}-\eqref{eq:compactintersectionlast}, as $A_{k_n}$ and $S_\kn$ are sets of finite perimeter, and \eqref{eq:H1} while in the second inequality we used \eqref{eq:blowupb2} and \eqref{eq:blowupb4}.


Now, we claim that
\begin{equation}
	\label{eq:blowupb5}
	\overline {R_{\nu_A}} \cap (\overline{ \sigma_{\rho_n} ( A_{k_n})} \setminus \Int( \sigma_{\rho_n} ( S_{k_n}) ) ) \cngK 
    \overline {R_{\nu_A}} \cap T_0.
\end{equation}
We proceed by contradiction, let $x_n \in \overline {R_{\nu_A}} \cap \overline{ \sigma_{\rho_n} ( A_{k_n})} \setminus \Int( \sigma_{\rho_n} ( S_{k_n}) )$ such that $x_n \to x$ and assume that $x \in \Int(\overline{R_{\nu_A}} \cap H_0 )$ or $x \in R_{\nu_A} \setminus  H_0 $. In the first case, there exists $\ep >0$ such that $\sdist(x, \partial H_0) = - \ep$. By \eqref{eq:blowupb3} 
we observe that $\sdist (x,  \partial \sigma_{\rho_n} ( S_{k_n} )) \to -\ep$ uniformly in $R_{\nu_A}$ as $n \to \infty$. Furthermore, for $n$ large enough, $x_n \in B_{{\ep}/{2}}(x)$ and it follows that $\sdist(x_n,  \partial \sigma_{\rho_n} ( S_{k_n} ))$ is negative. Therefore, for $n$ large enough, $x_n \in \Int( \sigma_{\rho_n} ( S_{k_n} )))$ 
which is an absurd. Analogously if $x \in R_{\nu_A} \setminus  H_0 $ we have that $\sdist(x, \partial H_0) =  \ep$ and by \eqref{eq:blowupb3}, $\sdist (x,  \partial \sigma_{\rho_n} ( A_{k_n} )) \to \ep$ uniformly in $R_{\nu_A}$ as $n \to \infty$, similarly as before, we can conclude, for $n$ large enough, that $x_n \in R_{\nu_A} \setminus \overline{ \sigma_{\rho_n} ( A_{k_n})}$, which is an absurd. 
			
Now, let $x \in \overline{ R_{\nu_A}} \cap T_0$. By Kuratowski convergence there exists $\{x_n \} \subset R_{\nu_A} \cap  \partial \sigma_{\rho_n} (  A_{k_n})$ such that $x_n \to x$. We see that for all $n \in \N$, $x_n \in R_{\nu_A} \cap \overline{ \sigma_{\rho_n} ( A_{k_n})} $ and $x_n \notin   R_{\nu_A} \setminus  \Int( \sigma_{\rho_n} ( S_{k_n}) )$, if not, there exists $n' \in \N$ such that $x_{n'} \in R_{\nu_A}\cap  \Int( \sigma_{\rho_{n'}} ( S_{k_{n'}}) ) \subset \subset R_{\nu_A} \cap \Int(\sigma_{\rho_{n'}} ( A_{k_{n'}})) $, which is an absurd.
			
Since $\left\{ (A_\kn, S_{h_\kn,K_\kn}) \right\} \subset \B_\mathbf{m}$, we know that 
$(\sigma_{\rho_n}({A_{k_n}}), S_{(1/{\rho_n})h_{k_n} (\rho_n\cdot), \sigma_{\rho_n}({K_{k_n}})}) \in \B_{\textbf{m}}(\sigma_{\rho_n}(\Omega))$. 
In view of \eqref{eq:blowupb3} and  \eqref{eq:blowupb5}  by applying Lemma \ref{lema:creationOneBdr} to $(\sigma_{\rho_n}({A_{k_n}}), S_{(1/{\rho_n})h_{k_n}(\rho_n\cdot), \sigma_{\rho_n}({K_{k_n}})})$ and $R_{\nu_A}$, 
with $\phi_{\alpha} (\cdot) = \varphi_\alpha (0, \cdot) $ for $\alpha = \text{F,S,FS}$, and by fixing $\ep \in (0,1)$, there exists $n^1_\ep\in \N$ such that for every $n \ge n^1_\ep$,
\begin{equation}
	\label{eq:blowupb7}
	\begin{split}
	 {\mathcal{S}}_L (\sigma_{\rho_n}({A_{k_n}}),S_{ (1/{\rho_n})h_{k_n}(\rho_n\cdot), \sigma_{\rho_n}({K_{k_n}})}, R_{\nu_A}) \ge 
   \int_{ T_0 \cap \overline{R_{\nu_A}}}{\varphi (0,\nu_{T_0}) \, \, d\Hs^1} - \ep \ge 2 \varphi (0,\nu_{T_0}) - \ep.
\end{split}
\end{equation}

Moreover, by the uniform continuity of the Finsler norm $\varphi_\alpha  $ for $\alpha = \text{F,S,FS}$  
there exists $n^2_\ep\ge n^1_\epsilon$ such that 
\begin{equation}
	\label{eq:inequalitymu}
	\begin{split}
		\mu_\kn  (\overline{R_{\rho_n}}) 
		& \ge \mu_\kn  (R_{\rho_n}) \\
  &\ge \int_{R_{\rho_n} \cap \partial^* A_{k_n} \setminus \partial S_{k_n}}{\varphi_{\mathrm{ F}}(0, \nu_{A_{k_n}}(0))\, \, d\Hs^1} \\
  & \quad +   \int_{R_{\rho_n} \cap (\partial A_{k_n} \setminus \partial S_{k_n})\cap \left(A_{k_n}^{(0)} \cup A_{k_n}^{(1)}\right)}{2 \varphi_{\mathrm{ F}}(0, \nu_{A_{k_n}} (0)  )\, d\Hs^1} \\
		&\quad +  \int_{R_{\rho_n}  \cap \partial ^*{S_{k_n}}\cap \partial^* {A_{k_n}}}{\varphi (0, \nu_{A_{k_n}(0)} ) \, d\Hs^1} \\
  & \quad + \int_{R_{\rho_n} \cap (\partial^* {S_{k_n}} \setminus \partial {A_{k_n}})\cap A_{k_n}^{(1)}}{ \varphi_{\mathrm{ FS}} (0, \nu_{S_{k_n}}(0) )\, d\Hs^1} \\
		& \quad + \int_{R_{\rho_n} \cap (\partial {S_{k_n}} \cap \partial {A_{k_n}}) \cap \left(S_{k_n}^{(1)} \cup A_{k_n}^{(0)}\right)}{2 \varphi (0, \nu_{A_{k_n}}(0) )\, d\Hs^1} \\  
        & \quad + \int_{R_{\rho_n} \cap \partial {S_{k_n}} \cap \partial ^* {A_{k_n}} \cap S_{k_n}^{(0)}}{  (\varphi + \varphi_{\mathrm{ FS}}) (0, \nu_{A_{k_n}} (0) ) \, d\Hs^1}  \\
		& \quad + \int_{R_{\rho_n} \cap (\partial {S_{k_n}} \setminus \partial {A_{k_n}}) \cap \left(S_{k_n}^{(1)} \cup S_{k_n}^{(0)}\right) \cap A_{k_n}^{(1)} }{ 2 \varphi_{\mathrm{ FS}} (0, \nu_{S_{k_n}}(0) ) \, d\Hs^1}  \\
		& \quad + \int_{R_{\rho_n}\cap  \partial^* {S_{k_n}} \cap \partial {A_{k_n}} \cap A_{k_n}^{(1)}}{(\varphi_{\mathrm{ F}} +\varphi) (0, \nu_{A_{k_n}}(0) ) \, d\Hs^1} \\
		& \quad + \int_{R_{\rho_n} \cap (\partial {S_{k_n}} \cap \partial {A_{k_n}} ) \cap S_{k_n}^{(0)} \cap A_{k_n}^{(1)} }{  (\varphi_{\mathrm{ F}} +   \varphi +   \varphi_{\mathrm{ FS}}) (0, \nu_{S_{k_n}} (0) ) \, d\Hs^1} \\
		& \quad  - \ep \left( \Hs^1 (\overline{R_{\rho_n}} \cap \partial A_{k_n}) + \Hs^1 (\overline{R_{\rho_n}} \cap \partial S_{k_n} \setminus \partial A_{k_n})  \right) ,\\
  & =: {\mathcal{S}}_L ({A_{k_n}}, S_{h_{k_n}, {K_{k_n}}}, R_{\rho_n})  - \ep \left( \Hs^1 (\overline{R_{\rho_n}} \cap \partial A_{k_n}) + \Hs^1 (\overline{R_{\rho_n}} \cap \partial S_{k_n} \setminus \partial A_{k_n})  \right) \\
	\end{split}	
\end{equation}
for every $n \ge n^2_\ep$, where in the first inequality we used the definition of $\mu_\kn$ and in the second inequality \eqref{eq:continuitynorms}, and hence, 
	\begin{equation}
 \label{eq:blowupb71}
	\begin{split}
	    \mu_\kn (\overline{R_{\rho_n}})  &\ge \rho_n \left( {\mathcal{S}}_L (\sigma_{\rho_n} ({A_{k_n}}), S_{(1/{\rho_n})h_{k_n}(\rho_n\cdot), \sigma_{\rho_n}({K_{k_n}})}, R_{\nu_A})  \right) \\
     & \qquad - \ep \left( \Hs^1 (\overline{R_{\rho_n}} \cap \partial A_{k_n}) + \Hs^1 (\overline{R_{\rho_n}} \cap \partial S_{k_n} \setminus \partial A_{k_n})  \right)   \\
     & \ge  2 \rho_n \varphi (0,\nu_{T_0}) - \ep \rho_n - \ep  \left( \Hs^1 (\overline{R_{\rho_n}} \cap \partial A_{k_n}) + \Hs^1 (\overline{R_{\rho_n}} \cap \partial S_{k_n} \setminus \partial A_{k_n})  \right),
	\end{split}
    \end{equation}
    where in the first inequality we used \eqref{eq:inequalitymu} and the properties of the blow up map, 
    and in the last inequality we used \eqref{eq:blowupb7}. 
    Finally, we conclude that
	\begin{equation}
    \label{eq:blowupbfinal}
	    \begin{split}
	        \frac{d \mu_0}{\, d\Hs^1 \mres ( \partial ^* S_{h,K} \cap \partial ^* A)}(0) &\ge \liminf_{n \to \infty}{\frac{\mu_\kn (\overline{R_{\rho_n}})}{2 \rho_n}} \\
	       & \ge \varphi (0,\nu_{T_0}) - \frac{ \ep}{2} \\
        & \qquad \qquad - \ep  \limsup_{n \to \infty}{\frac{\Hs^1 (\overline{R_{\rho_n}} \cap \partial A_{k_n}) + \Hs^1 (\overline{R_{\rho_n}} \cap \partial S_{k_n} \setminus \partial A_{k_n}) }{2 \rho_n}  } \\
	       & \ge \varphi (0,\nu_{T_0}) - \frac{ \ep}{2} - \frac{\ep}{c_1}  \frac{d \mu_0}{\, d\Hs^1 \mres ( \partial ^* S_{h,K} \cap \partial ^* A)}(0),
	    \end{split}
	\end{equation}
   where in the first inequality we used \eqref{eq:blowupb4}, in the second inequality we used  \eqref{eq:blowupb71} and in the last inequality we used \eqref{eq:blowupb6}. By (b3) and taking $\ep \to 0 ^+$, in the inequality above, we deduce \eqref{eq:blowupb}.
		\end{proof}
		
  \begin{proof}[Proof of \eqref{eq:blowupc1}] 
By the $\Hs^1$-rectifiability of $\partial A$, by \cite[Proposition A.4]{KP} (applied with $K$ taken as $ \partial  A$),  and by the Besicovitch Derivation theorem, the set of points $x \in (\partial A \setminus \partial S_{h,K}) \cap \Ao$ not satisfying the following 4 conditions:
	\begin{enumerate}
	\item[(c1)] $\theta^*(\partial A, x) = \theta_* (\partial A, x) = 1$,
	\item[(c2)] $\nu_A(x)$ exists,
	\item[(c3)] $\overline{Q_{1,\nu_{A}(x)}} \cap  \partial {\sigma}_{\rho, x} ( A) \cngK \overline{Q_{1,\nu_{A}(x)}} \cap T_{x, \nu_A(x)}$ as $\rho \to 0$, 
	\item[(c4)] $ \frac{d \mu_0}{\, d\Hs^1 \mres ((\partial A \setminus \partial {S_{h,K}}) \cap  A^{(1)})} \xp$ exists and it is finite,
			\end{enumerate}
is $\Hs^1$-negligible. Therefore, we prove \eqref{eq:blowupc1} for any fixed $x \in (\partial A \setminus \partial S_{h,K}) \cap \Ao$ satisfying (c1)-(c4). Without loss of generality we assume that $x=0$ and $\nu_A (0)  = \mathbf{e}_\mathbf{1}$, and we use the notation $T_0:=T_{0, \nu_A(0)}$. Again by the Besicovitch Derivation Theorem there exists a subsequence $\rho_n \searrow 0$ such that
	\begin{equation}
	\label{eq:blowupc2}
		\mu_0 (\partial Q_{\rho_n} )=0, \quad \lim_{k \rightarrow +\infty}{\mu_k (\overline{Q_{\rho_n}})} = \mu_0 (Q_{\rho_n})
	\end{equation}
and
	\begin{equation}
	\label{eq:blowupc3}
		\frac{d \mu_0}{\, d\Hs^1 \mres ((\partial A \setminus \partial {S_{h,K}}) \cap  A^{(1)})}(0) = \lim_{n \rightarrow \infty}{\frac{\mu_0 (Q_{\rho_n})}{2 \rho_n}}.
	\end{equation}
 By (c3) and applying Proposition \ref{prop:aux1}-(a) to $A$ we have that 
 \begin{equation}
 \label{eq:diagc1}
     \sdist (\cdot, \partial \sigma_{\rho_n} ( A )) \rightarrow -\dist (\cdot,  T_0)
 \end{equation}
 uniformly in $\overline{Q_1}$ as $n \to \infty$. Furthermore, from the $\tau_\B$-convergence it follows that 
 \begin{equation}
 \label{eq:diagc2}
     \sdist (\cdot, \partial A_k ) \rightarrow \sdist(\cdot, \partial A )
 \end{equation}
 uniformly in $\overline{Q_{1}}$ as $k \to \infty$ and hence, by \eqref{eq:diagc1} and \eqref{eq:diagc2}, a standard diagonalization argument yields that there exists a subsequence $\left\{ (
	A_{k_n},S_{ h_{k_n}, K_\kn} ) \right\}$ such that
	\begin{equation}
		\label{eq:blowupc4}
		\sdist (\cdot , \partial \sigma_{\rho_n} ( A_{k_n})) \rightarrow -\dist (\cdot, T_0).
	\end{equation}
uniformly in $\overline{Q_1}$ as $n \to \infty$ and, by also using \eqref{eq:blowupc2}, such that
	\begin{equation}
	\label{eq:blowupc5}
		\mu_\kn (\overline{Q_{\rho_n}}) \le \mu_0 (Q_{\rho_n}) + \rho_n^2,
	\end{equation}
	for every $n\in \N$. By arguing as in \eqref{eq:blowupb6}, we infer that
	\begin{equation}
	\label{eq:blowup6}
	\limsup_{n \to \infty}{\frac{\Hs^1 (\overline{Q_{\rho_n}} \cap \partial A_{k_n}) + \Hs^1 (\overline{Q_{\rho_n}} \cap \partial S_{k_n} \setminus \partial A_{k_n}) }{2 \rho_n}  } \le  c_1^{-1}  \frac{d \mu_0}{\, d\Hs^1 \mres (((\partial A \setminus \partial {S_{h,K}}) \cap  A^{(1)})}(0) .
	\end{equation}			
By \eqref{eq:blowupc4} and by applying Lemma \ref{lem:A2}, we have that $ \overline{Q_1} \setminus \sigma_{\rho_n}{(A_{k_n})} \cngK  \overline{Q_1} \cap T_0$ as $n \to \infty$. Since the number of connected components of $\partial \sigma_{\rho_n} (  A_\kn)$ lying inside of $Q_1$ does not exceed $m_1$, by \cite[Lemma 4.5]{KP} (applied by taking $m_0,\delta,\phi$ in the notation of  \cite[Lemma 4.5]{KP} as $m_1 $, $\ep$, and $\varphi_{\mathrm{ F}}(0, \cdot)$, respectively) implies that there exists $n^1_\ep \ge n_\ep$ such that for every $n \ge n^1_\ep$,
\begin{equation}
	\label{eq:blowupc7}
	\begin{split}
	 &\int_{Q_1 \cap \brbak}{\hspace{-0.45cm}\varphi_{\mathrm{ F}}(0, \nubak(0) )\, \, d\Hs^1} \\
	  & \qquad + \int_{Q_1 \cap \bbak \cap (\bzda \cup \boda)}{\hspace{-0.45cm}2 \varphi_{\mathrm{ F}}(0, \nubak(0) )\, d\Hs^1} \\
& \ge 2 \int_{\overline{Q_1} \cap T_0}{\varphi_{\mathrm{ F}}(0,\mathbf{e}_\textbf{1}) \, \, d\Hs^1} - \ep = 4 \varphi_{\mathrm{ F}}(0,\mathbf{e}_\textbf{1}) - \ep.
	\end{split}
\end{equation}
In view of Remark \ref{remark:1}-(iii) and by $\tau_\B$-convergence, there exists a ball $B_{r(0)}(0)$ and $n^2_\ep \ge n^1_\ep$ such that 
$Q_{\rho_n} \cap \partial S_{k_n}\subset B_{r(0)}(0) \cap \partial S_{k_n} = \emptyset$ for any $n \ge  n^2_\ep$, and thus,
	\begin{equation}
	\label{eq:blowupc8}
		\emptyset = Q_{\rho_n} \cap \partial S_{k_n} = \rho_n (Q_1 \cap \partial \sigma_{\rho_n} (S_{k_n}) ).
	\end{equation}
Therefore, by \eqref{eq:blowupc7} and \eqref{eq:blowupc8}, we obtain that
\begin{equation}
	\label{eq:blowupc9}
    \begin{split}
	& \int_{Q_1 \cap (\brbak \setminus \partial \sigma_{\rho_n}( S_{k_n}))}{\varphi_{\mathrm{ F}}(0, \nubak )\, \, d\Hs^1} \\ 
    & \quad + \int_{Q_1 \cap (\bbak \setminus \bbsk) \cap (\bzda \cup \boda)}{2 \varphi_{\mathrm{ F}}(0, \nubak (0))\, d\Hs^1} \\
    & \ge 
    4 \varphi_{\mathrm{ F}}(0,\mathbf{e}_2) - \ep.
    \end{split}
	\end{equation}		
Furthermore, by \eqref{eq:continuitynorms} it follows that there exists $n^3_\ep \ge n^2_\ep$ such that 
\begin{align}
	\mu_\kn  (\overline{Q_{\rho_n}}) &\ge  \int_{Q_{\rho_n} \cap  \partial^* {A_{k_n}} \setminus { \partial {S_{k_n}}}}{\varphi_{\mathrm{ F}}(0, \nu_{A_{k_n}} )\, d\Hs^1} \nonumber \\
 & \quad +   \int_{Q_{\rho_n} \cap (\partial A_{k_n} \setminus \partial S_{k_n})\cap (A_{k_n}^{(0)} \cup A_{k_n}^{(1)})}{2 \varphi_{\mathrm{ F}}(0, \nu_{A_{k_n}}  )\, d\Hs^1} \nonumber \\
	& \quad  - \ep \left( \Hs^1 (\overline{Q_{\rho_n}} \cap \partial A_{k_n}) + \Hs^1 (\overline{Q_{\rho_n}} \cap \partial S_{k_n} \setminus \partial A_{k_n})  \right) \nonumber \\
	& = \rho_n \left( \int_{Q_1 \cap (\brbak \setminus \partial \sigma_{\rho_n}( S_{k_n}))}{\varphi_{\mathrm{ F}}(0, \nubak )\, d\Hs^1} \right. \nonumber \\ 
	& \quad +\left. \int_{Q_1 \cap (\bbak \setminus \bbsk)\cap (\bzda \cup \boda)}{2 \varphi_{\mathrm{ F}}(0, \nubak )\, d\Hs^1} \right) \nonumber \\ 
	& \quad - \ep \left( \Hs^1 (\overline{Q_{\rho_n}} \cap \partial A_{k_n}) + \Hs^1 (\overline{Q_{\rho_n}} \cap \partial S_{k_n} \setminus \partial A_{k_n})  \right) \nonumber \\
	& \ge  4 \rho_n \varphi_{\mathrm{ F}}(0,\mathbf{e}_\textbf{1}) - \ep \rho_n - \ep \left( \Hs^1 (\overline{Q_{\rho_n}} \cap \partial A_{k_n}) + \Hs^1 (\overline{Q_{\rho_n}} \cap \partial S_{k_n} \setminus \partial A_{k_n})  \right), \label{eq:blowupc10}
	\end{align}
for every $n \ge n^3_\ep$, where in the first inequality we argued as in \eqref{eq:inequalitymu} (with $Q_{\rho_n}$ instead of $R_{\rho_n}$) and we used the non-negativeness of $\psi_\kn$, in the equality we used properties of the blow up map, and in the second inequality we used \eqref{eq:blowupc9}. Finally, by \eqref{eq:blowupc3}, \eqref{eq:blowupc5} and \eqref{eq:blowupc10} and by repeating the same arguments of \eqref{eq:blowupbfinal}, we deduce that
			\begin{align*}
				\frac{d \mu_0}{\, d\Hs^1 \mres ((\partial A \setminus \partial {S_{h,K}}) \cap  A^{(1)})}(0) 
				& \ge2\varphi_{\mathrm{ F}}(0, \mathbf{e}_\mathbf{1}) - \frac{ \ep}{2} - \frac \ep {c_1}  \frac{d \mu_0}{\, d\Hs^1 \mres ((\partial A \setminus \partial {S_{h,K}}) \cap  A^{(1)})}(0) .
			\end{align*}
By (c4) and taking $\ep \to 0 ^+$, in the inequality above, we deduce \eqref{eq:blowupc1}.
		\end{proof}
		
  \begin{proof}[Proof of \eqref{eq:blowupc_2}] Since $\varphi' \le \varphi_{\mathrm{F}}$, we repeat the same arguments of the proof of \eqref{eq:blowupc1} by using Proposition \ref{prop:aux1}-(b) and \cite[Lemma 4.4]{KP} instead of Proposition \ref{prop:aux1}-(a) and \cite[Lemma 4.5]{KP}, respectively.
		\end{proof}
		
		\begin{proof}[Proof of \eqref{eq:blowupd}] 
We observe that \eqref{eq:blowupd} follows from the same arguments used in  \eqref{eq:blowupa}, which are based on \cite[Theorem 20.1]{M}, by ``interchanging the roles'' of $A_k,A$ with $S_k, S_{h,K}$.
		\end{proof}
  
	\begin{proof}[Proof of \eqref{eq:blowupe}] By the $\Hs^1$-rectifiability of $\partial A$ and $\partial {S_{h,K}}$, by \cite[Proposition A.4]{KP} (applied with $K$ taken as first $\partial A$ and then $\partial S_{h,K}$), and  by the Besicovitch Derivation Theorem the set of points  $x \in  (\partial {S_{h,K}} \cap \partial A) \cap S^{(1)}_{h,K}$ not satisfying the following 4 conditions: 
	\begin{enumerate}
	\item[(f1)] $\theta^* (\partial A,x) = \theta_* (\partial A , x) = \theta^* (\partial {S_{h,K}},x) = \theta_* (\partial {S_{h,K}} , x)= 1 $,
	\item[(f2)] $\nu_A \xp$, $\nu_{S_{h,K}} \xp$ exist and either $\nu_A \xp = \nu_{S_{h,K}} \xp$ nor $\nu_A \xp = -\nu_{S_{h,K}} \xp$,
	\item[(f3)] $\overline{Q_{1, \nu_A\xp} (x)} \cap  \partial {\sigma}_{\rho, x} ( A) \cngK \overline{Q_{1, \nu_A\xp} (x)} \cap T_{x, \nu_A\xp}$ and $\overline{Q_{1, \nu_A\xp} (x)} \cap  \partial {\sigma}_{\rho, x} ( S_{h,K}) \cngK \overline{Q_{1, \nu_A\xp} (x)} \cap T_{x, \nu_A\xp}$, 
	\item[(f4)] $ \frac{d \mu_0}{\, d\Hs^1 \mres ( (\partial {S_{h,K}} \cap \partial A) \cap S^{(1)}_{h,K})} \xp$ exists and it is finite.
	\end{enumerate}
	is $\Hs^1$-negligible. Therefore we prove \eqref{eq:blowupe} for any fixed $x \in  (\partial {S_{h,K}} \cap \partial A) \cap S^{(1)}_{h,K}$ satisfying 
   (f1)-(f4). By (f2) and without loss of generality we assume that $x=0$ and $\nu_A (0) = \textbf{e}_\textbf{1}$, and we denote $T_0 := T_{0, \nu_A(0)}$. 
   Again by the Besicovitch Derivation Theorem there exists a subsequence $\rho_n \searrow 0$ such that
	\begin{equation}
	\label{eq:blowupf1}
	\mu_0 (\partial Q_{\rho_n} )=0, \quad  \lim_{k \rightarrow +\infty}{\mu_k (\overline{Q_{\rho_n}})} = \mu_0 (Q_{\rho_n})
	\end{equation}
	and
	\begin{equation}
	\label{eq:blowupf2}
	\frac{d \mu_0}{\, d\Hs^1 \mres ((\partial {S_{h,K}} \cap \partial A) \cap S^{(1)}_{h,K})}(0) = \lim_{n \rightarrow \infty}{\frac{\mu_0 (Q_{\rho_n})}{2 \rho_n}}.
	\end{equation}
By (f3) and applying Proposition \ref{prop:aux1}-(a) to $A$ and $S_{h,K}$ we have that 
\begin{equation}
\label{eq:diagf1}
\sdist (\cdot, \partial \sigma_{\rho_n} ( A )) \rightarrow - \dist (\cdot, T_0)\quad\text{and}\quad\sdist (\cdot,\partial \sigma_{\rho_n} ( {S_{h,K}})) \rightarrow -\sdist (\cdot, T_0)
\end{equation}
uniformly in $\overline{Q_1}$ as $n \to \infty$. 
Furthermore, from the $\tau_\B$-convergence it follows that
\begin{equation}
\label{eq:diagf2}
\sdist (\cdot, \partial A_k ) \rightarrow \sdist(\cdot, \partial A )\quad\text{and}\quad\sdist (\cdot, \partial {S_k} ) \rightarrow \sdist(\cdot, \partial {S_{h,K}} )
\end{equation}
uniformly in $\overline{Q_1}$ as $k \to \infty$ and hence, by \eqref{eq:diagf1} and \eqref{eq:diagf2}, a standard diagonalization argument yields that there exists a subsequence  
$\left\{ (A_{k_n}, h_{k_n}, K_{k_n}  ) \right\}$ such that
\begin{equation}
\label{eq:blowupf3}
	\sdist (\cdot , \partial \sigma_{\rho_n} ( A_{k_n})) \rightarrow -\dist (\cdot, T_0) \quad \textrm{and} \quad \sdist (\cdot , \partial \sigma_{\rho_n} ( S_{k_n} )) \rightarrow -\dist (\cdot, T_0)
\end{equation}
uniformly in $\overline{Q_1}$ as $n \to \infty$ and by \eqref{eq:blowupf1} such that
\begin{equation}
	\label{eq:blowupf4}
	\mu_\kn (\overline{Q_{\rho_n}}) \le \mu_0 (Q_{\rho_n}) + \rho_n^2,
\end{equation}
for every $n \in \N$. By arguing as in \eqref{eq:blowupb6} we deduce that
\begin{equation}
	\label{eeq:blowupf5}
	\begin{split}
	\limsup_{n \to \infty}{\frac{\Hs^1 (\overline{Q_{\rho_n}} \cap \partial A_{k_n}) + \Hs^1 (\overline{Q_{\rho_n}} \cap \partial S_{k_n} \setminus \partial A_{k_n}) }{2 \rho_n}  }  \le  c_1^{-1}  \frac{d \mu_0}{\, d\Hs^1 \mres ((\partial {S_{h,K}} \cap \partial A) \cap S^{(1)}_{h,K})}(0) .
				\end{split}
			\end{equation}

By \eqref{eq:blowupf3} and by applying Lemma \ref{lem:A2} we deduce that 
\begin{equation}
    \label{eq:eq:blowupf4bis}
    {\overline{Q_1}} \setminus  \sigma_{\rho_n} (A_{k_n}) \cngK  \Uco \cap T_0 \quad \text{and} \quad{\overline{Q_1}} \setminus  \sigma_{\rho_n} (S_{k_n}) \cngK \Uco \cap T_0.
\end{equation} 
Since $\left\{ (A_\kn,S_{ h_\kn,K_\kn}) \right\} \subset \B_\mathbf{m}$, we know that $(\sigma_{\rho_n}({A_{k_n}}), S_{(1/{\rho_n})h_{k_n} (\rho_n\cdot), \sigma_{\rho_n}({K_{k_n}})}) \in \B_{\textbf{m}}(\sigma_{\rho_n}(\Omega))$. 
By \eqref{eq:eq:blowupf4bis} and by applying Lemma \ref{lema:crackssubstrate} to $(\sigma_{\rho_n}({A_{k_n}}), S_{(1/{\rho_n})h_{k_n}(\rho_n\cdot), \sigma_{\rho_n}({K_{k_n}})})$ and $Q_1$, with $\phi_{\alpha} (\cdot) = \varphi_\alpha (0, \cdot) $ for $\alpha = \text{F,S,FS}$, there exists $n^1_\ep \in \N$ such that for every $n \ge n^1_\ep$,
	\begin{equation}
	\label{eq:blowupf6bis1}
	\begin{split}
	{\mathcal{S}}_L (\sigma_{\rho_n}({A_{k_n}}), S_{(1/{\rho_n})h_{k_n}(\rho_n\cdot), \sigma_{\rho_n}({K_{k_n}})}, Q_1)& 
	 \ge 2\int_{ \overline{Q_1} \cap T_0}{\varphi (0,\textbf{e}_\textbf{1}) \, \, d\Hs^1} - \ep \ge 4 \varphi (0,\textbf{e}_\textbf{1}) - \ep.
				\end{split}
			\end{equation}
By \eqref{eq:blowupf6bis1}, by the uniform continuity in \eqref{eq:continuitynorms} and by repeating the same arguments of \eqref{eq:blowupb71} we obtain that there exists $n^2_\ep \ge n^1_\ep$ such that 
	\begin{equation}
	\label{eq:blowupf7}
	\mu_\kn (\overline{Q_{\rho_n}}) \ge 4 \varphi (0,\mathbf{e}_\textbf{1}) - \ep \rho_n - \ep \left( \Hs^1 (\overline{Q_{\rho_n}} \cap \partial A_{k_n}) + \Hs^1 (\overline{Q_{\rho_n}} \cap \partial S_{k_n} \setminus \partial A_{k_n})  \right),
	\end{equation}
 for every $n \ge n^2_\ep$. 
By \eqref{eq:blowupf2}, \eqref{eq:blowupf4}, \eqref{eq:blowupf7} and by arguing as in \eqref{eq:blowupbfinal} we have that
\begin{align*}
	\frac{d \mu_0}{\, d\Hs^1 \mres ((\partial {S_{h,K}} \cap \partial A) \cap S^{(1)}_{h,K})}(0)
	& \ge 2 \varphi (0, \mathbf{e}_\mathbf{1}) - \frac{ \ep}{2} - \frac \ep {c_1}  \frac{d \mu_0}{\, d\Hs^1 \mres ((\partial {S_{h,K}} \cap \partial A) \cap S^{(1)}_{h,K})}(0).		\end{align*}
Finally, by (f4) and taking $\ep \to 0^+$, in the inequality above, we reach  \eqref{eq:blowupe}.
		\end{proof}
  
		\begin{proof}[Proof of \eqref{eq:blowupea1}]
Since $\varphi' \le \varphi$, we repeat the same arguments of the proof of \eqref{eq:blowupe} in view of the fact that 
	$ \partial S_{h,K} \cap \partial A \cap S^{(0)}_{h,K} \cap A^{(0)} = \partial S_{h,K} \cap \partial A \cap  A^{(0)} $ by 
	employing Proposition \ref{prop:aux1}-(b) in place of Proposition \ref{prop:aux1}-(a) and Lemma \ref{lema:filamentssubstrate} in place of Lemma \ref{lema:crackssubstrate}.
		\end{proof}
		
	\begin{proof}[Proof of \eqref{eq:blowupf}]
In order to obtain \eqref{eq:blowupf} we  combine the arguments of the proof of \eqref{eq:blowupb} and the proof of \eqref{eq:blowupe}, by using the argumentations of the former with Proposition \ref{prop:aux1}-(c) with $\nu_A=\textbf{e}_\textbf{1}$ for the sets $A$ and $A_k$ and their convergence, and the argumentations of the latter with Proposition \ref{prop:aux1}-(b) for  the sets $S_{h,K}$ and $S_k$ and their convergence, but employing 
Lemma \ref{lem:delaminationofS} in place of Lemmas \ref{lema:creationOneBdr} and \ref{lema:crackssubstrate}, which were used in such previous proofs.
\end{proof}

		\begin{proof}[Proof of \eqref{eq:blowupg}] 
  We repeat the same arguments of the proof of \eqref{eq:blowupc1}
  which are based on \cite[Lemma 4.5]{KP}, 
by ``interchanging the roles'' of $A_k,A$ with $S_k, S_{h,K}$.
		\end{proof}

		\begin{proof}[Proof of \eqref{eq:blowupga}]  We repeat the same arguments of the proof of \eqref{eq:blowupc_2}
  which are based on \cite[Lemma 4.4]{KP}, 
by ``interchanging the roles'' of $A_k,A$ with $S_k, S_{h,K}$.
		\end{proof}

		\begin{proof}[Proof of \eqref{eq:blowuph}]
  By the definition of $\partial ^* S_{h,K}$, by the $\Hs^1$-rectifiability of $\partial A$, by \cite[Proposition A.4]{KP} (applied with $K$ taken as first $\partial A$ and then $\partial S_{h,K}$), and the Besicovitch Derivation Theorem the set of points $x \in \partial ^*{S_{h,K}} \cap \partial A   \cap A^{(1)}$ not satisfying the following 4 conditions:
	\begin{enumerate}
	\item[(h1)] $\theta^* (\partial A,x) = \theta_* (\partial A , x) = 1 $,
	\item[(h2)] $\nu_A \xp$, $\nu_{S_{h,K}} \xp$ exist and, either $\nu_A \xp = \nu_{S_{h,K}} \xp$ or $\nu_A \xp = - \nu_{S_{h,K}} \xp$,
	\item[(h3)] for every open rectangle $R$ containing $x$ with sides parallel or perpendicular to $\textbf{e}_\textbf{1}$ we have that  $\overline{R} \cap  \partial {\sigma}_{\rho, x} ( A) \cngK \overline{R} \cap T_{x,\nu_A\xp}$ and $ \overline{R} \cap  \partial {\sigma}_{\rho, x} ( {S_{h,K}}) \cngK  \overline{R} \cap T_{x,\nu_A\xp}$ as $\rho \to 0$, where $T_{x,\nu_A\xp}$ is the approximate tangent line at $x$ of $\partial A$ (or of $\partial {S_{h,K}}$),
	\item[(h4)] $ \frac{d \mu_0}{d \Hs^1 \mres ( \partial ^*{S_{h,K}} \cap \partial A   \cap A^{(1)})} \xp$ exists and it is finite,
			\end{enumerate}
   is $\Hs^1$-negligible. Therefore, we prove \eqref{eq:blowuph} for any fixed $x \in  \partial ^*{S_{h,K}} \cap \partial A   \cap A^{(1)}$ satisfying (h1)-(h4). Without loss of generality we assume that $x=0$ and we denote $T_0 = T_{0,\nu_A(0)}$. Furthermore, by using (h2) we choose in (h3) the rectangle $R_{\nu_A}:= Q_1$ if $\nu_A(0)  = \textbf{e}_i$ for $i =\textbf{1},\textbf{2}$,  or $R_{\nu_A} :=(-\cos \theta_{\nu_A }, \cos \theta_{\nu_A } ) \times_{\R^2} (-\sin \theta_{\nu_A }, \sin \theta_{\nu_A } )$, where $\theta_{\nu_A }$ is the angle formed between the tangent line $T_0$ and the $x_1$-axis, otherwise. For any $\rho >0$, we write $R_\rho := \rho R_{\nu_A}$.

   In view of the definition of $R_{\nu_A}$ and again by using also the Besicovitch Derivation Theorem (see \cite[Theorem 1.153]{FL}) there exists a subsequence $\rho_n \searrow 0$ such that
\begin{equation}
	\label{eq:blowupg1}
	\mu_0 (\partial R_{\rho_n} )=0 , \quad  \lim_{k \rightarrow +\infty}{\mu_k (\overline{R_{\rho_n}})} = \mu_0 (R_{\rho_n})
\end{equation}
and
\begin{equation}
	\label{eq:blowupg2}
	\frac{d \mu_0}{d \Hs^1 \mres (\partial ^*{S_{h,K}} \cap \partial A   \cap A^{(1)})}(0) = \lim_{n \rightarrow \infty}{\frac{\mu_0 (R_{\rho_n})}{2 \rho_n}}.
\end{equation}
We now claim that 
\begin{equation}
    \label{eq:diagg1}
\sdist (\cdot, \partial \sigma_{\rho_n} ( A )) \rightarrow -\dist (\cdot, T_0)\quad\text{and}\quad\sdist (\cdot,\partial \sigma_{\rho_n} ( {S_{h,K}})) \rightarrow \sdist (\cdot, \partial H_0)
\end{equation}
uniformly in $\overline{R_{\nu_A}}$ as $n \to \infty$. To prove the claim we can for example observe that by \cite[Proposition A.4]{KP} we have (not only (h3), but also) that $\overline{Q_r} \cap  \partial {\sigma}_{\rho, x} ( A) \cngK \overline{Q_r} \cap T_x \quad \text{and} \quad  \overline{Q_r} \cap  \partial {\sigma}_{\rho, x} ( {S_{h,K}}) \cngK  \overline{Q_r} \cap T_x \quad \text{as }\rho \to 0$ for any square $Q_{r}$ such that $R_{\nu_A} \subset Q_r$
and hence, by Proposition \ref{prop:aux1} Items (a) and (c), applied to $Q_r$, $\sdist (\cdot, \partial \sigma_{\rho_n} ( A )) \rightarrow -\dist (\cdot, T_0)\quad\text{and}\quad\sdist (\cdot,\partial \sigma_{\rho_n} ( {S_{h,K}})) \rightarrow \sdist (\cdot, \partial H_0)$ uniformly in $\overline{Q_r} \supset R_{\nu_A}$.

Furthermore, from the $\tau_\B$-convergence it follows that
\begin{equation}
\label{eq:diagg2}
\sdist (\cdot, \partial A_k ) \rightarrow \sdist(\cdot, \partial A )\quad\text{and}\quad\sdist (\cdot, \partial {S_k} ) \rightarrow \sdist(\cdot, \partial {S_{h,K}} )
\end{equation}
uniformly in $\overline{R_{\nu_A}}$ as $k \rightarrow \infty$ and hence,
by \eqref{eq:diagg1} and \eqref{eq:diagg2}, a standard diagonalization argument yields that there exists a subsequence $\left\{ (A_{k_n}, h_{k_n},K_{k_n} ) \right\}$ such that
\begin{equation}
    \label{eq:diagg3}
\sdist (\cdot, \partial \sigma_{\rho_n} ( A_\kn )) \rightarrow -\dist (\cdot, T_0)\quad\text{and}\quad\sdist (\cdot,\partial \sigma_{\rho_n} ( {S_{\kn}})) \rightarrow \sdist (\cdot, \partial H_0)
\end{equation}
uniformly in $\overline{R_{\nu_A}}$ as $n \to \infty$ and by \eqref{eq:blowupg1} such that 
	\begin{equation}
		\label{eq:blowupg4}
		\mu_\kn (\overline{R_{\rho_n}}) \le \mu_0 (R_{\rho_n}) + \rho_n^2,
	\end{equation}
	for any $n \in \N$. By \eqref{eq:blowupg1}, \eqref{eq:blowupg4} and arguing as in \eqref{eq:blowupb5} we infer that
			\begin{equation}
				\label{eq:blowupg5}
				\begin{split}
					\limsup_{n \to \infty}{\frac{\Hs^1 (\overline{R_{\rho_n}} \cap \partial A_{k_n}) + \Hs^1 (\overline{R_{\rho_n}} \cap \partial S_{k_n} \setminus \partial A_{k_n}) }{2 \rho_n}  }  \le  c_1^{-1}  \frac{d \mu_0}{d \Hs^1 \mres ( \partial ^*{S_{h,K}} \cap \partial A   \cap A^{(1)})}(0) .
				\end{split}
			\end{equation}
By applying Lemma \ref{lem:A2}, we have that $\overline{R_{\nu_A}} \setminus \sigma_{\rho_n} ( A_{k_n}) \cngK T_0$ and in view of Remark \ref{remark:1}-(i) we have that $\overline{R_{\nu_A}} \cap  \partial \sigma_{\rho_n} ( S_{k_n} ) \cngK T_0$. Since $\left\{ (A_\kn, S_{h_\kn,K_\kn}) \right\} \subset \B_\mathbf{m}$, we know that $(\sigma_{\rho_n}({A_{k_n}}),S_{ (1/{\rho_n})h_{k_n} (\rho_n\cdot), \sigma_{\rho_n}({K_{k_n}})}) \in \B_{\textbf{m}}(\sigma_{\rho_n}(\Omega))$. By applying Lemma \ref{lema:creationdelamin} with $\phi_{\alpha} (\cdot) = \varphi_\alpha (0, \cdot) $ for $\alpha = \text{F,S,FS}$ and 
$\delta = \ep$, there exists $n^1_\ep\in \N$ such that for every $n \ge n^1_\ep$,
	\begin{equation}
		\label{eq:blowupg6}
		\begin{split}
		\mathcal{S}_L (\sigma_{\rho_n}({A_{k_n}}),S_{ \sigma_{\rho_n}({S_{k_n}}), \sigma_{\rho_n}({K_{k_n}})},R_{\nu_A} )&  \ge \int_{\overline{R_{\nu_A}}\cap T_0 }{ 		\varphi_{\mathrm{ F}} (0,\nu_{T_0}))  + \varphi (0,\nu_{T_0})  
  \, d\Hs^1} - \ep \\
  &= 
  2 ( \varphi_{\mathrm{ F}}(0,\nu_{T_0}) + \varphi (0,\nu_{T_0})  ) 
  - \ep.
		\end{split}
	\end{equation}
	By definition of $\mu_k$, the non-negativeness of $\varphi_{\mathrm{ F}}, \varphi$ and $\varphi_{\mathrm{ FS}}$, \eqref{eq:continuitynorms}, \eqref{eq:blowupg6} and arguing as in \eqref{eq:blowupb71} we deduce that  
	\begin{equation}
	    \label{eq:blowupg7}
	\mu_\kn (\overline{Q_{\rho_n}}) \ge   2 ( \varphi_{\mathrm{ F}}(0,\nu_{T_0}) + \varphi (0,\nu_{T_0})  )  
	 - \ep \rho_n - \ep \left( \Hs^1 (\overline{Q_{\rho_n}} \cap \partial A_{k_n}) + \Hs^1 (\overline{Q_{\rho_n}} \cap \partial S_{k_n} \setminus \partial A_{k_n})  \right),
	\end{equation}		
for every $n > n'_\ep$. By \eqref{eq:blowupg2}, \eqref{eq:blowupg4}, \eqref{eq:blowupg7} and by arguing as in 
\eqref{eq:blowupbfinal} we have that
	\begin{equation*}
 \begin{split}
		\frac{d \mu_0}{\, d\Hs^1 \mres ( \partial A \cap \partial ^* S \cap A^{(1)} )}(0) &\ge 
\varphi_{\mathrm{ F}}(0,\nu_{T_0}) + \varphi (0,\nu_{T_0})  
 - \frac{ \ep}{2} \\
 & \qquad \qquad- \ep c_1^{-1}  \frac{d \mu_0}{\, d\Hs^1 \mres ((\partial {S_{h,K}} \cap \partial A) \cap S^{(1)}_{h,K})}(0).
 \end{split}
	\end{equation*}
	Finally, by (h4) and by taking $\ep \to 0^+$ in the inequality above, we deduce \eqref{eq:blowupg}.
		\end{proof}
	\begin{proof}[Proof of \eqref{eq:blowupi}]
			We repeat the arguments of \eqref{eq:blowupe} by using Proposition \ref{prop:aux1}-(a) 
   for the sets $A$ and $A_k$,   and Proposition \ref{prop:aux1}-(b) for   the sets $S_{h,K}$ and $S_k$, and by replacing Lemma \ref{lema:creationOneBdr} with Lemma \ref{lema:creationcrackdelamination}.
		\end{proof}
	
	\end{proof}

We are finally in position to prove Theorem \ref{thm:lowersemicontinuity}.

\begin{proof}[Proof of Theorem \ref{thm:lowersemicontinuity}]
Without loss of generality, we assume that the \emph{liminf} in the right side of \eqref{eq:lowersemicontinuityF} is reached and finite in $\R$. From Theorem \ref{thm:lowersemicontinuityS} it follows that
		\begin{equation}
			\label{lweqF1}
			\mathcal{S} (A,S_{h,K}) \le \liminf_{k \rightarrow \infty}{ \mathcal{S} (A_k,S_{ h_k,K_k})}.
		\end{equation}
		In view of the definition  of $\mathcal{F}$, in order to reach the assertion, it suffices  to establish the lower semicontinuity of $\mathcal{W}$, to which the rest of the proof is devoted. 
  
  Let $D \subset \subset \textrm{Int}(A)$, by the fact that $\Int{(A_k)} \xrightarrow{\mathcal{K}} \Int{(A)}$, we deduce that $D \subset \subset \Int (A_k)$  for $k$ large enough. As $u_k \rightarrow u$ a.e. in $\Int (A)$, then, $u_k \rightarrow u$ a.e. in $D$. Furthermore, since $e(u_k)$ are bounded in the $L^2(D)$ norm, we have that $e(u_k) \rightharpoonup e(u)$ in $L^2 (D)$. By convexity of $\mathcal{W}(D, \cdot)$ we obtain that
		\begin{align*}
			\mathcal{W}(D, u) \le \liminf_{k \rightarrow +\infty}{\mathcal{W}(D, u_k)} \le \liminf_{k \rightarrow +\infty}{\mathcal{W}(A_k, u_k)}.
		\end{align*}
		To conclude it is now enough to let $D \nearrow \Int (A)$.
	\end{proof}

	
	\section{Existence}
\label{sec:existence}
 
In view of Theorems \ref{thm:compactness} and \ref{thm:lowersemicontinuity} we are in  position to prove  Theorem \ref{thm:existence} by  employing the \emph{direct method of the calculus of variations}.
	
	\begin{proof}[Proof of Theorem \ref{thm:existence}]
		 Fix $m \in \N$ and let $\{(A_k, S_{h_k,K_k},u_k)\} \subset \mathcal{C}_\mathbf{m}$ be a minimizing sequence of $\mathcal{F}$ such that $\mathcal{L}^2({A_k})= \mathbbm{v}_1$, $\mathcal{L}^2({S_{h_k,K_k}})= \mathbbm{v}_0$, and
		$$\sup_{k\in\mathbb{N}} {\mathcal{F}(A_k,S_{h_k,K_k},u_k)}< \infty.$$
		By Theorem \ref{thm:compactness} there exist a subsequence $\{(A_{k_l},S_{h_{k_l},K_{k_l}}, u_{k_l})\}$, a sequence $\{(\widetilde{A}_l ,S_{h_{k_l}, \widetilde{K}_{k_l}}, v_l)\} \subset \mathcal{C}_\mathbf{m}$, and $(A,S_{h,K},u) \in \mathcal{C}_\mathbf{m}$ such that $(\widetilde{A}_l ,S_{h_{k_l}, \widetilde{K}_{k_l}}, v_l) \xrightarrow{\tau_{\mathcal{C}}} (A,S_{h,K},u)$ as $l \to \infty$ and
		\begin{equation}
			\label{eq:existence1}
			\liminf_{l\to \infty}{\mathcal{F}(\widetilde{A}_l ,S_{h_{k_l}, \widetilde{K}_{k_l}}, v_l)} = \liminf_{l \to \infty}{\mathcal{F} (A_{k_l},S_{h_{k_l},K_{k_l}}, u_{k_l})}.
		\end{equation} 
		By Theorem \ref{thm:lowersemicontinuity}, we have that
		\begin{equation}
			\label{eq:existance2}
			\mathcal{F}(A,S_{h,K},u) \le \liminf_{l \to \infty}{\mathcal{F}(\widetilde{A}_l ,S_{h_{k_l}, \widetilde{K}_{k_l}}, v_l)}.
		\end{equation}
		We claim that $\{(\widetilde{A}_l ,S_{h_{k_l}, \widetilde{K}_{k_l}}, v_l)\}$ and $(A,S_{h,K},u)$ satisfy the volume constraints of \eqref{eq:const}. Indeed, by Theorem \ref{thm:compactness}, for any $l  \ge 1$, $\mathbbm{v}_1 = \mathcal{L}^2\left({A_{k_l}} \right) = \mathcal{L}^2({\widetilde{A}_l})$ and $\mathbbm{v}_0 = \mathcal{L}^2({S_{h_{k_l}, {K}_{k_l}}}) = \mathcal{L}^2\left({S_{h_l, \widetilde K_l}}\right)$, where ${S_{h_l, \widetilde K_l}} \in \AS$. Thanks to the fact that $(\widetilde{A}_l, S_{h_{k_l}, K_{k_l}}) \xrightarrow{\tau_{\mathcal{\B}}} (A, S_{h,K} )$ as $l \to \infty$, by applying \cite[Lemma 3.2]{KP} we infer that $\widetilde{A}_l \to A$ in $L^1(\R^2)$ as $l \to \infty$, and thus $\mathcal{L}^2({A}) = \mathbbm{v}_1$, and, similarly, we deduce that $\mathcal{L}^2({S_{h,K}})= \mathbbm{v}_0$. Finally, from \eqref{eq:existence1} and \eqref{eq:existance2} we deduce that
	\begin{equation*}
			\begin{split}
				\inf_{(A,S_{h,K},u) \in \mathcal{C}_\mathbf{m},\, \mathcal{L}^2({A})= \mathbbm{v}_1,\, \mathcal{L}^2({S_{h,K}})= \mathbbm{v}_0} \mathcal{F}(A,S_{h,K},u) &= \lim_{k \to \infty}{\mathcal{F} (\widetilde{A}_l ,S_{h_{k_l}, \widetilde{K}_{k_l}}, v_l)}\\
				& \ge \liminf_{l\to \infty}{\mathcal{F}\left(\widetilde{A}_l ,S_{h_{k_l}, \widetilde{K}_{k_l}}, v_l\right)}  \\
    &\ge \mathcal{F}(A,S_{h,K},u) 
			\end{split}
		\end{equation*}
		and hence, $(A,S_{h,K},u)$ is a solution of the minimum problem \eqref{eq:const}. By observing that \eqref{eq:existence1} and \eqref{eq:existance2} hold true also by replacing $\mathcal{F}$ with $\mathcal{F}^{\lambda}$ for ${\lambda} := (\lbd_0, \lbd_1),$ with $\lbd_0,\lbd_1>0$, we deduce that also the unconstrained minimum problem \eqref{eq:uncost} can be solved by employing the same method. This concludes the proof. 
	\end{proof}

 \section*{Acknowledgement} 
 The authors thank Shokhrukh Kholmatov for the various discussions on the topic. 
 The authors acknowledge the support received from the Austrian Science Fund (FWF) projects P 29681 and TAI 293, from the Vienna Science and Technology Fund (WWTF) together with the City of Vienna and Berndorf Privatstiftung through Project MA16-005, and from BMBWF through the OeAD-WTZ project HR 08/2020.  R. Llerena thanks for the support obtained by a French public grant as part of the ``Investissement d'avenir'' project ANR-11-LABX-0056-LMH, LabEx LMH. P. Piovano is member of the Italian ``Gruppo Nazionale per l'Analisi Matematica, la Probabilit\`a e le loro Applicazioni'' (GNAMPA) and has received funding from the GNAMPA-INdAM 2022 project CUP: E55F22000270001 and 2023 Project CUP: E53C22001930001. P. Piovano also acknowledges the support obtained by the Italian Ministry of University and Research  (MUR) through  the PRIN Project ``Partial differential equations and related geometric-functional inequalities''.  Finally,  P. Piovano is also grateful for the support received as \emph{Visiting Professor and Excellence Chair} and from the \emph{Visitor in the Theoretical Visiting Sciences Program} (TSVP) at the Okinawa Institute of Science and Technology (OIST), Japan.


	\nocite{*}
	\bibliographystyle{amsplain}
	\bibliography{reference4.bib}

\providecommand{\bysame}{\leavevmode\hbox to3em{\hrulefill}\thinspace}
\providecommand{\MR}{\relax\ifhmode\unskip\space\fi MR }
\providecommand{\MRhref}[2]{%
  \href{http://www.ams.org/mathscinet-getitem?mr=#1}{#2}
}
\providecommand{\href}[2]{#2}
\begin{thebibliography}{10}

\bibitem{Al}
F.~J. Almgren, \emph{Existence and regularity almost everywhere of solutions to
  elliptic variational problems with constraints}, Mem. Amer. Math. Soc.
  \textbf{4-165} (1976), viii+199.

\bibitem{AB}
L.~Ambrosio and A.~Braides, \emph{Functionals defined on partitions in sets of
  finite perimeter {I}: integral}, J. Math. Pures Appl. \textbf{69} (1990),
  285--305.

\bibitem{AB2}
\bysame, \emph{Functionals defined on partitions in sets of finite perimeter
  {II}: semicontinuity}, J. Math. Pures Appl. \textbf{69} (1990), 307--333.

\bibitem{AFP}
L.~Ambrosio, N.~Fusco, and D.~Pallara, \emph{{Functions of Bounded Variation
  and Free Discontinuity Problems}}, Oxford University Press, New York, 2000.

\bibitem{AT}
R.~Asaro and W.~Tiller, \emph{Interface morphology development during stress
  corrosion cracking: Part {I}. {V}ia surface diffusion}, Metall. Trans.
  \textbf{3} (1972), 1789--1796.

\bibitem{bartle}
R.~Bartle, \emph{The {E}lements of {I}ntegration and {L}ebesgue {M}easure},
  Wiley, New York, 1995.

\bibitem{BonCris}
M.~Bonacini and R.~Cristoferi, \emph{Area {Q}uasi-minimizing {P}artitions with
  a {G}raphical {C}onstraint: {R}elaxation and {T}wo-{D}imensional {P}artial
  {R}egularity}, J Nonlinear Sci \textbf{32} (2022), no.~6, 93.

\bibitem{BFG}
D.~Bucur, I.~Fragalà, and A.~Giacomini, \emph{The {M}ultiphase
  {M}umford-{S}hah problem}, SIAM J. Imaging Sci. \textbf{12} (2019),
  1561--1583.

\bibitem{BFG2}
\bysame, \emph{Multiphase free discontinuity problems: Monotonicity formula and
  regularity results}, Ann. Inst. Henri Poincare (C) Anal. Non Lineaire
  \textbf{38} (2021), 1553--1582.

\bibitem{cara1}
D.~Caraballo, \emph{The triangle inequalities and lower semi-continuity of
  surface energy of partitions}, Proc. R. Soc. Edinb. A: Math. \textbf{139}
  (2009), no.~3, 449–457.

\bibitem{cara2}
\bysame, \emph{{BV}-{E}llipticity and {L}ower {S}emicontinuity of {S}urface
  {E}nergy of {C}accioppoli {P}artitions of $\mathbb{R}^n$}, J Geom Anal
  \textbf{23} (2013), 202--220.

\bibitem{cara3}
\bysame, \emph{Existence of surface energy minimizing partitions of $\mathbb
  {R}^{n}$ satisfying volume constraints}, Trans. Amer. Math. Soc. \textbf{369}
  (2016), 1517--1546.

\bibitem{cermelligurtin1}
P.~Cermelli and M.E. Gurtin, \emph{The dynamics of solid-solid phase
  transitions 2. {I}ncoherent interfaces}, Arch. Rational Mech. Anal.
  \textbf{127} (1994), no.~1, 41--99.

\bibitem{ChB}
A.~Chambolle and E.~Bonnetier, \emph{Computing the equilibrium configuration of
  epitaxially strained crystalline films}, SIAM J. Appl. Math. \textbf{62}
  (2002), 1093--1121.

\bibitem{CC}
A.~Chambolle and V.~Crismale, \emph{Existence of strong solutions to the
  {D}irichlet problem for the {G}riffith energy}, Calc. Var. Partial
  Differential Equations \textbf{58} (2019), 136.

\bibitem{CTV}
M.~Conti, S.~Terracini, and G.~Verzini, \emph{A variational problem for the
  spatial segregation of reaction-diffusion systems}, Indiana Univ. Math. J.
  \textbf{54-3} (2005), 779--815.

\bibitem{CF}
V.~Crismale and M.~Friedrich, \emph{Equilibrium {C}onfigurations for
  {E}pitaxially {S}trained {F}ilms and {M}aterial {V}oids in
  {T}hree-{D}imensional {L}inear {E}lasticity}, Arch. Ration. Mech. Anal.
  \textbf{237} (2020), no.~2, 1041--1098.

\bibitem{Dm}
G.~Dal~Maso, \emph{An introduction to {$\Gamma$}-convergence}, Birkhäuser,
  Boston, 1993.

\bibitem{DMS}
G.~{Dal Maso}, J.~{Morel}, and S.~{Solimini}, \emph{A variational method in
  image segmentation: existence and approximation results}, Acta Math.
  \textbf{168} (1992), 89--151.

\bibitem{D}
A~Danescu, \emph{The {A}saro--{T}iller--{G}rinfeld instability revisited}, Int.
  J. Solids Struct. \textbf{38} (2001), 4671--4684.

\bibitem{DP}
E.~{Davoli} and P.~{Piovano}, \emph{Analytical validation of the
  {Y}oung–{D}upr\'e law for epitaxially-strained thin films}, Math. Models
  Methods Appl. Sci \textbf{29} (2019), 2183--2223.

\bibitem{DP2}
\bysame, \emph{Derivation of a heteroepitaxial thin-film model}, Interfaces
  Free Bound. \textbf{22} (2020), 1--26.

\bibitem{DphM}
G.~De~Philippis and F.~Maggi, \emph{Regularity of {F}ree {B}oundaries in
  {A}nisotropic {C}apillarity {P}roblems and the {V}alidity of {Y}oung's
  {L}aw}, Arch. Ration. Mech. Anal. \textbf{216} (2015), 473--568.

\bibitem{EG}
L.C. Evans and R.F. Gariepy, \emph{Measure {T}heory and {F}ine {P}roperties of
  {F}unctions}, Taylor \& Francis, New York, 1991.

\bibitem{F}
K.~Falconer, \emph{The geometry of fractal sets}, no.~85, Cambridge
  {U}niversity {P}ress, Cambridge, 1986.

\bibitem{FFLM}
I.~Fonseca, N.~Fusco, G.~Leoni, and V.~Millot, \emph{Material voids in elastic
  solids with anisotropic surface energies}, J. Math. Pures Appl. \textbf{96}
  (2011), 591--639.

\bibitem{FFLM2}
I.~Fonseca, N.~Fusco, G.~Leoni, and M.~Morini, \emph{Equilibrium configurations
  of epitaxially strained crystalline films: existence and regularity results},
  Arch. Ration. Mech. Anal. \textbf{186} (2007), 477--537.

\bibitem{FL}
I.~Fonseca and G.~Leonni, \emph{Modern methods in the {C}alculus of
  {V}ariations: {$L^p$}-spaces}, Springer, New York, 2007.

\bibitem{FM}
I.~Fonseca and S.~M\"uller, \emph{Quasi-{C}onvex {I}ntegrands and {L}ower
  {S}emicontinuity in {$L^1$}}, SIAM J. Appl. Math. \textbf{23} (1992),
  1081--1098.

\bibitem{FrMa}
G.A. Francfort and J.-J. Marigo, \emph{Revisiting brittle fracture as an energy
  minimization problem}, Mech. Phys. Solids \textbf{46} (1998), 1319--1342.

\bibitem{FG}
E.~Fried and M.~Gurtin, \emph{A unified treatment of evolving interfaces
  accounting for small deformations and atomic transport with emphasis on
  grain-boundaries and epitaxy}, Adv. Appl. Mech. \textbf{40} (2004), 1--177.

\bibitem{friesolo}
M.~Friedrich, M.~Perugini, and F.~Solombrino, \emph{Lower semicontinuity for
  functionals defined on piecewise rigid functions and on {GSBD}}, J. Funct.
  Anal. \textbf{280} (2021), 108929.

\bibitem{G}
A.~Giacomini, \emph{A generalization of {G}o{\l}\k ab's theorem and
  applications to fracture mechanics}, Math. Models Methods Appl. Sci.
  \textbf{12} (2002), 1245--1267.

\bibitem{Gr}
M.A Grinfeld, \emph{The stress driven instability in elastic crystals:
  Mathematical models and physical manifestations}, J. Nonlinear Sci.
  \textbf{3} (1993), 35--83.

\bibitem{Gurtin1}
M.E. Gurtin, \emph{The dynamics of solid-solid phase transitions 1. {C}oherent
  interfaces}, Arch. Rational Mech. Anal. \textbf{123} (1993), no.~4, 305--335.

\bibitem{KP}
Sh. {Kholmatov} and P.~{Piovano}, \emph{{A Unified Model for Stress-Driven
  Rearrangement Instabilities}}, Arch. Rational Mech. Anal. (2020), 415–488.

\bibitem{KP1}
\bysame, \emph{Existence of minimizers for the {SDRI} model in 2d: wetting and
  dewetting regime with mismatch strain}, Adv. Calc. (2023).

\bibitem{KP2}
\bysame, \emph{Existence of minimizers for the {SDRI} model in $\mathbb{R}^n$:
  Wetting and dewetting regimes with mismatch strain}, arXiv: Analysis of PDEs
  (2023).

\bibitem{KreutzP}
L.~Kreutz and P.~Piovano, \emph{{Microscopic Validation of a Variational Model
  of Epitaxially Strained Crystalline Films}}, SIAM J. Appl. Math. \textbf{53}
  (2021), 453--490.

\bibitem{Baldelli:2014}
A.A. {Le\'on Baldelli}, J.-F. {Babadjian}, B.~{Bourdin}, D.~{Henao}, and
  C.~{Maurini}, \emph{A variational model for fracture and debonding of thin
  films under in-plane loadings}, J. Mech. Phys. Solids \textbf{70} (2014),
  320--348.

\bibitem{L}
G.~Leoni, \emph{A first course in {S}obolev spaces}, Graduate studies in
  mathematics; 105, American Math. Soc., Providence, 2009.

\bibitem{LlP1}
R.~Llerena and P.~Piovano, \emph{Solutions for a free-boundary problem modeling
  film multilayers with coherent and incoherent interfaces}, Preprint (2022).

\bibitem{M}
F.~Maggi, \emph{{Sets of Finite Perimeter and Geometric Variational Problems:
  An Introduction to Geometric Measure Theory}}, Cambridge University Press,
  2012.

\bibitem{Mp}
P.~Mattila, \emph{{Geometry of sets and measures in Euclidean spaces: fractals
  and rectifiability}}, Cambridge University Press, Cambridge, 1999.

\bibitem{Morgan}
F.~Morgan, \emph{Lowersemicontinuity of energy clusters}, Proc. R. Soc. Edinb.
  A: Math. \textbf{127} (1997), 819–822.

\bibitem{MS}
D.~Mumford and J.~Shah, \emph{Optimal approximations by piecewise smooth
  functions and associated variational problems}, Comm. Pure Appl. Math.
  \textbf{42-5} (1989), 577--685.

\bibitem{R}
H.L. Royden, \emph{Real analysis}, Macmillan, New York, 1966.

\bibitem{Rudin}
W.~Rudin, \emph{Analysis}, Oldenbourg, München, 2005.

\bibitem{S}
D.J Srolovitz, \emph{On the stability of surfaces of stressed solids}, Acta
  Metal. \textbf{37} (1989), 621--625.

\end{thebibliography}
\end{document}